\documentclass{article}
\usepackage[left=1in,top=1in,right=1in,bottom=1in,head=.1in]{geometry}
\usepackage{authblk}
\usepackage{amsmath}
\usepackage{amssymb}
\usepackage{amsthm}
\usepackage{amsfonts}
\usepackage{mathrsfs} 
\usepackage{bm}
\usepackage{bbm}
\usepackage{xparse}
\usepackage{array}
\usepackage{graphicx}
\usepackage{subcaption}
\usepackage{caption}
\usepackage{float}
\usepackage[dvipsnames]{xcolor}

\usepackage{hyperref} 

\hypersetup{
colorlinks=true,
citecolor=ForestGreen,
linkcolor=RoyalBlue,   
urlcolor=RubineRed!90!black
}
\usepackage{cleveref}

\usepackage{natbib}

\usepackage[ruled, vlined]{algorithm2e}
\usepackage{multirow}
\usepackage{aliascnt}

\numberwithin{equation}{section}
\newtheorem{definition}{Definition}[section]

\newtheorem{theorem}{Theorem}[section]
\newaliascnt{lemma}{theorem}
\newtheorem{lemma}[lemma]{Lemma}
\aliascntresetthe{lemma}

\newaliascnt{corollary}{theorem}
\newtheorem{corollary}[corollary]{Corollary}
\aliascntresetthe{corollary}

\newaliascnt{proposition}{theorem}
\newtheorem{proposition}[proposition]{Proposition}
\aliascntresetthe{proposition}

\newaliascnt{remark}{theorem}
\newtheorem{remark}[remark]{Remark}
\aliascntresetthe{remark}

\newaliascnt{example}{theorem}

\aliascntresetthe{example}

\newtheorem{assumption}{Assumption}

\Crefname{theorem}{Theorem}{Theorems}
\Crefname{lemma}{Lemma}{Lemmas}
\Crefname{corollary}{Corollary}{Corollaries}
\Crefname{proposition}{Proposition}{Propositions}
\Crefname{remark}{Remark}{Remarks}
\Crefname{example}{Example}{Examples}
\Crefname{definition}{Definition}{Definitions}
\Crefname{assumption}{Assumption}{Assumptions}

\DeclareMathOperator{\E}{\mathbb{E}}
\DeclareMathOperator{\Bernoulli}{Bernoulli}
\DeclareMathOperator{\asconv}{\; \xrightarrow{\text{a.s.}} \;}

\newcommand{\stable}{\mathcal{F}_\infty \operatorname{-stable}}

\DeclareMathOperator*{\argmax}{arg\,max}
\DeclareMathOperator{\tv}{\mathsf{TV}}
\newcommand*\diff{\mathop{}\!\mathrm{d}}
\newcommand{\indep}{\perp \!\!\! \perp}
\newcommand{\diversity}{\mathsf{S}_{\alpha, \mu}}
\newcommand{\res}{\mathsf{Res}}
\newcommand{\dd}{\mathrm{d}}
\newcommand{\PP}{\mathbb{P}}
\newcommand{\simp}{{\mathsf{p}}}
\newcommand{\editline}[1]{#1}
\newcommand{\editlinenew}[1]{#1}

\begin{document}

\title{
Asymptotic Inference for Exchangeable Gibbs Partitions
}

\author{Takuya Koriyama
\\
{The University of Chicago}\\
\href{mailto:tkoriyam@uchicago.edu}{\texttt{tkoriyam@uchicago.edu}
}
}
\maketitle

\begin{abstract}
We study the asymptotic properties of parameter estimation and predictive inference under the exchangeable Gibbs partition, 
characterized by a discount parameter $\alpha\in(0,1)$ and a triangular array $v_{n,k}$ satisfying a backward recursion. Assuming that $v_{n,k}$ admits a mixture representation over the Ewens--Pitman family $(\alpha, \theta)$, with $\theta$ integrated by an unknown mixing distribution, we show that the (quasi) maximum likelihood estimator $\hat\alpha_n$ (QMLE) for  $\alpha$ is asymptotically mixed normal. This generalizes earlier results for the Ewens--Pitman model to a more general class. We further study the predictive task of estimating the probability simplex $\mathsf{p}_n$, which governs the allocation of the $(n+1)$-th item, conditional on the current partition of $[n]$. Based on the asymptotics of the QMLE $\hat{\alpha}_n$, we construct an estimator $\hat{\mathsf{p}}_n$ and derive the limit distributions of the $f$-divergence $\mathsf{D}_f(\hat{\mathsf{p}}_n||\mathsf{p}_n)$ for general convex functions $f$, including explicit results for the TV distance and KL divergence. These results lead to asymptotically valid confidence intervals for both parameter estimation and prediction.
\end{abstract}

\tableofcontents

\section{Introduction}
For any positive integer \( n \in \mathbb{N} \) and any integer \( k \) with \( 1 \le k \le n \), a \emph{partition} of \([n] \equiv \{1, 2, \dots, n\}\) into \( k \) blocks is an unordered collection of nonempty, pairwise disjoint subsets whose union is \([n]\). Let \(\mathcal{P}_n^k\) denote the set of all such partitions of \([n]\) into \( k \) blocks, and define \(\mathcal{P}_n \equiv \bigcup_{k=1}^n \mathcal{P}_n^k\) to be the set of all partitions of \([n]\). For instance, when $n=3$, we have $\mathcal{P}_3 = \mathcal{P}_3^{1} \cup \mathcal{P}_3^{2} \cup \mathcal{P}_3^{3}$ with 
\begin{align*}
  \mathcal{P}_3^{1} &= \bigl\{\{1, 2, 3\}\bigr\}, \\\mathcal{P}_3^{2} &= \Bigl\{\bigl\{\{1, 2\},  \{3\}\bigr\}, \  \bigl\{\{1\}, \{2, 3\}\bigr\}, \ \bigl\{ \{1,3\}, \{2\}\bigr\}\Bigr\},\\
  \mathcal{P}_{3}^3 &= \bigl\{\{1\},\{2\},\{3\}\bigr\}. 
\end{align*}
An \emph{exchangeable Gibbs partition} is a stochastic process taking values in \(\{\mathcal{P}_n\}_{n=1}^\infty\) defined as follows:

\begin{algorithm}[H]
  \caption{Sequential definition of the Exchangeable Gibbs partition}\label{alg:gibbs}
  \KwIn{$\alpha \in (-\infty, 1)$ and a non-negative triangular array $(v_{n,k})$ satisfying the backward recursion}
  \begin{equation}\label{eq:recursion}
     v_{1,1} = 1 \quad \text{and} \quad 
  v_{n,k} = (n - \alpha k)\, v_{n+1, k} + v_{n+1, k+1} \ \  \text{for all} \ \ 1\le k \le n, \ n\in \mathbb{N}
  \end{equation}
  Initialize with $\Pi_1 = \big\{\{1\}\big\}$\;
  \For{$n = 1, 2, \dots$}{
    Given $\Pi_n = \{U_1, \dots, U_{\mathsf{k}_n}\} \in \mathcal{P}_n$, where $\mathsf{k}_n$ is the current number of blocks, define the probability simplex $\bm{\simp}_n \in \Delta^{\mathsf{k}_n + 1}$ by
    \begin{equation}\label{eq:intro_simplex}
      \bm{\simp}_n = (\simp_{n,0}, \simp_{n,1}, \dots, \simp_{n,\mathsf{k}_n}), \quad 
      \simp_{n,i} \equiv
      \begin{cases}
        \frac{v_{n+1, \mathsf{k}_n + 1}}{v_{n,\mathsf{k}_n}} & i = 0, \\[5pt]
        \frac{v_{n+1, \mathsf{k}_n}}{v_{n,\mathsf{k}_n}} \big( |U_i| - \alpha \big) & i \in \{1, \dots, \mathsf{k}_n\}.
      \end{cases}
    \end{equation}
    Generate $\Pi_{n+1} \in \mathcal{P}_{n+1}$ by assigning the $(n+1)$-th element according to
    \begin{equation}\label{eq:rule}
      (n+1) \in 
      \begin{cases}
        \text{a new block} & \text{with probability } \simp_{n,0}, \\
        U_i & \text{with probability } \simp_{n,i}, \quad i \in \{1, \dots, \mathsf{k}_n\}.
      \end{cases}
    \end{equation}
  }
\end{algorithm}

Here, the parameter \(\alpha\) and the triangular array \((v_{n,k})\) satisfy \(\alpha \in (-\infty, 1)\) and the backward recursion \eqref{eq:recursion}
so that \(\bm{\simp}_n\) defined in \eqref{eq:intro_simplex} is indeed a valid probability simplex, i.e., \(\sum_{i} \simp_{n,i} = 1\) and $\simp_{n,i}\ge 0$.   
An example of an array \((v_{n,k})\) satisfying \eqref{eq:recursion} is given by the well-known two-parameter family:
\begin{align}\label{eq:v_nk_ewens}
  v_{n,k}(\alpha, \theta) 
  \;=\; \frac{\prod_{i=1}^{k-1} (\theta + i \alpha)}{\prod_{i=1}^{n-1} (\theta + i)}, 
  \quad \text{with} \quad
  \begin{cases}
    \theta > -\alpha, & \text{if } 0 \le \alpha < 1, \\
    \theta = -k \alpha \text{ for some } k \in \mathbb{N}, & \text{if } \alpha < 0.
  \end{cases}
\end{align}
The corresponding random partition is known as the \emph{Ewens--Pitman partition $(\alpha,\theta)$}.

Notably, the marginal distribution of 
\(\Pi_n = \{ U_1, \dots, U_{\mathsf{k}_n} \} \in \mathcal{P}_n\) under the Gibbs partition model is given by
\begin{align}\label{eq:intro_marginal_likelihood}
  \PP\bigl( \Pi_n = \{ U_1, \dots, U_{\mathsf{k}_n} \} \bigr)
  = v_{n, \mathsf{k}_n}
  \prod_{i=1}^{\mathsf{k}_n} \prod_{j=1}^{|U_i|-1} (j - \alpha)
  = v_{n, \mathsf{k}_n}
    \prod_{j=2}^{n} 
    \left( \prod_{i=1}^{j-1} (i - \alpha) \right)^{\mathsf{k}_{n,j}},
\end{align}
where 
\(\mathsf{k}_{n,j} = \sum_{i=1}^{\mathsf{k}_n} \mathbbm{1}\big\{ |U_i| = j \big\}\) 
denotes the number of blocks of size \(j\). 
Since \(\sum_{j=1}^{n} \mathsf{k}_{n,j} = \mathsf{k}_n\), equation \eqref{eq:intro_marginal_likelihood} shows that the block sizes \((\mathsf{k}_{n,j})_{j=1}^n\) are sufficient statistics for the Gibbs partition model.  
In particular, the marginal likelihood \eqref{eq:intro_marginal_likelihood} is invariant under any permutation of the labels in \([n] = \{1, \dots, n\}\).  
In general, random partitions with this property are called \emph{exchangeable random partitions}, and a large class of such exchangeable partitions can be represented by the Gibbs partition model.  
Indeed, \cite{gnedin2006exchangeable} show that for any exchangeable partition model whose marginal likelihoods $\PP(\Pi_n = \{ U_1, \dots, U_{\mathsf{k}_n} \})$ are consistent as \(n\) varies and take the product form 
\( V_{n, \mathsf{k}_n} \prod_{i=1}^{\mathsf{k}_n} W_{|U_i|} \)
for some non-negative triangular array \((V_{n,k})\) and weight \((W_\ell)\), then $W_\ell$ must take the form $\prod_{j=1}^{\ell-1} (j - \alpha)$ for some $\alpha\in(-\infty,1)$ and the array \((V_{n,k})\)  must satisfy the backward recursion \eqref{eq:recursion}.

\paragraph{Asymptotics of the Gibbs partition}
By the definition \eqref{eq:intro_simplex}--\eqref{eq:rule}, the probability that the next item belongs to a current block $U_i$ is proportional to $(|U_i|-\alpha)$, which implies that the Gibbs partition exhibits a \emph{rich-get-richer} dynamics: blocks with larger sizes are more likely to attract new elements, while the discount parameter \(\alpha\) counteracts this tendency. Therefore, intuitively, the smaller the value of \(\alpha\), the smaller the expected number of non-empty blocks \(\mathsf{k}_n\). This intuition is reflected in the asymptotic behavior of \(\mathsf{k}_n\): \cite{gnedin2006exchangeable} show that for any Gibbs partition model with a parameter $\alpha\in (-\infty, 1)$ and an array $(v_{n,k})$ satisfying \eqref{eq:recursion}, there exists a positive random variable \(\mathsf{S}_\alpha\), whose law depends on \(\alpha\) and \((v_{n,k})\), such that
\begin{align}\label{eq:intro_sparse}
  \frac{\mathsf{k}_n}{c_n(\alpha)} \; \xrightarrow{\text{a.s.}} \; \mathsf{S}_\alpha
  \quad \text{where} \quad 
  c_n(\alpha) \equiv 
  \begin{cases}
    n^\alpha, & 0 < \alpha < 1, \\
    \log n, & \alpha = 0, \\
    1, & \alpha < 0,
  \end{cases}
\end{align}
When \(\alpha \in (0,1)\), the limit \(\mathsf{S}_\alpha\) is referred to as the \emph{$\alpha$-diversity}.  
Furthermore, for \(\alpha \in (0,1)\), the block size distribution $(\mathsf{k}_{n,j}/\mathsf{k}_n)_{j=1}^n$ converges to a discrete distribution:
\begin{align}\label{eq:intro_powerlaw}
  \forall j \in \mathbb{N}, \quad 
  \frac{\mathsf{k}_{n,j}}{\mathsf{k}_n} \; \xrightarrow{\text{a.s.}} \; 
  \mathfrak{p}_\alpha(j) 
  \equiv \frac{\alpha \prod_{i=1}^{j-1} (i - \alpha)}{j!},
\end{align}
where \(\mathsf{k}_{n,j} / \mathsf{k}_n\) denotes the proportion of blocks of size \(j\) among the current partition of $[n]$; see \cite[Lemma 3.11]{pitman2006combinatorial} and \cite{gnedin2007notes}. The limit $\mathfrak{p}_\alpha(j)$ defines a valid probability mass function on $\mathbb{N}$, known as the \emph{Sibuya distribution} in the literature \cite{sibuya1979generalized}. It has been studied in connection with the stable distribution and Mittag-Leffler distribution \cite{christoph1998discrete, devroye1993triptych, pillai1995discrete, pakes1995characterization}; see also \cite{resnick2007heavy} for its relevance in extreme value theory.
By Stirling's approximation, the tail of the discrete distribution behaves as
$$
\mathfrak{p}_\alpha(j) \sim \frac{\alpha}{\Gamma(1-\alpha)} \cdot j^{-(1 + \alpha)} \quad \text{as $j\to+\infty$}.
$$
Combined with \eqref{eq:intro_powerlaw}, this means that the block size distribution $(\mathsf{k}_{n,j}/\mathsf{k}_n)_{j=1}^n$ converges asymptotically to a power law with exponent $(\alpha+1)$. 

\paragraph{Related Literature}  
The asymptotic properties of the block sizes $(\mathsf{k}_{n,j})_{j=1}^n$, as characterized in \eqref{eq:intro_sparse}–\eqref{eq:intro_powerlaw} above, make the Gibbs partition a versatile and powerful statistical model across a wide range of domains. Notable applications include species sampling problems \cite{balocchi2022bayesian, favaro2021near, favaro2009bayesian, sibuya2014prediction, battiston2018multi}, nonparametric Bayesian inference \cite{caron2017generalized, dahl2017random, ayed2019beyond, rigon2025enriched}, disclosure risk assessment \cite{favaro2021bayesian, hoshino2001applying}, network analysis \cite{crane2016edge, naulet2021asymptotic}, natural language processing \cite{teh2006hierarchical, sato2010topic}, and forensic science \cite{cereda2022learning}. While this list is not exhaustive, it highlights the broad applicability of Gibbs-type partitions, particularly their ability to capture power-law cluster size distributions that arise naturally in many real-world settings.

In this literature, considerable attention has been devoted to estimating the discount parameter $\alpha \in (0,1)$. A natural estimator, $\hat{\alpha}_n^{\text{naive}} \equiv \log \mathsf{k}_n / \log n$, is consistent by virtue of \eqref{eq:intro_sparse}, but converges at the slow rate of $1/\log n$. This suboptimal rate results from the fact that $\hat{\alpha}_n^{\text{naive}}$ depends only on the summary statistic $\mathsf{k}_n = \sum_{j=1}^n \mathsf{k}_{n,j}$, without leveraging the full sufficient statistic $(\mathsf{k}_{n,j})_{j=1}^n$ for the likelihood \eqref{eq:intro_marginal_likelihood}. 

More refined estimators that utilize the full sufficient statistic $(\mathsf{k}_{n,j})_{j=1}^n$ have been recently studied. Let us define $\hat{\alpha}_n$ as the maximizer of the likelihood \eqref{eq:intro_marginal_likelihood} with the weight $v_{n,k}$ replaced by the Ewens--Pitman weight \eqref{eq:v_nk_ewens} with $\theta = 0$:
\begin{align*}
\hat{\alpha}_n \in \argmax_{\alpha\in (0,1)} \ \alpha^{\mathsf{k}_n} \prod_{j=2}^{n} 
    \left( \prod_{i=1}^{j-1} (i - \alpha) \right)^{\mathsf{k}_{n,j}}.
\end{align*}
We refer to this estimator as the quasi-maximum likelihood estimator (QMLE). The QMLE attains the convergence rate of $n^{-\alpha/2}$, as established in \cite{favaro2021bayesian,balocchi2022bayesian, franssen2022bernstein}. The analysis in these works builds on Kingman's representation theorem \cite{kingman1978representation, kingman1982coalescent}, which serves as the analogue of de Finetti's theorem for exchangeable random partitions. 

However, the precise asymptotic distribution of the QMLE has thus far been established only for the Ewens--Pitman model. In particular, \cite{koriyama2022asymptotic} derive the following limit distribution for the QMLE:
\begin{align}\label{eq:intro_qmle_limit_dist}
\sqrt{\mathsf{k}_n \mathfrak{i}(\alpha)} \cdot (\hat{\alpha}_n - \alpha) \xrightarrow{\mathrm{d}} \mathcal{N}(0,1),
\end{align}
where $\mathfrak{i}(\alpha)$ denotes the Fisher information of the discrete distribution with pmf $\mathfrak{p}_\alpha(j)$:
\[
\mathfrak{i}(\alpha) = \sum_{j=1}^\infty \mathfrak{p}_\alpha(j) \left( \frac{\partial}{\partial \alpha} \log \mathfrak{p}_\alpha(j) \right)^2.
\]
This result provides an asymptotic confidence interval for $\alpha$. Then, a natural question is whether the convergence in \eqref{eq:intro_qmle_limit_dist} also holds under Gibbs partition models with general weight $(v_{n,k})$.

In addition to parameter estimation, statisticians are also interested in \textit{prediction}—a practical task often more relevant in applications. Specifically, we aim to estimate the random simplex $\bm{\simp}_n \in \Delta^{\mathsf{k}_n + 1}$ defined in \eqref{eq:intro_simplex}, which describes the probability of assigning the $(n+1)$-st element to each block. This prediction problem is more challenging, as it requires simultaneous estimation of both $\alpha$ and the triangular array $(v_{n,k})$, and the length of the simplex $\bm{\simp}_n$ grows at rate $n^\alpha$. 
Indeed, the only prior work in this direction is \cite{arbel2021approximating}, which analyzes the approximation error of estimating $\bm{\simp}_n$ by that of the Ewens--Pitman partition $(\alpha,\theta)$. However, in practical settings, the parameters $(\alpha,\theta)$ are unknown and must be estimated from data. The additional error and complications introduced by this estimation step are not addressed in \cite{arbel2021approximating}.

\paragraph{Contribution}  
In this paper, we address the two aforementioned questions under the assumption that $v_{n,k}$ admits a mixture representation of the Ewens--Pitman weights $v_{n,k}(\alpha, \theta)$, where $\theta$ is integrated with respect to a mixing measure $\mu$ (see \Cref{assumption}). Under this setting, we show that the QMLE $\hat{\alpha}_n$ retains the asymptotic mixed normality given by \eqref{eq:intro_qmle_limit_dist}. In addition, we construct a fully data-driven estimator \(\bm{\hat\simp}_n\) for the simplex \(\bm{\simp}_n\) and establish a limit distribution for the $f$-divergence $\mathsf{D}_f(\bm{\hat{\simp}}_n \| \bm{\simp}_n)$, covering both the total variation distance and the Kullback--Leibler divergence. In particular, we show
\[
n \sqrt{\frac{\mathfrak{i}(\alpha)}{\mathsf{k}_n}} \cdot \tv(\bm{\hat\simp}_n, \bm{\simp}_n) \xrightarrow{\mathrm{d}} |\mathcal{N}(0,1)|,
\]
which enables predictive inference for the assignment of the $(n+1)$-th element, equipped with a confidence interval.

\editline{
\paragraph{Notation} 
For any non-negative sequence $\{a_n\}_{n=1}^\infty, \{b_n\}_{n=1}^\infty$, we write $a_n \sim  b_n$ if $\lim_{n\to+\infty} a_n/b_n = 1$,  $a_n\ll b_n$ if $\lim_{n\to+\infty}a_n/b_n = 0$, and $a_n \lesssim b_n$ if $\limsup_{n\to+\infty} a_n/b_n <+\infty$. 
}
\paragraph{Organization}
\Cref{sec:alpha_diversity} introduces the working assumptions and establishes key asymptotic properties of $\mathsf{k}_n$. \Cref{sec:parameter_estimate} focuses on the estimation of $\alpha$, presenting the QMLE and deriving its asymptotic distribution. \Cref{sec:prediction} develops an estimator $\hat{\bm{\simp}}_n$ for the simplex $\bm{\simp}_n$ and derives the asymptotic distribution of the $f$-divergence $\mathsf{D}_f(\bm{\hat{\simp}}_n \| \bm{\simp}_n)$. Finally, \Cref{sec:simulation} provides numerical experiments that illustrate and validate the theoretical findings.

\section{$\alpha$-diversity}\label{sec:alpha_diversity}

As we explained in the introduction, the Gibbs partition has a discount parameter $\alpha\in (-\infty, 1)$ and a triangular array \(\{v_{n,k}\}\), satisfying the backward recursion \eqref{eq:recursion}. Throughout this paper, we focus on the regime $\alpha\in(0,1)$, where a canonical example is given by the Ewens--Pitman models, parameterized by \(\theta > -\alpha\):
\[
v_{n,k}(\alpha, \theta) 
= \frac{
    \prod_{i=1}^{k-1} (\theta + i\alpha)
}{
    \prod_{i=1}^{n-1} (\theta + i)
}, 
\quad \theta > -\alpha.
\]
In this paper, we assume that the array \(v_{n,k}\) admits a mixture representation over this Ewens--Pitman family:

\begin{assumption}\label{assumption}
Let $\alpha\in (0,1)$. Suppose there exists a probability measure \(\mu\), independent of \(\alpha\), with bounded support \(\operatorname{supp}(\mu) \subset [\underline{\theta}, \bar{\theta}]\) for some \(\underline{\theta}, \bar{\theta} \in (-\alpha, +\infty)\) satisfying \(\underline{\theta} \le 0 \le \bar{\theta}\), such that 
\[
\forall n\in \mathbb{N}, \ \forall k\in \{1, \dots, n\}, \quad 
v_{n,k}
= \int \dd \mu(\theta) v_{n,k}(\alpha, \theta)
= \int 
\dd \mu(\theta)
\frac{
    \prod_{i=1}^{k-1} (\theta + i\alpha)
}{
    \prod_{i=1}^{n-1} (\theta + i)
}
\]
\end{assumption}
The bounded support assumption on $\mu$ is mainly imposed for technical reasons in our proofs. In \Cref{sec:simulation}, we report numerical experiments for several unbounded mixing distributions, including the half-normal distribution $|N(0,\sigma^2)|$ and half $t$-distributions $|t_{\mathrm{df}}|$ with varying degrees of freedom. These experiments are intended only as empirical evidence suggesting that the bounded-support condition may not be sharp; they should not be interpreted as a proof that the theoretical results extend to such unbounded cases. In \Cref{sec:discussion}, we briefly discuss the main difficulty in relaxing this condition, leaving a rigorous extension for future work.

We write  \(v_{n,k} = v_{n,k}(\alpha, \mu)\) to emphasize the dependence on both $\alpha$ and $\mu$. In the special case where \(\mu\) is the Dirac measure \(\delta_{\theta}\) for some \(\theta \in (-\alpha, +\infty)\), the mixture reduces to the Ewens--Pitman model \((\alpha, \theta)\). It is straightforward to verify that any triangular array $\{v_{n,k}\}$ satisfying \Cref{assumption} also satisfies the backward recursion \eqref{eq:recursion}.


It is worth discussing, among all solutions \(v_{n,k}\) to the recursion \eqref{eq:recursion}, what subset is captured by \Cref{assumption}. In this regard, \cite{gnedin2006exchangeable} shows that any Gibbs partition can be realized via conditional iid sampling from a random discrete measure known as an \emph{$\alpha$-stable Poisson--Kingman model}, parameterized by a scalar \(\alpha\in (0,1)\) and a probability measure \(\gamma\) on \((0, \infty)\). With this representation theorem, the corresponding weight $v_{n,k}$ can be written explicitly as 
\begin{align*}
  v_{n,k} = v_{n,k}(\alpha; \gamma) \equiv \int_0^\infty \frac{\alpha^k t^{-n}}{\Gamma(n - k\alpha) f_\alpha(t)} \left( \int_0^t s^{n - k\alpha -1} f_\alpha(t - s) \, \dd s \right) \dd \gamma(t),
\end{align*}
where \(f_\alpha\) is the density whose Laplace transform satisfies $\int_0^\infty e^{-\lambda x} f_\alpha(x) \, \dd x = e^{-\lambda^\alpha}$ (see  \cite{lijoi2008investigating}). Moreover, the law \(\gamma\) can be estimated by the asymptotic law of the number of blocks \(\mathsf{k}_n\) via the following transformation:
\[
\gamma =^d \mathsf{S}_\alpha^{-1/\alpha} \quad \text{where} \quad \mathsf{S}_\alpha = \lim_{n\to+\infty} n^{-\alpha} \mathsf{k}_n \text{ (a.s.)}. 
\]
Consequently, the class of solutions covered by \Cref{assumption} corresponds to the range of distributions of \(\mathsf{S}_\alpha^{-1/\alpha}\) that can be generated by mixtures over the Ewens--Pitman family via \(\mu\). In the next theorem, we briefly examine the law of \(\mathsf{S}_\alpha\). \editline{We will also visualize the limit distribution $\diversity$ for varying $\mu$ by 
plotting the histogram of $n^{-\alpha}\mathsf{k}_n$ for sufficiently large $n$ (see \Cref{fig:diversity}), demonstrating that \Cref{assumption} covers a non-trivial subclass of the Gibbs partitions that cannot be represented by the Ewens--Pitman family ($\mu=\delta_\theta$).} A comprehensive characterization of the law $\diversity$ remains open and is left for future work.

\begin{theorem}[$\alpha$-diversity]\label{thm:Lp_converge}
  Let \Cref{assumption} be satisfied. Then, there exists a positive random variable $\diversity$ with bounded second moment such that 
  \begin{align*}
    n^{-\alpha} \mathsf{k}_n \to \diversity \quad \text{almost surely and in $L_2$}.
  \end{align*}
  Furthermore, the first moment of the almost sure limit $\diversity$ is bounded as 
  \begin{align}\label{eq:alpha_diversity_moment_ineq}
          \frac{1}{\alpha} \frac{\Gamma(\underline\theta+1)}{\Gamma(\underline\theta + \alpha)} \le \E[\diversity] \le \frac{1}{\alpha} \frac{\Gamma(\bar\theta+1)}{\Gamma(\bar\theta + \alpha)}. 
  \end{align}
\end{theorem}
This theorem establishes not only the almost sure convergence but also the $L_2$ convergence, which does not directly follow from the results of \cite{gnedin2006exchangeable}. Our proof leverages a basic martingale convergence argument, inspired by the approach in \cite{bercu2024martingale}. 

\begin{remark}
When $\mu=\delta_\theta$, the convergence $n^{-\alpha}\mathsf{k}_n\to \diversity$ is well studied in the literature. In this case, the law of $\diversity$ is referred to as generalized Mittag-Leffler distribution, which has finite $p$-th moments for any $p>-(1+\theta/\alpha)$, with the following closed-form expression:
$$
\E\bigl[(\mathsf{S}_{\alpha,\delta_\theta})^p\bigr] = \frac{\Gamma(\theta+1)}{\Gamma(\theta/\alpha + 1)} \frac{\Gamma(\theta/\alpha + p+1)}{\Gamma(\theta + p\alpha + 1)}. 
$$
This is consistent with the moment bound in \eqref{eq:alpha_diversity_moment_ineq} by taking $\mu=\delta_\theta$. 
Moreover, several works have extended the convergence $n^{-\alpha}\mathsf{k}_n\to \diversity$ for $\mu=\delta_\theta$: 
\cite{dolera2020berry} proves a Berry-Esseen-type theorem, showing that the Kolmogorov distance between $n^{-\alpha}\mathsf{k}_n$ and $\diversity$ decays at the rate $n^{-\alpha}$.  \cite{bercu2024martingale} shows the CLT of the form $n^{\alpha/2} (n^{-\alpha} \mathsf{k}_n-\mathsf{S}_{\alpha, \mu}) \to^d N \sqrt{\mathsf{S}_{\alpha,\mu}'}$ for $N\sim N(0,1)\indep \diversity' =^d \diversity$, and the law of the iterated logarithm for $\mathsf{k}_n$. 
\end{remark}

Let $\ell_n(\alpha, \mu)$ be the log-likelihood of the Gibbs partition where the parameters  $(\alpha, \{v_{n,k}\})$ satisfy \Cref{assumption} for some mixture $\mu$. 
From the marginal likelihood formula \eqref{eq:intro_marginal_likelihood},  $\ell_n(\alpha, \mu)$ can be expressed as 
\begin{align}\label{eq:loglikelihood}
\ell_n(\alpha; \mu) \equiv 
 \log \Bigl(\int \dd \mu(\theta) \frac{\prod_{i=1}^{\mathsf{k}_n-1} (\theta + i\alpha)}{\prod_{i=1}^{n-1}(\theta + i)} \Bigr) + \sum_{j=2}^{n} \mathsf{k}_{n,j} \sum_{i=1}^{j-1} \log(i-\alpha).   
\end{align}
As an application of the $L_2$ convergence result in \Cref{thm:Lp_converge}, we now investigate the asymptotic behavior of the Fisher information with respect to the discount parameter $\alpha\in (0,1)$. In particular, we analyze the leading term of the variance $\E[(\partial_\alpha \ell_n(\alpha; \mu))^2]$, which characterizes the Fisher information for the parameter $\alpha$ in the Gibbs partition model. 

\begin{theorem}\label{theorem:fisher}
The Fisher Information for the discount parameter $\alpha\in(0,1)$ in the Gibbs partition model behaves asymptotically as 
  $$
  \E\Bigl[\bigl(\partial_\alpha \ell_n(\alpha; \mu)\bigr)^2\Bigr] \sim n^\alpha \E[\diversity] \mathfrak{i}(\alpha), 
  $$
  where $\diversity=\lim_{n\to+\infty}n^{-\alpha} \mathsf{k}_n$ denotes the $\alpha$-diversity, and 
  $\mathfrak{i}(\alpha)$ is the Fisher Information for $\alpha$ in the discrete distribution with pmf $\mathfrak{p}_\alpha(j)$, defined as 
  \begin{align}\label{eq:fisher_sibuya}
    \mathfrak{i}(\alpha) \equiv \sum_{j=1}^\infty \mathfrak{p}_\alpha(j) \Bigl(\frac{\partial}{\partial \alpha}\log \mathfrak{p}_\alpha(j)\Bigr)^2  \quad \text{with} \quad \mathfrak{p}_\alpha(j) 
= \frac{\alpha \prod_{i=1}^{j-1} (i - \alpha)}{j!}.
  \end{align}
\end{theorem}
With $\alpha\in(0,1)$, \Cref{theorem:fisher} implies  that the Fisher Information grows sub-linearly in $n$. When the mixing distribution $\mu$ is a Dirac measure $\mu=\delta_{\theta}$, this result recovers the result in \cite{koriyama2022asymptotic}. 

The Fisher information $\mathfrak{i}(\alpha)$ of the discrete distribution with pmf $\mathfrak{p}_\alpha(j)$ will also play a central role in constructing confidence intervals for both $\alpha$ and the predictive simplex $\bm{\simp}_n$ (see \Cref{cor:ci_alpha}, \Cref{cor:CI_additive}, and \Cref{theorem:ratio_consistency}). We introduce the following closed-form expression:

\begin{proposition}\label{theorem:fisher_info_sibuya}
  Let $\mathfrak{i}(\alpha)$ be the Fisher Information of the discrete distribution $\mathfrak{p}_\alpha(j)$ defined as \eqref{eq:fisher_sibuya}. Then, $\mathfrak{i}(\alpha)$ admits the following expression:
  \begin{align*}
    \forall \alpha\in (0,1), \quad 
    \mathfrak{i}(\alpha) = \frac{\pi}{\alpha \sin(\pi\alpha)}. 
  \end{align*}
\end{proposition}

\section{Estimation of discount parameter $\alpha$}\label{sec:parameter_estimate}
In this section, we consider the estimation of the discount parameter $\alpha\in (0,1)$ in the Gibbs partition model under \Cref{assumption}, where the mixing distribution $\mu$ is unknown. We introduce the \emph{Quasi-Maximum Likelihood Estimator} (QMLE).

\begin{definition}[Quasi-Maximum Likelihood Estimator]\label{def:qmle}
The QMLE, denoted by $\hat{\alpha}_n$, is defined as
\[
\hat{\alpha}_n \in \underset{\alpha \in (0,1)}{\operatorname{argsup}} \ \ell_n(\alpha; \delta_0) = \underset{\alpha \in (0,1)}{\operatorname{argsup}} \ \left\{ (\mathsf{k}_n - 1)\log \alpha + \sum_{j=2}^{n} \mathsf{k}_{n,j} \sum_{i=1}^{j-1} \log(i - \alpha) \right\},
\]
where $\ell_n(\alpha; \delta_0)$ is the log-likelihood \eqref{eq:loglikelihood} with the mixing distribution $\mu$ set to the Dirac measure at 0.
\end{definition}

The QMLE corresponds to the MLE under a simplified Gibbs model with $\mu = \delta_0$. Notably, its computation depends only on the block size statistics $(\mathsf{k}_{n,j})_{j=1}^n$, which, as previously noted, serve as the sufficient statistics for the Gibbs models. 
Following algebraic simplifications (see, e.g., \cite{favaro2021near, koriyama2022asymptotic, carlton1999applications}), the QMLE $\hat{\alpha}_n$ is characterized as:
\begin{align}\label{eq:qmle_cases}
\hat{\alpha}_n =
\begin{cases}
0 & \text{if } \mathsf{k}_n = 1, \\
1 & \text{if } \mathsf{k}_n = n, \\
\text{the unique solution } \alpha \in (0,1) \text{ to } \partial_\alpha \ell_n(\alpha; \delta_0) = 0 & \text{if } 1 < \mathsf{k}_n < n,
\end{cases}
\end{align}
where the stationary condition $\partial_\alpha \ell_n(\alpha; \delta_0) = 0$ is explicitly given by:
\begin{align}\label{eq:stationary_condition}
0 = \partial_\alpha \ell_n(\alpha; \delta_0) = \frac{\mathsf{k}_n - 1}{\alpha} - \sum_{j=2}^{n} \mathsf{k}_{n,j} \sum_{i=1}^{j-1} \frac{1}{i - \alpha}, \quad \alpha \in (0,1).
\end{align}
The following proposition shows that the boundary cases $\mathsf{k}_n \in \{1, n\}$ occur with vanishing probability, ensuring that the QMLE typically lies in the open interval $(0,1)$ and satisfies the stationary condition. 

\begin{proposition}\label{prop:mle_existence}
Under \Cref{assumption}, it holds that 
\begin{align*}
\PP(\mathsf{k}_n = 1) &\lesssim n^{-(\alpha + \underline\theta)} = o(1), \\
\PP(\mathsf{k}_n = n) &\lesssim n^{\bar{\theta}(\alpha^{-1} - 1)} \alpha^n = o(1)
\end{align*}
\end{proposition}

Combining \eqref{eq:qmle_cases} and \Cref{prop:mle_existence}, we conclude that with high probability, the QMLE $\hat{\alpha}_n$ lies in the open interval $(0,1)$ and is the unique solution to the equation \eqref{eq:stationary_condition}.

We now turn to the asymptotic distribution of the QMLE $\hat{\alpha}_n$, and show that it is \emph{asymptotically mixed normal}. To formulate this result rigorously, we introduce a notion of stochastic convergence known as \emph{stable convergence}. We begin by defining this concept in a general setting.

\begin{definition}[{\cite[Definition 3.15]{hausler2015stable}}]\label{def:stable_conv}
  Let $(\Omega, \mathcal{F}, P)$ be a probability space, and let $\mathcal{X}$ be a separable metrizable topological space equipped with its Borel $\sigma$-field $\mathcal{B}(\mathcal{X})$. 
For a sub-$\sigma$-field $\mathcal{G} \subset \mathcal{F}$, a sequence of $(\mathcal{X}, \mathcal{B}(\mathcal{X}))$-valued random variables $(X_n)_{n\geq 1}$ is said to converge $\mathcal{G}\operatorname{-stably}$ to $X$, denoted by $X_n \rightarrow X$ $\mathcal{G}\operatorname{-stably}$, if 
 \begin{align}\label{eq:stable_convergence_rv}
    \lim_{n\rightarrow\infty} \E\Bigl[f\E\bigl[h(X_n)\mid\mathcal{G}\bigr]\Bigr]
     = \E\Bigl[f\E\bigl[h(X)\mid\mathcal{G}\bigr]\Bigr]
\end{align}
for any $\mathcal{F}$-measurable and integrable function $f$, and any bounded continuous function $h$ on $\mathcal{X}$. 
\end{definition}
Taking $f \equiv 1$ in \eqref{eq:stable_convergence_rv}, we see that the stable convergence is stronger than the usual weak convergence. When $\mathcal{G}$ is the trivial $\sigma$-field $\{\emptyset, \Omega\}$, the $\mathcal{G}$-stable convergence reduces to the weak convergence. Indeed, in this case we have 
$\E[f\E[h(X_n)|\mathcal{G}]] = (\int f \diff P) \cdot \E[h(X_n)]$ so that \eqref{eq:stable_convergence_rv} becomes $\E[h(X_n)]\to \E[h(X)]$ for any bounded continuous function $h$, which is the definition of weak convergence. 


\cite[Theorem 3.18]{hausler2015stable} extends several known results for weak convergence to the setting of stable convergence. Below, we present the continuous mapping theorem and Slutsky’s lemma for stable convergence, both of which will be used throughout this paper.

\begin{lemma}[{\cite[Theorem 3.18]{hausler2015stable}}]\label{lm:CS_stable}
Let $(\mathcal{X}, \mathcal{B}(\mathcal{X}))$ and $(\mathcal{Y}, \mathcal{B}(\mathcal{Y}))$ be separable metrizable spaces with a metric $d$. Suppose $(X_n)_{n\geq 1}$ is a sequence of $(\mathcal{X}, \mathcal{B}(\mathcal{X}))$-valued random variables satisfying $X_n \rightarrow X$ $\mathcal{G}\operatorname{-stably}$. Then, the following hold:
\begin{enumerate}
    \item If $\mathcal{X} = \mathcal{Y}$ and $d (X_n, Y_n) \rightarrow 0$ in probability, then $Y_n \rightarrow X$ $\mathcal{G}\operatorname{-stably}$.
    \item If $Y_n \to Y$ in probability and $Y$ is $\mathcal{G}$-measurable, then $(X_n, Y_n) \rightarrow (X, Y)$ $\mathcal{G}\operatorname{-stably}$.
    \item  If $g: \mathcal{X} \rightarrow \mathcal{Y}$ is $(\mathcal{B}(\mathcal{X}), \mathcal{B}(\mathcal{Y}))$-measurable and continuous $P^{X}$-a.s., then $g(X_n) \rightarrow g(X)$ $\mathcal{G}\operatorname{-stably}$.
\end{enumerate}
\end{lemma}
It is important to emphasize that the second statement allows the random variable $Y$ to be $\mathcal{G}$-measurable, rather than merely a deterministic constant. This highlights that Slutsky’s lemma holds in a stronger form under stable convergence. \\

We are now ready to state the limit distribution of the QMLE. In what follows, we take the sub $\sigma$-field $\mathcal{G}$ to the limiting $\sigma$-field $\mathcal{F}_\infty = \sigma(\cup_{n=1}^\infty \mathcal{F}_n)$, where $\mathcal{F}_n=\sigma(\Pi_n)$ is the $\sigma$-field generated by the random partition $\Pi_n\in \mathcal{P}_n$ of $[n]$. With this notation, we will show that the QMLE $\hat{\alpha}_n$, scaled by $n^{\alpha/2}$, converges $\mathcal{F}_\infty$-stably to a variance mixture of centered normal distributions.

\begin{theorem}[Asymptotic Mixed Normality]\label{thm:mle_clt}
  Let $\mathcal{F}_\infty = \sigma(\cup_{n=1}^\infty \mathcal{F}_n)$ with $\mathcal{F}_n=\sigma(\Pi_n)$. 
  Then, we have 
  $$
  \sqrt{n^\alpha \mathfrak{i}(\alpha)} \cdot (\hat{\alpha}_n-\alpha) \to  N/\sqrt{\diversity} \quad \stable
  $$
  where $N\sim \mathcal{N}(0,1)$ is independent of $\mathcal{F}_\infty$, and $\diversity=\lim_{n\to+\infty} n^{-\alpha} \mathsf{k}_n$ is the $\alpha$-diversity, which is $\mathcal{F}_\infty$-measurable. 
\end{theorem} 

\Cref{thm:mle_clt} was established for the Ewens--Pitman model by \cite{koriyama2022asymptotic}. Here we extend the result to the general Gibbs partition under \Cref{assumption}. 
\Cref{thm:mle_clt} implies that the QMLE converges at rate $n^{-\alpha/2}$, which is slower than the standard iid parametric rate due to the fact that $\alpha\in(0,1)$. More significantly, the limiting distribution is not a fixed normal but a variance mixture of centered normals, where the randomness in the variance arises from the $\alpha$-diversity $\diversity$. 

In conjunction with the asymptotic behavior of the Fisher Information in \Cref{theorem:fisher}, we obtain:
$$
\sqrt{\underbrace{\E\Bigl[\bigl(\partial_\alpha \ell_n(\alpha; \mu)\bigr)^2\Bigr]}_{\text{Fisher Information}}} \cdot (\hat\alpha_n -\alpha) \to  \underbrace{N\cdot \sqrt{\frac{\E[\diversity]}{\diversity}}}_{\text{variance $\ge 1$}} \quad \stable. 
$$
By Jensen's inequality and the independence of $N\sim \mathcal{N}(0,1)$ and $\diversity$, it follows that the variance of the limit distribution exceeds $1$. This contrasts with the standard parametric iid setting, where the MLE achieves the Cramér--Rao lower bound. 

As a direct consequence of  \Cref{thm:mle_clt}, we obtain an asymptotic confidence interval for $\alpha$ using the almost sure convergence  $n^{-\alpha} \mathsf{k}_n \to \diversity$ and 
Slutsky's lemma for stable convergence (see
\Cref{lm:CS_stable}-(2)). 
\begin{corollary}[Confidence interval of $\alpha$]\label{cor:ci_alpha}
  Let $\mathsf{k}_n$ be the number of non-empty blocks and $\mathfrak{i}(\alpha)$ be the Fisher Information of the discrete distribution $\mathfrak{p}_\alpha(j)$ (see \Cref{theorem:fisher_info_sibuya}). Then, we have 
  \begin{align*}
    \sqrt{\mathsf{k}_n \mathfrak{i}(\alpha)} (\hat{\alpha}_n-\alpha) \to N \quad  \stable
  \end{align*}
  where $N\sim \mathcal{N}(0,1)$ is independent of $\mathcal{F}_\infty$. Consequently, for any $\epsilon\in (0,1)$, letting $\tau_{1-\epsilon/2}$ be the $(1-\epsilon/2)$-quantile of $\mathcal{N}(0,1)$, we have 
  $$
  \lim_{n\to+\infty} \PP\left(\alpha \in \Bigl[\hat\alpha_n \pm \frac{\tau_{1-\epsilon/2}}{\sqrt{\mathsf{k}_n  \mathfrak{i}(\hat{\alpha}_n) }}\Bigr]\right) = 1-\epsilon. 
  $$
\end{corollary}

\section{Estimation of probability simplex $\simp_n$}\label{sec:prediction}
In this section, we study the estimation of the probability simplex $\bm{\simp}_n$, defined as:
\begin{align}\label{eq:def_simplex_true}
  \bm\simp_n = (\simp_{n,0}, \dots, \simp_{n, \mathsf{k}_n})  \equiv \Bigl(\frac{v_{n+1, \mathsf{k}_{n}+1}}{v_{n, \mathsf{k}_n}} , 
  \frac{v_{n+1, \mathsf{k}_n}}{v_{n,\mathsf{k}_n}}(|U_1|-\alpha), \cdots,   \frac{v_{n+1, \mathsf{k}_n}}{v_{n,\mathsf{k}_n}}(|U_{\mathsf{k}_n}|-\alpha)
  \Bigr),
\end{align}
\editline{and we denote $\bm\simp_n(I)=\sum_{i\in I} \simp_{n,i}$ for all $I\subset \{0, 1, \dots \mathsf{k}_n\}$.} 
Recall from \eqref{eq:rule} that, under the Gibbs partition model, given a partition $(U_1, U_2, \dots, U_{\mathsf{k}_n})$ of $[n]$, the next $(n+1)$-th item is assigned according to 
\begin{align*}
  (n+1) \in \begin{cases}
    \text{a new block} & \text{with probability } \simp_{n,0} \\
    \text{$U_i$} & \text{with probability } \simp_{n,i}, \quad i \in \{1, 2, \dots \mathsf{k}_n\}
  \end{cases}.
\end{align*}
We aim to estimate this
random simplex $\bm{\simp}_n$. Define the estimator $\bm{\hat\simp}_n$ as 
\begin{align}\label{eq:def_simplex_estimator}
\bm{\hat\simp}_n = (\hat\simp_{n,0}, \hat\simp_{n,1}, \dots, \hat\simp_{n,\mathsf{k}_n}) \equiv \Bigl(
  \frac{\mathsf{k}_n}{n}\hat{\alpha}_n,  
\frac{|U_1|-\hat{\alpha}_n}{n}, \dots, \frac{|U_{\mathsf{k}_n}|-\hat{\alpha}_n}{n}
  \Bigr), 
\end{align}
where $\hat{\alpha}_n$ is the QMLE. This estimator is derived from the true simplex $\bm{\simp}_{n}$ by taking \( \mu = \delta_0 \) and replacing \( \alpha \) with \( \hat{\alpha}_n \). 

\begin{remark}\label{rem:comparison}
\editline{
  One may also consider alternative estimators. For example, setting $\alpha=0$ yields the simple frequency estimator
$$
\bm{\tilde{\simp}}_n \equiv \Bigl(
0, \frac{|U_1|}{n}, \dots, \frac{|U_{\mathsf{k}_n}|}{n}
\Bigr).
$$
Another possibility is to keep $\alpha=\hat{\alpha}_n$ while setting $\mu=\delta_\theta$ for some $\theta \neq 0$. However, the former estimator is suboptimal in terms of convergence rate, whereas the latter has exactly the same limit distribution as our estimator $\bm{\hat{\simp}}_n$ (see \ref{sec:comparison}). We therefore focus on $\bm{\hat{\simp}}_n$.
}  
\end{remark}

We investigate the asymptotic behavior of the $f$-divergence between \( \bm{\hat{\simp}}_n \) and \( \bm{\simp}_n \):
\begin{align*}
  \mathsf{D}_f(\bm{\hat{\simp}}_n \| \bm{\simp}_n) \equiv \sum_{i=0}^{\mathsf{k}_n} \simp_{n,i} \cdot f\left(\frac{\hat{\simp}_{n,i}}{\simp_{n,i}}\right),
\end{align*}
where \( f: [0, \infty) \to \mathbb{R} \) is a convex function with \( f(1) = 0 \). Notable examples include:
\begin{itemize}
  \item KL divergence: \( f(x) = x \log x \), giving
  \[ \mathsf{KL}(\bm{\hat{\simp}}_n \| \bm{\simp}_n) = \sum_{i=0}^{\mathsf{k}_n} \hat{\simp}_{n,i} \cdot \log\left(\frac{\hat{\simp}_{n,i}}{\simp_{n,i}}\right). \]
  \item Total Variation (TV) distance: \( f(x) = |x - 1| / 2 \), yielding
  \[\tv(\bm{\hat{\simp}}_n, \bm{\simp}_n) = \frac{1}{2} \sum_{i=0}^{\mathsf{k}_n} |\hat{\simp}_{n,i} - \simp_{n,i}| = \sup_{I \subset \{0, 1, \dots, \mathsf{k}_n\}} \left| \bm{\hat{\simp}}_n(I)-\bm{\simp}_n(I) \right|. \]
\end{itemize}
We first consider the case where $f$ is locally twice differentiable and its second derivative is Hölder continuous at $1$. That is, there exists $\delta > 0, C > 0$, and $\beta > 0$ such that $f$ is twice differentiable on $(1 - \delta, 1 + \delta)$, and
$|f''(1 + x) - f''(1)| \le C |x|^\beta$ for all $|x| < \delta$. 
\begin{theorem}\label{thm:f_div}
  Suppose $f$ is locally twice differentiable and its second derivative is Hölder continuous at $1$. Then, we have 
  \begin{align*}
     n\cdot \mathsf{D}_f(\bm{\hat\simp}_n|| \bm{\simp}_n ) \to \frac{\alpha f''(1)}{2} |N|^2 \quad \stable
  \end{align*}
  where $N\sim \mathcal{N}(0,1)\indep \mathcal{F}_\infty$. 
\end{theorem}
For \( f(x) = x \log x \), this gives \( n \cdot \mathsf{KL}(\bm{\hat{\simp}}_n \| \bm{\simp}_n) \to \frac{\alpha}{2}|N|^2 \). Thus, KL divergence decays at the rate \( n^{-1} \), matching the typical parametric iid cases. By Pinsker's inequality ($\tv \le \sqrt{2^{-1}\mathsf{KL}}$), this implies \(\tv(\bm{\hat\simp}_n, \bm{\simp}_n)  = O_p(n^{-1/2}) \). However, a sharper analysis reveals $\E[\tv(\bm{\hat\simp}_n, \bm{\simp}_n)^2]^{1/2} \lesssim n^{-1+\frac{\alpha}{2}}$ (note $n^{-1+\frac{\alpha}{2}} \ll  n^{-1/2}$ by $\alpha\in(0,1)$).



\begin{theorem}\label{theorem:tv}
  Let $\tv(\bm{\hat\simp}_n, \bm{\simp}_n) = \frac{1}{2}\sum_{i=0}^{\mathsf{k}_n} |\hat\simp_{n,i}-\simp_{n,i}|$. Then:
  \begin{itemize}
    \item $n^{1-\frac{\alpha}{2}} \cdot  \tv(\bm{\hat\simp}_n, \bm{\simp}_n)$ is bounded in $\mathcal{L}_2$, i.e., $\E[\tv(\bm{\hat\simp}_n, \bm\simp_n)^2]^{1/2} \lesssim n^{-1+\frac{\alpha}{2}}$. 
    \item The limit distribution is given by
    $$
      n^{1-\frac{\alpha}{2}} \cdot  \tv(\bm{\hat\simp}_n,  \bm\simp_n)  \to \sqrt{\frac{\diversity}{\mathfrak{i}(\alpha)}} |N| \quad \stable
    $$
  where $\diversity=\lim_{n\to+\infty} \mathsf{k}_n/n^\alpha$ and $N\sim \mathcal{N}(0,1) \indep \mathcal{F}_\infty$. 
  \end{itemize}
\end{theorem}
Combining the second claim and Slutsky's lemma in \Cref{lm:CS_stable}-(2),  using the identity $\tv(\bm{\hat\simp}_n, \bm{\simp}_n) = \sup_{I\subset \{0, \dots, \mathsf{k}_n\}} |\sum_{i\in I} (\hat\simp_{n,i}-\simp_{n,i})|$, we obtain the following corollary:
\begin{corollary}[Uniform confidence interval]\label{cor:CI_additive} 
    Let $\mathsf{k}_n$ be the number of non-empty blocks and $\mathfrak{i}(\alpha)$ be the Fisher Information of the discrete distribution $\mathfrak{p}_\alpha(j)$ (see \Cref{theorem:fisher_info_sibuya}). Then, we have 
  $$
  n \sqrt{\frac{\mathfrak{i}(\alpha)}{\mathsf{k}_n}}\cdot 
   \tv(\bm{\hat\simp}_n, \bm\simp_n) \to  |N|  \quad \stable. 
  $$
  Thus, for any \( \epsilon \in (0,1) \), letting \( \tau_{1 - \epsilon/2} \) denote the \((1 - \epsilon/2)\)-quantile of $\mathcal{N}(0,1)$, we have 
  \begin{align}\label{eq:additive_CI}
  \lim_{n\to+\infty}\PP\left(
  \forall I\subset \{0, 1, \dots \mathsf{k}_n\}, \quad 
\bm{\simp}_n(I) \in  \Bigl[
\bm{\hat\simp}_n(I) \pm
    \frac{\sqrt{\mathsf{k}_n}}{n} \frac{\tau_{1-\epsilon/2}}{\sqrt{\mathfrak{i}(\hat{\alpha}_n)}} 
  \Bigr]\right) = 1-\epsilon. 
  \end{align}
\end{corollary}
\editline{
We refer to this confidence interval as uniform confidence interval (CI) because it works uniformly for all $I\subset \{0, 1, \dots, \mathsf{k}_n\}$. On the other hand, the width $ \frac{\sqrt{\mathsf{k}_n}}{n}$ is the same for every $I$, so that the relative width $\frac{\sqrt{\mathsf{k}_n}}{n}/\bm{\hat\simp}_n(I)$ may fail to converge to $0$ for some $I$. 

Let us characterize such subsets $I$. When $0\in I$, using \( \hat{\simp}_{n,0} = \mathsf{k}_n \hat{\alpha}_n / n \) and $\bm{\hat\simp}_n(I)=\sum_{i\in I} \hat{\simp}_{n,i}\ge \hat{\simp}_{n,0} $, the relative width is bounded as
$$
\frac{\frac{\sqrt{\mathsf{k}_n}}{n}}{\bm{\hat\simp}_n(I)} \le \frac{\frac{\sqrt{\mathsf{k}_n}}{n}}{\hat\simp_{n,0}}  = \frac{1}{\sqrt{\mathsf{k}_n} \hat\alpha_n} = O_p(n^{-\alpha/2}) = o_p(1). 
$$
In other words, if $I$ includes $0$, then the relative width for $\bm{\simp}_n(I)$ converges to $0$. On the other hand, when $0 \not\in I$ (equivalently $I\subset \{1, \dots, \mathsf{k}_n\}$), using \( \hat{\simp}_{n,i} = (|U_i| - \hat{\alpha}_n)/n \) for \( i \in \{1, \dots \mathsf{k}_n\} \), we have 
$$
\frac{\frac{\sqrt{\mathsf{k}_n}}{n}}{\bm{\hat\simp}_n(I)} = \frac{\sqrt{\mathsf{k}_n}}{\sum_{i \in I} (|U_i| - \hat{\alpha}_n)}.
$$
This shows that, if $I\subset \{1, \dots, \mathsf{k}_n\}$ and $\sum_{i\in I} |U_i|$ is relatively small compared with $\sqrt{\mathsf{k}_n}$, the relative width of the uniform CI for $\bm{\simp}_n(I)$ does not converge to $0$. We therefore propose an alternative confidence interval whose relative width shrinks for all such subsets $I\subset \{1, \dots \mathsf{k}_n\}$. 
}
\begin{theorem}[Local confidence interval]\label{theorem:ratio_consistency}
For any random variable $\delta_n$ satisfying $\mathsf{k}_n^{-1} \le \delta_n = o_p(n^{-\alpha/2})$,  define the random set $\mathcal{I}_n$ as 
  \begin{align}\label{eq:def_random_subsets}
      \mathcal{I}_n \equiv \biggl\{
    I\subset \{1, \dots \mathsf{k}_n\}: \frac{1}{|I|}\sum_{i\in I} |U_i| \le n \cdot \delta_n
\biggr\}. 
  \end{align}
Then, it holds that 
\begin{align}\label{eq:forall_I}
  \lim_{n\to+\infty} \PP\Biggl(\forall I\in \mathcal{I}_n, \quad \bm\simp_n(I) \in \biggl[\bm{\hat{\simp}}_n(I) \pm \frac{|I|}{n \sqrt{\mathsf{k}_n}} \frac{\tau_{1-\epsilon/2}}{\sqrt{\mathfrak{i}(\hat\alpha_n)}}\biggr]\Biggr) = 1-\epsilon.
\end{align}
\end{theorem}
\editline{
We may choose the threshold $\delta_n$, for example, as $\delta_n = \mathsf{k}_n^{-0.5001}$ or $\delta_{n} = (\sqrt{\mathsf{k}_n} \log n)^{-1} \vee \mathsf{k}_n^{-1}$, both of which satisfy the required condition. 
Note that the random set $\mathcal{I}_n$ is not empty, especially $\{1, \dots, \mathsf{k}_n\} \in \mathcal{I}_n$, since $\mathsf{k}_n^{-1} \sum_{i=1}^{\mathsf{k}_n} |U_i| =\mathsf{k}_n^{-1} n\le n \delta_n$ by the condition $\delta_n \ge1/\mathsf{k}_n$. 
We refer to this new confidence interval as local CI, since its width depends on $|I|$, whereas it is valid only for subsets $I$ in  $\mathcal{I}_n$. 
}

\editline{
Finally, we claim that this local CI and the previous uniform CI complement each other in the following sense. 
\begin{proposition}\label{prop:compare_uniform_local}
The uniform CI and the local CI satisfy the following. 
\begin{enumerate}
  \item The absolute width of the local CI is at most that of the uniform CI:
  \begin{align*}
    \max_{I\in \mathcal{I}_n} \frac{|I|}{n \sqrt{\mathsf{k}_n}} = \frac{\sqrt{\mathsf{k}_n}}{n}.
  \end{align*}
  \item The relative width of the local CI converges to $0$ uniformly for all $I\in \mathcal{I}_n$:
  \begin{align*}
  \max_{I\in \mathcal{I}_n} \frac{\frac{|I|}{n \sqrt{\mathsf{k}_n}}}{\bm{\hat\simp}_n(I)} \le \frac{1}{\sqrt{\mathsf{k}_n} (1-\hat{\alpha}_n)} = O_p(n^{-\alpha/2}).
\end{align*}
\item 
For all $I\subset \{0, 1, \dots, \mathsf{k}_n\}$ but $I\notin \mathcal{I}_n$, the relative width of the uniform CI is uniformly bounded  as 
\begin{align*}
  \max_{I\subset \{0, 1, \dots, \mathsf{k}_n\}, I \notin \mathcal{I}_n} \frac{\frac{\sqrt{\mathsf{k}_n}}{n}}{\bm{\hat\simp}_n(I)} \le \frac{1}{\sqrt{\mathsf{k}_n}\hat\alpha_n} \vee \frac{\sqrt{\mathsf{k}_n}}{n\delta_n -1}.
\end{align*}
\end{enumerate}
\end{proposition}
For the third claim, we may choose the threshold $\delta_n$ in the definition of $\mathcal{I}_n$ as $\delta_n = (\sqrt{\mathsf{k}_n} \log n)^{-1} \vee \mathsf{k}_n^{-1}$. Then, the relative width of the uniform CI converges to $0$ uniformly:
$$
  \max_{I\subset \{0, 1, \dots, \mathsf{k}_n\}, I \notin \mathcal{I}_n}  \frac{\frac{\sqrt{\mathsf{k}_n}}{n}}{\bm{\hat\simp}_n(I)} = O_p(n^{-\alpha/2} \vee 
  n^{-1+\alpha} \log n
  ) = o_p(1).
$$
Combined with the second claim, we see that the uniform CI and the local CI complement each other; the local CI provides a valid confidence interval (its relative width converges to $0$) for all $I\in \mathcal{I}_n$, whereas the uniform CI gives a valid confidence interval for any other $I \notin \mathcal{I}_n$. 
}

\section{Numerical simulation}\label{sec:simulation}
\editline{
We first verify \Cref{thm:Lp_converge}. The main purpose is to visualize the limit distribution $\diversity = \lim_{n\to+\infty}n^{-\alpha}\mathsf{k}_n$ under \Cref{assumption}
 for various mixing distributions $\mu$. We first fix $\alpha=0.3$ and consider the discrete distributions, including the Dirac measures $\delta_0, \delta_3, \delta_{10}$ and the point mass mixture $3^{-1}(\delta_0 + \delta_3 + \delta_{10})$. 
 We generate $10000$ independent realizations of the random partition processes up to $n=50000$ and we plot the histogram of $n^{-\alpha} \mathsf{k}_n$. 
While the limit distribution of $\diversity$ for the Ewens--Pitman family $\mu=\delta_\theta$ is unimodal, \Cref{fig:diversity} suggests that choosing a multimodal $\mu$ can make $\diversity$ multimodal. Next, we consider continuous mixing distributions, including the half-norm distribution and the uniform distribution, and plot the histogram of $n^{-\alpha}\mathsf{k}_n$ in \Cref{fig:diversity}. We see that these continuous distributions induce larger variance than the Ewens--Pitman family $(\mu=\delta_\theta)$. In summary, these numerical simulations suggest that the subclass of Gibbs partitions satisfying \Cref{assumption} includes nontrivial families that are not covered by the Ewens--Pitman family. 

\begin{figure}
  \centering
  \includegraphics[width=\linewidth]{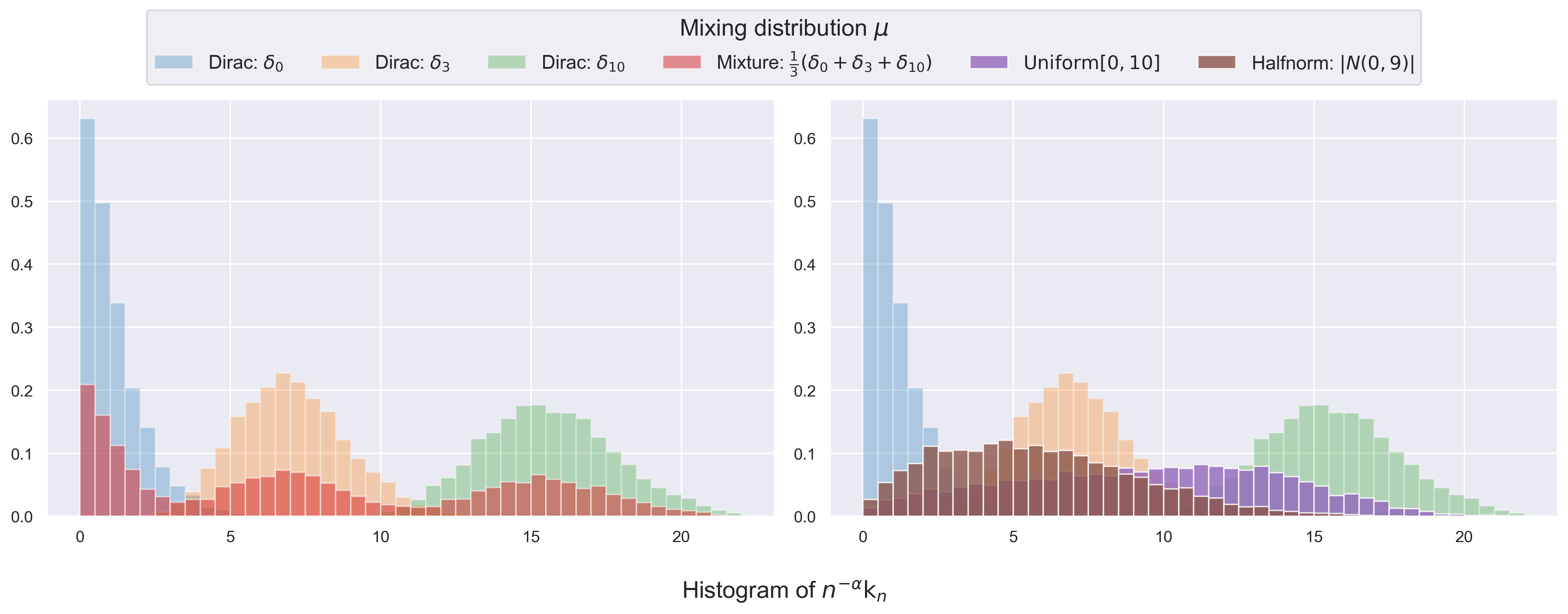}
  \caption{Histogram of the number of nonempty blocks $\mathsf{k}_n$ normalized by $n^{\alpha}$. Parameters are taken as $n=50000$, $\alpha=0.3$, and sample size = $10000$. }
  \label{fig:diversity}
\end{figure}
}

Next, we verify our limit distribution results; our theoretical findings (\Cref{cor:ci_alpha}, \Cref{theorem:tv}, and \Cref{thm:f_div}) can be summarized as 
\begin{align}\label{eq:result_summary}
  \sqrt{\mathsf{k}_n \mathfrak{i}(\alpha)} (\hat\alpha_n-\alpha) \to \mathcal{N}(0,1), \quad n \sqrt{\frac{\mathfrak{i}(\alpha)}{\mathsf{k}_n}}\tv(\bm{\hat\simp}_n, \bm{\simp}_n) \to |\mathcal{N}(0,1)|, \quad \frac{2n}{\alpha} \mathsf{KL}(\bm{\hat{\simp}}_n||\bm{\simp}_n)\to |\mathcal{N}(0,1)|^2. 
\end{align}
To empirically validate these results, we conduct a simulation study and visualize the convergence through QQ plots.
\editline{
  We fix $\alpha=0.8$ and consider the following mixing distribution $\mu$: the Dirac measure $\delta_0$, the two-point mixture $\frac{1}{2}\delta_0 + \frac{1}{2}\delta_3$, the uniform distribution over $[0,3]$, and the half-norm distribution $|\mathcal{N}(0,1)|$. 
We generate $1000$ independent realizations of the random partition process up to $n=20000$. For each realization, we compute the QMLE $\hat\alpha$, the TV distance $\tv(\bm{\hat\simp}_n, \bm{\simp}_n)$, and the KL divergence $\mathsf{KL}(\bm{\hat{\simp}}_n, \bm{\simp}_n)$. 
\Cref{fig:qq} shows the QQ plots for each of the three normalized statistics, demonstrating close agreement with the respective limiting distributions and thereby corroborating our theoretical claims under the assumptions of the paper. Interestingly, the empirical quantiles also match the theoretical quantiles well for the half-normal mixing distribution, which is not covered by Assumption 1. We also repeated the experiment for the half t-distribution $\mu=|t_{\mathrm{df}}|$ with varying degrees of freedom $\mathrm{df}$; the corresponding QQ plots show similar empirical behavior. These additional simulations suggest that the bounded-support assumption may be relaxable, but they do not provide a proof of such an extension.}

\begin{figure}[htbp]
  \centering
  \includegraphics[width=\linewidth]{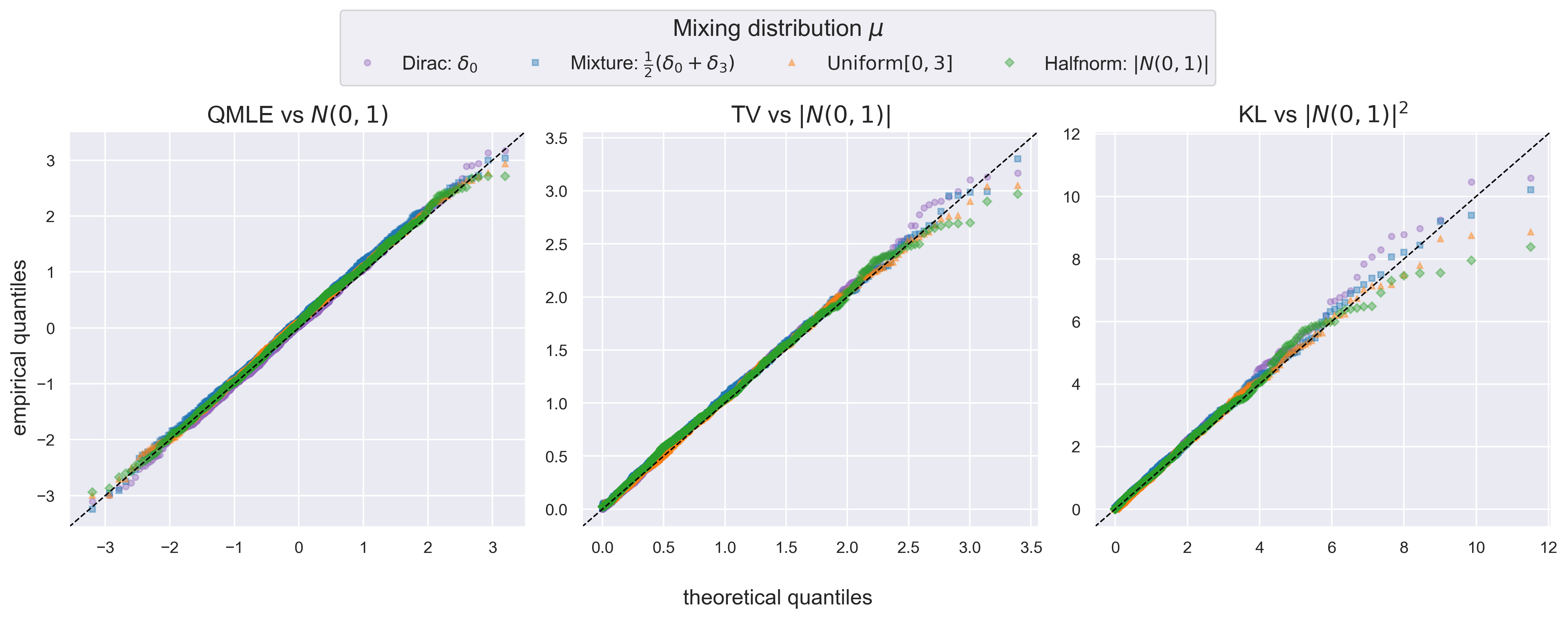}
    \includegraphics[width=\linewidth]{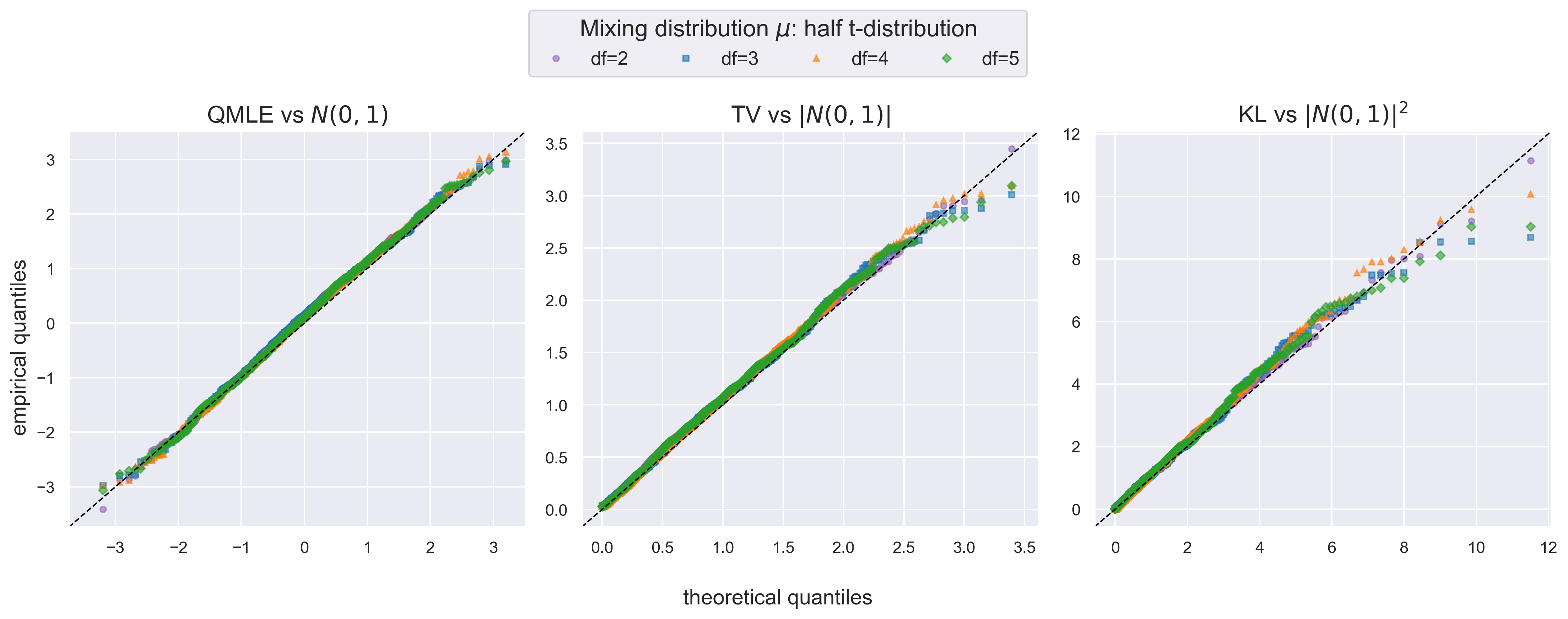}
      \caption{QQ plots comparing empirical quantiles of normalized statistics to the corresponding theoretical distributions (see \eqref{eq:result_summary}) with parameters $n=20000$, $\alpha=0.8$, and sample size = $1000$.}
      \label{fig:qq}
\end{figure}

\editline{
Finally, we verify \Cref{theorem:ratio_consistency}. Since $\mathcal{I}_n$ is a subset of the power set $2^{\{1, \dots, \mathsf{k}_n\}}$ that can be exponentially large in $n$, we verify a tractable consequence. Note that the proof of \Cref{theorem:ratio_consistency} reveals that the main claim \eqref{eq:forall_I} also holds with $\forall I\in \mathcal{I}_n$ replaced by $\exists I\in \mathcal{I}_n$ (see \Cref{rm:exist_I}). 
  Therefore, for any distribution $\rho$ on $\mathcal{I}_n$, it holds that 
    \begin{align}\label{eq:random_I}
    \lim_{n\to+\infty} \PP_{I\sim \rho} \Biggl(\bm\simp_n(I) \in \biggl[\bm{\hat{\simp}}_n(I) \pm \frac{|I|}{n \sqrt{\mathsf{k}_n}} \frac{\tau_{1-\epsilon/2}}{\sqrt{\mathfrak{i}(\hat\alpha_n)}}\biggr]\Biggr) = 1-\epsilon, 
  \end{align}
where $\PP_{I\sim \rho}$ is taken over the randomness of the random partition and the randomness of $I\sim \rho$. 
  In numerical simulation, we verify \eqref{eq:random_I} by taking $\rho$ to be the uniform distribution $\text{Unif}(\mathcal{I}_n)$ and setting $\epsilon=0.05$, while $\delta_n = \mathsf{k}_n^{-0.51}$ in the definition of $\mathcal{I}_n$. We generate $1000$ independent realizations of the random partition process up to $n=40000$, and for each realization, we sample $I\sim \text{Unif}(\mathcal{I}_n)$ using rejection sampling and check whether $\bm\simp_n(I)$ falls within the confidence interval. We report the empirical coverage in \Cref{tab:coverage}. We see that the coverage probability is close to $0.95$, in agreement with \eqref{eq:random_I}. 

  All code used to produce the above results is available at the following GitHub repository \url{https://github.com/TakuyaKoriyama/gibbs-partition-code}. 

\begin{table}[htpb]
\centering
\caption{Empirical coverage of the local confidence interval \eqref{eq:random_I}. We fix $n=40000$ and sample size $=1000$. }
\label{tab:coverage}
\begin{tabular}{c c c c c c}
\hline
& & \multicolumn{4}{c}{Mixing distribution $\mu$} \\
\cline{3-6}
\multirow{4}{*}{$\alpha$}
& 
& Dirac $\delta_{0}$
& Mixture $\frac{1}{2}(\delta_0 + \delta_3)$
& $\mathrm{Unif}[0,3]$
& Half-normal $|\mathcal{N}(0,1)|$ \\
\cline{2-6}
& 0.6 & 0.926 & 0.936 & 0.947 & 0.940 \\
& 0.7 & 0.933 & 0.937 & 0.936 & 0.936 \\
& 0.8 & 0.944 & 0.947 & 0.946 & 0.954 \\
\hline
\end{tabular}
\end{table}
}

\section{Discussion}\label{sec:discussion}
In this paper, we focused on statistical inference for the discount parameter $\alpha$, which determines the large-sample behavior of exchangeable Gibbs partitions and plays a central role in predictive inference. A related question may be formulated for other asymptotic theories of random partitions. For example, the statistical-mechanical approach to integer partitions studies limit shapes of Young diagrams by representing a partition of $n$ through its occupation numbers $\{\mathsf{k}_{n,j}\}_{j\ge1}$, where $\mathsf{k}_{n,j}$ denotes the number of parts of size $j$ and $\sum_{j\ge1} j\mathsf{k}_{n,j}=n$; see \cite{vershik1996statistical, vershik2001limit}. A particularly relevant class in this direction is given by multiplicative measures on integer partitions. In the expansive multiplicative models studied by \cite{erlihson2008limit}, the probability of a configuration $\{\mathsf{k}_{n,j}\}_{j\ge1}$ is specified, up to normalization, by weights of the form
$$
\prod_{j\ge1}\frac{a_j^{\mathsf{k}_{n,j}}}{\mathsf{k}_{n,j}!},
$$
where the weights satisfy $a_j\sim Cj^{p-1}$ for some constants $C, p>0$. The parameter $p$ then determines the corresponding limit shape of the Young diagram. Developing inference procedures for such structural parameters, and using them to extrapolate large-sample profiles, in analogy with the estimation of $\alpha$ and its use in predictive inference studied here, would be an interesting direction for future work.

We also note that our theorems require the mixing distribution $\mu$ to have bounded support, as in \Cref{assumption}. This assumption is used to obtain uniform bounds for terms that can be expressed as expectations of suitable test functions under the tilted measures
$$
1\le  k\le  n, \quad 
\dd \mu_{n,k}(\theta) \equiv
\frac{\dd \mu(\theta) v_{n,k}(\alpha;\theta)}{\int \dd\mu(\theta') v_{n,k}(\alpha;\theta')}
\quad \text{with} \quad
v_{n,k}(\alpha;\theta)
= \frac{\prod_{i=1}^{k-1} (\theta + i\alpha)}{\prod_{i=1}^{n-1}(\theta + i)},
$$
In particular, in several parts of the proof, the boundedness condition  \(\operatorname{supp}(\mu) \subset [\underline{\theta}, \bar{\theta}]\) for some \(\underline{\theta}, \bar{\theta} \in (-\alpha, +\infty)\) 
allows us to control integrals of the form $\int \dd\mu_{n,k}(\theta) h(\theta)$, for any test function $h$, simply by
$$
\inf_{\theta\in [\underline{\theta}, \bar{\theta}]} h(\theta) \le \int \dd\mu_{n,k}(\theta) h(\theta) \le \sup_{\theta\in [\underline{\theta}, \bar{\theta}]} h(\theta),
$$
see \ref{sec:tilted_measure}. If $\mu$ has unbounded support, the same argument would require non-asymptotic (in $1\le k\le n$) control  of the moments or tail probabilities of the tilted distribution $\mu_{n,k}$. A possible sufficient condition would be a tail assumption on $\mu$ ensuring that, after this tilting, the contribution from large values of $\theta$ (or values of $\theta$ close to $-\alpha$) remains negligible. Establishing such bounds appears to be the main technical difficulty in relaxing the bounded support condition, and we leave it for future work. 

\bibliographystyle{alpha}
\bibliography{reference}


\newpage

\appendix

\section{Basic lemma}
\subsection{Definition of tilted measure}\label{sec:tilted_measure}
Let $\mu$ be the mixing distribution appearing in \Cref{assumption}. For any $1\le k\le n$, define the tilted measure $\dd \mu_{n,k}$ as 
\begin{align}\label{eq:tilted_measure}
 1\le \forall k\le \forall n, \quad 
    \dd \mu_{n,k}(\theta) \equiv \frac{\dd \mu(\theta) v_{n,k}(\alpha, \theta)}{\int \dd\mu(\theta') v_{n,k}(\alpha,\theta')} \quad \text{with} \quad v_{n,k}(\alpha;\theta) = \frac{\prod_{i=1}^{k-1} (\theta + i\alpha)}{\prod_{i=1}^{n-1}(\theta + i)}.
\end{align}
Since the weight $v_{n,k}(\alpha;\theta)$ is strictly positive for all $\theta \in (-\alpha, +\infty)$, the support of the tilted probability measure $\mu_{n,k}$ is equal to the support of $\mu$. Combined with the condition on the support of $\mu$ from \Cref{assumption}, we have 
\begin{align}\label{eq:supp_v_nk}
1\le \forall k \le \forall n, \quad 
\operatorname{supp}(\mu_{n,k}) = \operatorname{supp} (\mu) \subset [\underline{\theta}, \bar\theta] \quad \text{with} -\alpha < \underline{\theta}\le 0 \le \bar\theta <+\infty.  
\end{align}
As we will see later, we will handle integrals of this form $\int \dd\mu_{n,k} (\theta) h(\theta)$ for some function $h:(-\alpha, \infty) \to \mathbb{R}$ many times. By \eqref{eq:supp_v_nk}, these integrals can be estimated as 
\begin{align}\label{eq:integral_controll}
\inf_{\theta\in [\underline{\theta}, \bar\theta]} h(\theta) \le \int \dd\mu_{n,k}(\theta) h(\theta) \le \sup_{\theta\in [\underline{\theta}, \bar\theta]} h(\theta) \quad \text{and} \quad \Bigl|\int \dd\mu_{n,k}(\theta) h(\theta)\Bigr|\le \sup_{\theta\in [\underline{\theta}, \bar\theta]} |h(\theta)|
\end{align}
We will frequently use this property.

\begin{lemma}[Concentration of triangular array $v_{n,k}$]\label{lm:concentrate_v_nk}
  Suppose the triangular array $v_{n,k}$ satisfies 
  \Cref{assumption}. Then, for any $n\in \mathbb{N}$ and  $1\le k\le n$, we have 
  \begin{align*}
    \Bigl|n \frac{v_{n+1, k}}{v_{n, k}} - 1\Bigr| \vee  \Bigl|\frac{v_{n+1, k+1}}{v_{n, k}} - \frac{k}{n}\alpha\Bigr|  \le \frac{\bar{\theta}-\underline{\theta}}{n+\underline\theta}. 
  \end{align*}
\end{lemma}
\begin{proof}
  By \Cref{assumption} and the definition of tilted probability measure \eqref{eq:tilted_measure}, we have 
  \begin{align*}
    \frac{v_{n+1, k}}{v_{n, k}} &= \frac{\int \dd\mu(\theta) \frac{\prod_{i=1}^{k-1} (\theta + i\alpha)}{\prod_{i=1}^{n-1}(\theta + i)}\frac{1}{n+\theta}}{\int\dd\mu(\theta)  \frac{\prod_{i=1}^{k-1} (\theta + i\alpha)}{\prod_{i=1}^{n-1}(\theta + i)}} =  \int \dd \mu_{n,k}(\theta) \frac{1}{\theta+n}= \frac{1}{n} + {\int \dd\mu_{n, k}(\theta) \bigl(\frac{1}{\theta+n}-\frac{1}{n}\bigr)}, \\
    \frac{v_{n+1, k+1}}{v_{n, k}} &= \frac{\int \dd\mu(\theta) \frac{\prod_{i=1}^{k-1} (\theta + i\alpha)}{\prod_{i=1}^{n-1}(\theta + i)}\frac{\theta+k\alpha}{n+\theta}}{\int\dd\mu(\theta)  \frac{\prod_{i=1}^{k-1} (\theta + i\alpha)}{\prod_{i=1}^{n-1}(\theta + i)}}  =  \int \dd\mu_{n,k}(\theta) \frac{\theta+k\alpha}{\theta + n} = \frac{k}{n}\alpha + {\int \dd\mu_{n, k} \bigl(\frac{\theta+k \alpha}{\theta+n}-\frac{k \alpha}{n}\bigr)}.
  \end{align*}
  By the property in \eqref{eq:integral_controll}, it holds that
  \begin{align*}
    \Bigl|\frac{v_{n+1, k}}{v_{n, k}} - \frac{1}{n}\Bigr| &\le 
  \sup_{\theta\in [\underline{\theta}, \bar{\theta}]} |\frac{1}{\theta+n}-\frac{1}{n}| =   \sup_{\theta\in [\underline{\theta}, \bar{\theta}]}  \Bigl|\frac{\theta}{n(n+\theta)}\Bigr|\le 
   \frac{1}{n} 
   \sup_{\theta\in [\underline{\theta}, \bar{\theta}]} \frac{|\theta|}{\theta+n}
   \\
  \Bigl|\frac{v_{n+1, k+1}}{v_{n, k}}-\frac{k}{n}\alpha\Bigr| &\le   
  \sup_{\theta\in [\underline{\theta}, \bar{\theta}]} \Bigl|\frac{\theta+k \alpha}{\theta+n}-\frac{k\alpha}{n}\Bigr| = 
    \sup_{\theta\in [\underline{\theta}, \bar{\theta}]} \Bigl|
      \frac{(n- k \alpha)\theta}{n(\theta+n)} 
    \Bigr|
    \le \sup_{\theta\in [\underline{\theta}, \bar{\theta}]} \frac{|\theta|}{\theta+n}
  \end{align*}
With $\sup_{\theta\in [\underline{\theta}, \bar{\theta}]} \frac{|\theta|}{\theta+n} \le 
  \frac{\bar\theta-\underline\theta}{n+\underline{\theta}}$, we complete the proof.  
\end{proof}

\subsection{Dominated convergence theorem for empirical measure $\sum_{j=1}^n \frac{\mathsf{k}_{n,j}}{\mathsf{k}_n} \delta_j$}

Given a partition $(U_1, U_2, \dots )$ of $[n]$, let $\mathsf{k}_{n,j}$ be the number of blocks of size $j$ and let $\mathsf{k}_n = \sum_{j=1}^n \mathsf{k}_{n,j}$ be the number of non-empty blocks.
Since $\mathsf{k}_{n,j}=0$ for any $j\ge n+1$, this induces the following empirical measure over $\mathbb{N}$:
$$
   \sum_{j=1}^\infty  \frac{\mathsf{k}_{n,j}}{\mathsf{k}_n} \delta_j.
$$
By equation \eqref{eq:intro_powerlaw}, this empirical measure converges to the discrete distribution $\mathfrak{p}_\alpha(j)$ pointwise, i.e.,
$$
\frac{\mathsf{k}_{n,j}}{\mathsf{k}_n} \asconv \mathfrak{p}_\alpha(j) = \frac{\alpha\prod_{i=1}^{j-1}(i-\alpha)}{j!}.
$$
During proofs in later sections, we will see many terms of the integral $\sum_{j=1}^\infty \frac{\mathsf{k}_{n,j}}{\mathsf{k}_n} f(j)$ for some function $f:\mathbb{N}\to\mathbb{R}$. The next lemma claims a dominated convergence-type theorem, i.e., if the test function $f$ is bounded, then the integral also converges almost surely. 
\begin{lemma}[Dominated convergence theorem]\label{lm:bounded_conv}
  As $n\to+\infty$, it holds that 
  $$\sum_{j=1}^\infty \Bigl|\frac{\mathsf{k}_{n,j}}{\mathsf{k}_n} - \mathfrak{p}_\alpha(j)\Bigr|\asconv 0$$
  Therefore, for any bounded function $f:\mathbb{N}\to\mathbb{R}$, we have
  $$
  \sum_{j=1}^n \frac{\mathsf{k}_{n,j}}{\mathsf{k}_n} f(j) \asconv \sum_{j=1}^\infty \mathfrak{p}_\alpha(j) f(j). 
  $$
\end{lemma}

\begin{proof}
This result follows from Scheff\'e's Lemma applied with 
 the pointwise convergence $\mathsf{k}_{n,j}/\mathsf{k}_n \asconv \mathfrak{p}_\alpha(j)$ (a.s.) and the fact that $(\mathsf{k}_{n,j}/\mathsf{k}_n)_{j=1}^\infty$ and $(\mathfrak{p}_\alpha(j))_{j=1}^\infty$ are probability measures on $\mathbb{N}$. 
\end{proof}

\section{$\alpha$-diversity and Fisher Information}

\subsection{Proof of \Cref{thm:Lp_converge}}
The proof in this section is inspired by \cite{bercu2024martingale}, which focuses on the case $\mu=\delta_\theta$, 
 and we borrow several notations from the paper. 
Let $\mathcal{F}_n$ be the $\sigma$-field generated by the random partition of $[n]$, i.e., $(U_1, \dots \cdots U_{\mathsf{k}_n})\in \mathcal{P}_n$. 
By the definition, the number of nonempty blocks $\mathsf{k}_n$ can be written as a summation of Bernoulli random variables as follows:
\begin{align*}
  \mathsf{k}_{n+1} = \mathsf{k}_n + \xi_{n+1}, \quad \xi_{n+1}|\mathcal{F}_n\sim \Bernoulli(p_n), \quad   p_n\equiv \frac{v_{n+1, \mathsf{k}_{n}+1}}{v_{n, \mathsf{k}_n}}.
\end{align*}
By \Cref{assumption}, the success probability $p_n$ can be expressed as an integral of tilted measure:
\begin{align*}
   p_n = \frac{\int\dd \mu(\theta) v_{n,\mathsf{k}_n}(\alpha, \theta) \frac{\theta+\mathsf{k}_n\alpha}{\theta + n}}{\int \dd \mu(\theta) v_{n, \mathsf{k}_n}(\alpha,\theta)} =\int \dd \mu_{n, \mathsf{k}_n}(\theta) \frac{\theta+\mathsf{k}_n\alpha}{\theta + n}. 
\end{align*}
Here $\dd \mu_{n,\mathsf{k}_n}$ is the tilted probability measure \eqref{eq:tilted_measure} with $k=\mathsf{k}_n$.  
Here, the integrand $(\theta+\mathsf{k}_n\alpha)/(\theta+n)$ is increasing in $\theta\in (-\alpha, +\infty)$ by $\alpha \mathsf{k}_n < 1 \cdot n$. Thus, using the property of \eqref{eq:integral_controll}, we get 
\begin{align}\label{eq:p_n_inequality}
 \underline{p}_n \equiv \frac{\underline\theta + \alpha \mathsf{k}_n}{\underline\theta + n}  \le p_n \le  \frac{\bar\theta + \alpha \mathsf{k}_n}{\bar\theta + n} =: \bar{p}_n
\end{align}
Then, using the estimate $p_n\le \bar{p}_n$, 
the conditional expectation $  \E[\mathsf{k}_{n+1}|\mathcal{F}_n]$ is bounded from above as 
\begin{align*}
  \E[\mathsf{k}_{n+1}|\mathcal{F}_n] 
  &=\mathsf{k}_n + p_n &&\mathsf{k}_{n+1}=\mathsf{k}_n + \xi_{n+1}, \ \xi_{n+1}|\mathcal{F}_n\sim \Bernoulli(p_n)\\
  &\le  \mathsf{k}_n + \frac{\bar{\theta}+\mathsf{k}_n\alpha}{\bar{\theta}+n} && p_n \le \bar{p}_n=\frac{\bar\theta + \alpha \mathsf{k}_n}{\bar\theta + n}. 
\end{align*}
Rearranging the above display, we get 
\begin{align*}
  \E\Bigl[\mathsf{k}_{n+1} + \frac{\bar{\theta}}{\alpha}\mid \mathcal{F}_n\Bigr] &\le\Bigl(\mathsf{k}_n + \frac{\bar{\theta}}{\alpha}\Bigr) \Bigl(
    1 + \frac{\alpha}{\bar{\theta}+n}
  \Bigr). 
\end{align*}
This means that $\bar M_n$ defined by 
\begin{align}\label{eq:M_n_upper}
  \bar M_n  \equiv \bar b_n(\mathsf{k}_n+\frac{\bar\theta}\alpha) \quad \text{with} \quad \bar b_n \equiv \prod_{k=1}^{n-1} \Bigl(
      1 + \frac{\alpha}{\bar{\theta}+k} 
    \Bigr)^{-1} = \prod_{k=1}^{n-1} \Bigl(
  \frac{\bar{\theta}+k}{\bar{\theta}+k + \alpha}
    \Bigr)
\end{align}
is a non-negative super martingale, i.e., 
$$
\bar M_{n} \ge 0, \quad 
\E[\bar M_{n+1} |\mathcal{F}_n] \le \bar{M}_n.
$$ 
Then, by Doob's martingale convergence theorem (cf. \cite[Theorem 4.2.12]{DU19}), there exists a positive random variable $\bar M_\infty$ such that 
$$
\bar M_n \asconv \bar M_\infty, \quad  \bar M_\infty\in L_1. 
$$
With $\Gamma(n+\alpha)/\Gamma(n)\sim n^\alpha$ for any $\alpha\in \mathbb{R}$, the coefficient $\bar{b}_n$ appearing in \eqref{eq:M_n_upper} behave asymptotically as 
\begin{align}\label{eq:b_upper_asym}
\bar{b}_n = \frac{\Gamma(\bar\theta+\alpha+1)}{\Gamma(\bar\theta+1)} \frac{\Gamma(\bar\theta+n)}{\Gamma(\bar\theta+\alpha+n)}  \sim n^{-\alpha} \frac{\Gamma(\bar\theta + \alpha+1)}{\Gamma(\bar\theta+1)}.   
\end{align}
Combined with the definition $\bar{M}_n = \bar b_n (\mathsf{k}_n + \bar\theta/\alpha)$ and $\bar{M}_n \asconv\bar{M}_\infty\in L_1$, we get 
\begin{align}\label{eq:Kn_as_conv}
n^{-\alpha} \mathsf{k}_n  \asconv \diversity \equiv \frac{\Gamma(\bar{\theta}+1)}{\Gamma(\bar{\theta}+\alpha+1)} \bar M_\infty\in L_1. 
\end{align}
Furthermore, since $\bar M_n$ is a super martingale with $\bar M_1 = 1+\frac{\bar\theta}{\alpha}$, it holds that 
\begin{align*}
  b_n \E[\mathsf{k}_n+\frac{\bar\theta}{\alpha}] = \E[\bar M_{n}] \le \E[\bar M_1] = 1+\frac{\bar\theta}{\alpha}. 
\end{align*}
Combined with \eqref{eq:b_upper_asym}, we obtain 
the upper bound of $\E[\mathsf{k}_n]$:
\begin{align}\label{eq:Kn_moment_upper_bound}
\E[\mathsf{k}_n] &\le \bar{b}_n^{-1} \Bigl(1+\frac{\bar\theta}{\alpha}\Bigr) - \frac{\bar\theta}{\alpha} \sim n^\alpha \frac{\Gamma(\bar\theta+1)} {\Gamma(\bar\theta+\alpha+1)}\Bigl(1+\frac{\bar\theta}{\alpha}\Bigr)  = n^\alpha \frac{\Gamma(\bar\theta+1)}{\alpha \Gamma(\bar\theta+\alpha)}.
\end{align}

Next, we will show $n^{-\alpha} \mathsf{k}_n\to\diversity $ in $L^2$ (stronger than \eqref{eq:Kn_as_conv}) by constructing a sub-martingale and applying Doob's maximum $L_p$ inequality. Now, similarly to \eqref{eq:M_n_upper}, we define $\underline{M}_n$ as 
$$
\underline{M}_n = \underline b_n(\mathsf{k}_n+\frac{\underline\theta}\alpha),  \quad  \underline b_n \equiv  \prod_{k=1}^{n-1} \Bigl(
\frac{\underline{\theta}+k}{\underline{\theta}+k + \alpha}
  \Bigr) = \frac{\Gamma(\underline\theta+\alpha+1)}{\Gamma(\underline\theta+1)} \frac{\Gamma(\underline\theta+n)}{\Gamma(\underline\theta+\alpha+n)}
$$
Then, by the same argument for $\bar{M}_n$, now using the lower estimate $\underline{p}_n\le p_n$ in \eqref{eq:p_n_inequality}, 
 one can show that  $\underline{M}_n$ is non-negative sub martingale, i.e., 
$$
\underline{M}_n\ge 0, \quad 
\E[\underline{M}_{n+1}|\mathcal{F}_n] \ge \underline M_n. 
$$
Furthermore, by the same argument for $\bar{b}_n$ in \eqref{eq:b_upper_asym}, the leading term of $\underbar{b}_n$ is given by 
\begin{align}\label{eq:b_lower_asym}
\underline{b}_n \sim 
n^{-\alpha} \frac{\Gamma(\underline\theta + \alpha+1)}{\Gamma(\underline \theta+1)}. 
\end{align}
By the same argument in \eqref{eq:Kn_moment_upper_bound}, now using the sub-martingale property of $\bar{M}_n$, we also get the lower bound of $\E[\mathsf{k}_n]$:
\begin{align}\label{eq:Kn_moment_lower_bound}
\E[\mathsf{k}_n] &\gtrsim n^\alpha \frac{\Gamma(\underline\theta+1)}{\alpha \Gamma(\underline\theta+\alpha)}.
\end{align}
Combining  \eqref{eq:Kn_as_conv},  \eqref{eq:b_lower_asym}, and $\underline{M}_n = \underline b_n(\mathsf{k}_n+\frac{\underline\theta}\alpha)$, we obtain
\begin{align*}
  \underline{M}_n \sim  \frac{\Gamma(\underline\theta + \alpha+1)}{\Gamma(\underline \theta+1)}  n^{-\alpha} \Bigl(\mathsf{k}_n + \frac{\underline{\theta}}{\alpha}\Bigr)  \asconv \underline{M}_\infty \equiv \frac{\Gamma(\underline\theta + \alpha+1)}{\Gamma(\underline \theta+1)} \diversity \in L_1. 
\end{align*}
Next, we will show that the non-negative sub martingale $\underline{M}_n$ is bounded in $L_2$, i.e., $\sup_{n\ge 1} \|\underline{M}_n\|_2 < +\infty$.  Note that the increment of $\bar{M}_n$ can be written as 
\begin{align*}
  \underline{M}_{n+1} - \underline{M}_n &= \underline{b}_{n+1} (\mathsf{k}_{n+1} + \frac{\underline \theta}{\alpha}) - \underline{b}_{n} (\editline{\mathsf{k}_n} + \frac{\underline \theta}{\alpha})\\
  &= \underline{b}_{n+1} (\editline{\mathsf{k}_n} + \xi_{n+1} + \frac{\underline \theta}{\alpha}) - \underline{b}_{n} (\editline{\mathsf{k}_n} + \frac{\underline \theta}{\alpha})\\
  &= \underline{b}_{n+1} \Bigl(
  \xi_{n+1} - (\frac{\underline{b}_n}{\underline{b}_{n+1}}-1) (\mathsf{k}_n + \frac{\underline{\theta}}{\alpha}) 
  \Bigr)\\
  &= \underline{b}_{n+1} \Bigl(
  \xi_{n+1} - \frac{\alpha}{n+\underline{\theta}} (\mathsf{k}_n + \frac{\underline{\theta}}{\alpha}) 
  \Bigr) && \underline b_n = \prod_{k=1}^{n-1} \Bigl(
\frac{\underline{\theta}+k}{\underline{\theta}+k + \alpha}
  \Bigr)  \\
  &= \underline{b}_{n+1} \Bigl(
  \xi_{n+1} - \underline{p}_{n}
  \Bigr) && \underline{p} = \frac{\alpha \mathsf{k}_n + \underline{\theta}}{n+\underline{\theta}}. 
\end{align*}
Then, we have 
\begin{align}
  \E[(\underline M_{n}-\underline M_1)^2] &= \E\Bigl[\Bigl(\sum_{i=1}^{n-1}(\underline M_{i+1}-\underline M_{i})\Bigr)^2\Bigr]\nonumber \\
  &= \sum_{k=1}^{n-1} \underline{b}_{k+1}^2 \E[(
  \xi_{k+1} - \underline{p}_{k}
)^2]  +  2 \sum_{1\le k<\ell\le n-1} \underline{b}_{k+1} \underline{b}_{\ell+1} \E\bigl[(
  \xi_{k+1} - \underline{p}_{k}
  )  (
  \xi_{\ell+1} - \underline{p}_{\ell}
  )  \bigr]\label{eq:quadratic_variation_Mn}
\end{align}
Now we claim $\sup_{n\ge 1}\E[(M_{n}-M_1)^2] <+\infty$. First, let us derive an upper bound of the first term. By the tower property, as $k\to+\infty$, 
\begin{align*}
  \E[(\xi_{k+1}-\underline{p}_k)^2] &=   \E[\E[(\xi_{k+1}-\underline{p}_k)^2|\mathcal{F}_k]]\\
  &= \E[p_k(1-\underline{p}_k)^2 + (1-p_k)\underline{p}_k^2] && \xi_{k+1}|\mathcal{F}_k \sim \Bernoulli(p_k)\\
  &\le 2 \E[\bar p_k]  && 0\le \underline{p}_k \le p_k \le \bar{p}_k \le 1  \\
  &= 2 \frac{\bar\theta  + \alpha \E[\mathsf{k}_k]}{k+\bar \theta}\\
  &\lesssim k^{-1+\alpha} && \text{$\E[\mathsf{k}_n]\lesssim n^\alpha$ from \eqref{eq:Kn_moment_upper_bound}}
\end{align*}
Combined with $\underline{b}_n \asymp n^{-\alpha}$ from \eqref{eq:b_lower_asym}, we have that, as $n\to+\infty$, 
\begin{align}
  \sum_{k=1}^{n-1} \underline{b}_{k+1}^2 \E\bigl[(
  \xi_{k+1} - \underline{p}_{k}
)^2 \bigr] \lesssim \sum_{k=1}^{n-1} k^{-2\alpha} k^{-1+\alpha} = \sum_{k=1}^{n-1} k^{-(1+\alpha)} = O(1). \label{eq:first_O1}
\end{align}
Next, let us bound the cross term $\underline{b}_{k+1} \underline{b}_{\ell+1} \E[(
  \xi_{k+1} - \underline{p}_{k}
)(
  \xi_{\ell+1} - \underline{p}_{\ell}
)]$. By $\underline p_n \le p_n \le \bar p_n$ and $\E[\xi_{n+1}|\mathcal{F}_n]=p_n$, 
  \begin{align*}
    \E[\xi_{n+1}-\underline{p}_n|\mathcal{F}_n] = \E[p_n - \underline{p}_n] \ge 0
  \end{align*}
  and
  \begin{align*}
p_n - \underline{p}_n 
&\le \bar p_n - \underline{p}_n \\
&= \frac{\bar\theta + \alpha \mathsf{k}_n}{\bar\theta + n} - \frac{\underline\theta + \alpha \mathsf{k}_n}{\underline \theta + n}\\
&= \frac{\bar\theta-\underline\theta}{(\bar\theta + n)(\underline{\theta}+n)}(n-\alpha \mathsf{k}_n)\\
&\le \frac{\bar\theta-\underline\theta}{(\underline{\theta}+n)} && 0\le \alpha \mathsf{k}_n < n,\  \bar\theta\ge 0
\end{align*}
for all $n$. 
Let us fix $k < \ell$. With $k+1\le \ell$, using the tower property, 
\begin{align*}
\E\bigl[(
  \xi_{k+1} - \underline{p}_{k}
  )(
  \xi_{\ell+1} - \underline{p}_{\ell}
  )|\mathcal{F}_{k+1}\bigr]
  &= (
  \xi_{k+1} - \underline{p}_{k}
  )\E\Bigl[(
  \xi_{\ell+1} - \underline{p}_{\ell}
  )|\mathcal{F}_{k+1}\Bigr] \\
  &= (
  \xi_{k+1} - \underline{p}_{k}
  )\E\Bigl[\E\bigl[(
  \xi_{\ell+1} - \underline{p}_{\ell}
  )|\mathcal{F}_{\ell}\bigr]|\mathcal{F}_{k+1}\Bigr]  \\
  &=  (
  \xi_{k+1} - \underline{p}_{k}
  )\E\bigl[(p_{\ell} - \underline{p}_{\ell}
  )|\mathcal{F}_{k+1}\bigr]   
\end{align*}
so 
\begin{align*}
  \Bigl|\E\bigl[(
  \xi_{k+1} - \underline{p}_{k}
  )(
  \xi_{\ell+1} - \underline{p}_{\ell}
  )|\mathcal{F}_{k+1}\bigr] \Bigr|
  &\le  \bigl|
  \xi_{k+1} - \underline{p}_{k}
  \bigr|\E\bigl[|p_{\ell} - \underline{p}_{\ell}
  | \mid \mathcal{F}_{k+1}\bigr] \le  \bigl|
  \xi_{k+1} - \underline{p}_{k}
  \bigr| \frac{\bar\theta-\underline\theta}{\underline{\theta}+\ell}
\end{align*}
This gives 
\begin{align*}
  \left|\E\bigl[(
  \xi_{k+1} - \underline{p}_{k}
  )(
  \xi_{\ell+1} - \underline{p}_{\ell}
  )\bigr]\right| &= \left|\E\Bigl[\E\bigl[(
  \xi_{k+1} - \underline{p}_{k}
  )(
  \xi_{\ell+1} - \underline{p}_{\ell}
  )|\mathcal{F}_{k+1}\bigr]\Bigr]\right|\\
  &\le  \E\Bigl[\Bigl|\E[(
  \xi_{k+1} - \underline{p}_{k}
  )(
  \xi_{\ell+1} - \underline{p}_{\ell}
  )|\mathcal{F}_{k+1}]\Bigr|\Bigr]\\
  &\le \E[|
  \xi_{k+1} - \underline{p}_{k}
  |] \cdot \frac{\bar\theta-\underline\theta}{\underline{\theta}+\ell}
\end{align*}
where 
\begin{align*}
   \E[|
  \xi_{k+1} - \underline{p}_{k}
  |] &=  \E[\E[|
  \xi_{k+1} - \underline{p}_{k}
  |\mid \mathcal{F}_k]]\\
  &=\E[p_k (1-\underline{p}_k) + (1-p_k) \underline{p}_k]\\
  &\le 2 \E[\underline{p}_k]\\
  &= 2 \frac{\bar\theta  + \alpha \E[\mathsf{k}_k]}{k+\bar \theta}.
\end{align*}
Therefore, the cross term is upper bounded as 
\begin{align*}
  \left|\sum_{1\le k<\ell\le n-1}\underline{b}_{k+1} \underline{b}_{\ell+1} \E[(
  \xi_{k+1} - \underline{p}_{k})(
  \xi_{\ell+1} - \underline{p}_{\ell}
 )]\right|
 &\le   \sum_{1\le k<\ell\le n-1}\underline{b}_{k+1} \underline{b}_{\ell+1} \Bigl|\E[(
  \xi_{k+1} - \underline{p}_{k})(
  \xi_{\ell+1} - \underline{p}_{\ell}
 )]\Bigr| \\
 &\le \sum_{1\le k<\ell\le n-1}\underline{b}_{k+1} \underline{b}_{\ell+1}  \frac{\bar\theta-\underline\theta}{\underline{\theta}+\ell} \cdot 2 \frac{\bar\theta  + \alpha \E[\mathsf{k}_k]}{k+\bar \theta}\\
 &=\sum_{\ell=2}^{n-1} \underline{b}_{\ell+1} \frac{\bar\theta-\underline{\theta}}{\ell+\underline{\theta}} \cdot \sum_{k=1}^{\ell-1} 2 \underline{b}_{k+1}  \cdot \frac{\bar\theta  + \alpha \E[\mathsf{k}_k]}{k+\bar \theta}
\end{align*}
Here, as $k\to+\infty$, using $\underline{b}_k\asymp k^{-\alpha}$ and $\E[\mathsf{k}_k]\lesssim k^\alpha$
$$
\underline{b}_{k+1}  \cdot \frac{\bar\theta  + \alpha \E[\mathsf{k}_k]}{k+\bar \theta} \lesssim k^{-\alpha} \frac{k^\alpha}{k}  = \frac{1}{k}, 
$$
so that  as $\ell\to+\infty$, 
$$
\sum_{k=1}^{\ell-1} 2 \underline{b}_{k+1}  \cdot \frac{\bar\theta  + \alpha \E[\mathsf{k}_k]}{k+\bar \theta} \lesssim \log \ell. 
$$
This in turn gives 
$$
\underline{b}_{\ell+1} \frac{\bar\theta-\underline{\theta}}{\ell+\underline{\theta}} \cdot \sum_{k=1}^{\ell-1} 2 \underline{b}_{k+1}  \cdot \frac{\bar\theta  + \alpha \E[\mathsf{k}_k]}{k+\bar \theta} \lesssim \ell^{-\alpha} \frac{1}{\ell} \log \ell = \frac{\log \ell}{\ell^{\alpha+1}}
$$
as $\ell\to+\infty$. Therefore, we have that, as $n\to+\infty$, 
\begin{align}
 \Bigl|\sum_{1\le k<\ell\le n-1}\underline{b}_{k+1} \underline{b}_{\ell+1} \E[(
  \xi_{k+1} - \underline{p}_{k})(
  \xi_{\ell+1} - \underline{p}_{\ell}
 )]\Bigr| 
 \lesssim \sum_{\ell=2}^{n-1}  \frac{\log \ell}{\ell^{\alpha+1}} = O(1) \label{eq:second_O1}
\end{align}
where we have used the fact that $\int_1^\infty x^{-(\alpha+1)} \log x$ is finite. 

Putting \eqref{eq:quadratic_variation_Mn}, \eqref{eq:first_O1}, and \eqref{eq:second_O1} together, we obtain 
$$
\sup_{n\ge 1}  \E[(\underline M_{n}-\underline M_1)^2] <+\infty. 
$$
With $\underline M_1=1+\frac{\underline\theta}{\alpha}$, which is a deterministic constant, we conclude that $\underline M_n$ is bounded in $L_2$. Since $\underline{M}_n$ is non-negative sub martingale, Doob's maximum $L_p$ inequality with $p=2$ (cf. \cite[Theorem 4.4.4]{DU19}) yields
\begin{align*}
  \E\Bigl[\bigl(\sup_{n\ge 1} |\underline{M}_n|\bigr)^2\Bigr] \le \Bigl(\frac{2}{2-1}\Bigr)^{2} \sup_{n\ge 1} \E[|\underline{M}_n|^2] < +\infty 
\end{align*}
On the other hand, since $\underline{M}_n$ converges to $\underline{M}_\infty$ almost surely as $n\to+\infty$, we have 
\begin{align*}
  \sup_{n\ge 1} |\underline{M}_n - \underline{M}_\infty|^2  \le \bigl(2 \sup_{n\ge 1} |\underline{M}_n|\bigr)^2.
\end{align*}
Therefore, the dominated convergence theorem gives 
$$
 \lim_{n\to+\infty} \E[|\underline{M}_n - \underline{M}_\infty|^2] =  \E[\lim_{n\to+\infty}|\underline{M}_n - \underline{M}_\infty|^2] = 0. 
$$
This concludes that $\underline{M}_n$ converges to $\underline{M}_\infty$ in $L_2$.
Moreover, since $\underline{M}_n$ is bounded in $L_2$, we also get $\underline{M}_\infty \in L_2$. 
Finally, combining 
$$
\underline{M}_n = \underline{b}_n (\mathsf{k}_n + \frac{\underline{\theta}}{\alpha}), \quad \underline{b}_n \sim n^{-\alpha} \frac{\Gamma(\underline\theta + \alpha+1)}{\Gamma(\underline \theta+1)}, \quad \underline{M}_\infty = \frac{\Gamma(\underline\theta + \alpha+1)}{\Gamma(\underline \theta+1)} \diversity, 
$$
and noting that $\underline{b}_n$ is deterministic, we conclude that $\diversity\in L_2$ and 
$n^{-\alpha} \mathsf{k}_n \to^{L^2} \diversity$.

The inequality \eqref{eq:alpha_diversity_moment_ineq} of $\E[\diversity]$  immediately follows 
from \eqref{eq:Kn_moment_upper_bound}, \eqref{eq:Kn_moment_lower_bound}, and the $L_1$ convergence $n^{-\alpha} \mathsf{k}_n \to^{L^1}\diversity$.

{
\subsection{Proof of \Cref{theorem:fisher_info_sibuya}}\label{sec:proof_fisher_info_sibuya}

We first recall the following expression from \cite{koriyama2022asymptotic}. 
\begin{lemma}[{\cite[Proposition 3.1]{koriyama2022asymptotic}}]\label{lemma:fisher_info_sibuya}
  Let $\mathfrak{i}(\alpha)$ be the Fisher information of the Sibuya distribution $\mathfrak{p}_\alpha$ defined in \eqref{eq:fisher_sibuya}. Then, for every $\alpha\in(0,1)$,
  $$
  \mathfrak{i}(\alpha)
  =
  \frac{1}{\alpha^2}
  +
  \sum_{j=1}^\infty \mathfrak{p}_\alpha(j)
  \sum_{i=1}^{j-1}\frac{1}{(i-\alpha)^2}
  =
  \frac{1}{\alpha^2}
  +
  \sum_{j=1}^\infty
  \frac{\mathfrak{p}_\alpha(j)}{\alpha(j-\alpha)}.
  $$
\end{lemma}
Note that the both of the infinite series representations will be used for the proof of other theorems later, but for now, let us derive the \Cref{theorem:fisher_info_sibuya} from the second representation. 
Since the probability generating function of the Sibuya distribution is given by
$$
\sum_{j=1}^\infty \mathfrak{p}_\alpha(j)x^j
=
1-(1-x)^\alpha,
\qquad x\in[0,1],
$$
using $
(j-\alpha)^{-1}=\int_0^1 x^{j-\alpha-1}dx
$ for $j\ge 1$ and 
Fubini's theorem, we have 
$$
\sum_{j=1}^\infty
\frac{\mathfrak{p}_\alpha(j)}{j-\alpha} 
= \sum_{j=1}^\infty \mathfrak{p}_\alpha(j)\int_0^1 x^{j-\alpha-1} dx = 
\int_0^1 x^{-\alpha-1}\{1-(1-x)^\alpha\}dx.
$$
By integration by parts,
$$
\int_0^1 x^{-\alpha-1}\{1-(1-x)^\alpha\}dx
=
\left[
-\frac{x^{-\alpha}}{\alpha}\{1-(1-x)^\alpha\}
\right]_0^1
+
\int_0^1 x^{-\alpha}(1-x)^{\alpha-1}dx
=
-\frac1\alpha+B(1-\alpha,\alpha).
$$
where $B(\cdot, \cdot)$ is the Beta function. 
Therefore, combined with the second formula in \Cref{lemma:fisher_info_sibuya}, 
$$
\mathfrak{i}(\alpha)
=
\frac{1}{\alpha^2}
+
\frac1\alpha
\left\{
-\frac1\alpha+B(1-\alpha,\alpha)
\right\}
= \frac{B(1-\alpha, \alpha)}{\alpha}. 
$$
Finally, using $B(1-\alpha,\alpha)=\Gamma(\alpha)\Gamma(1-\alpha)$ for the Gamma function $\Gamma$ and Euler's reflection formula $\Gamma(\alpha)\Gamma(1-\alpha)={\pi}/{\sin(\pi\alpha)}$, we obtain
$$
\mathfrak{i}(\alpha)
=
\frac{\pi}{\alpha\sin(\pi\alpha)},
$$
so the proof of \Cref{theorem:fisher_info_sibuya} is complete. 
}

\subsection{Proof of \Cref{theorem:fisher}}\label{sec:proof_fisher}
Let $\ell_n(\alpha;\mu)$ be the log-likelihood given by \eqref{eq:loglikelihood}. Since the marginal distribution of the Gibbs partition is a discrete distribution over $\mathcal{P}_n$, the expectation $\E$ over the law and derivative with respect to $\alpha$ is interchangeable. As a result, it holds that 
 \begin{align*}
  \E[(\partial_\alpha \ell_n(\alpha;\mu)^2)] =
   - \E[\partial_\alpha^2 \ell_n(\alpha;\mu)],
\end{align*}
and hence it suffices to show $n^{-\alpha} \E[\partial_\alpha^2 \ell_n(\alpha;\mu)] \to -\E[\diversity]\mathfrak{i}(\alpha)$. 

Here, the first and the second derivative of the log-likelihood function are given by 
\begin{align*}
  \partial_\alpha \ell_n(\alpha;\mu) &= \frac{\partial}{\partial \alpha} \left(\log \Bigl(\int \dd \mu(\theta) \frac{\prod_{i=1}^{\mathsf{k}_n-1} (\theta + i\alpha)}{\prod_{i=1}^{n-1}(\theta + i)} \Bigr)+ \sum_{j=2}^{n} \mathsf{k}_{n,j} \sum_{i=1}^{j-1} \log(i-\alpha)\right)\\
  &=  \frac{\frac{\partial}{\partial\alpha} \int \dd \mu(\theta) \frac{\prod_{i=1}^{\mathsf{k}_n-1} (\theta + i\alpha)}{\prod_{i=1}^{n-1}(\theta + i)}}{\int \dd \mu(\theta) \frac{\prod_{i=1}^{\mathsf{k}_n-1} (\theta + i\alpha)}{\prod_{i=1}^{n-1}(\theta + i)}} - \sum_{j=2}^{n} \mathsf{k}_{n,j} \sum_{i=1}^{j-1} \frac{1}{i-\alpha}, \\
  \partial_\alpha^2 \ell_n(\alpha;\mu)
  &=  \frac{\frac{\partial^2}{\partial\alpha^2} \int \dd \mu(\theta) \frac{\prod_{i=1}^{\mathsf{k}_n-1} (\theta + i\alpha)}{\prod_{i=1}^{n-1}(\theta + i)}}{\int \dd \mu(\theta) \frac{\prod_{i=1}^{\mathsf{k}_n-1} (\theta + i\alpha)}{\prod_{i=1}^{n-1}(\theta + i)}} - \left(\frac{\frac{\partial}{\partial\alpha} \int \dd \mu(\theta) \frac{\prod_{i=1}^{\mathsf{k}_n-1} (\theta + i\alpha)}{\prod_{i=1}^{n-1}(\theta + i)}}{\int \dd \mu(\theta) \frac{\prod_{i=1}^{\mathsf{k}_n-1} (\theta + i\alpha)}{\prod_{i=1}^{n-1}(\theta + i)}}\right)^2 - \sum_{j=1}^n \mathsf{k}_{n,j} \sum_{i=1}^{j-1}\frac{1}{(i-\alpha)^2}, 
\end{align*}
Now we claim that the derivative $\partial_\alpha, \partial_\alpha^2$ and the integration $\int\dd\mu$ appearing above are interchangeable. This is because the first and second derivative of the integrand $\frac{\prod_{i=1}^{\mathsf{k}_n-1} (\theta + i\alpha)}{\prod_{i=1}^{n-1}(\theta + i)}$ with respect to $\alpha$ are given by 
  \begin{align*}
    \frac{\partial}{\partial \alpha} \Bigl(\frac{\prod_{i=1}^{\mathsf{k}_n-1} (\theta + i\alpha)}{\prod_{i=1}^{n-1}(\theta + i)}\Bigr) &= \frac{\prod_{i=1}^{\mathsf{k}_n-1} (\theta + i\alpha)}{\prod_{i=1}^{n-1}(\theta + i)} \sum_{i=1}^{\mathsf{k}_n-1} \frac{i}{\theta+i\alpha}\\
    \frac{\partial^2}{\partial \alpha^2} \Bigl(\frac{\prod_{i=1}^{\mathsf{k}_n-1} (\theta + i\alpha)}{\prod_{i=1}^{n-1}(\theta + i)}\Bigr) &= \frac{\prod_{i=1}^{\mathsf{k}_n-1} (\theta + i\alpha)}{\prod_{i=1}^{n-1}(\theta + i)} \Bigl(\sum_{i=1}^{\mathsf{k}_n-1} \frac{i}{\theta+i\alpha}\Bigr)^2  - \frac{\prod_{i=1}^{\mathsf{k}_n-1} (\theta + i\alpha)}{\prod_{i=1}^{n-1}(\theta + i)} \sum_{i=1}^{\mathsf{k}_n-1} \frac{i^2}{(\theta+i\alpha)^2}.
  \end{align*}
By \Cref{assumption} 
the support of $\mu$ is contained in $[\underline{\theta}, \bar{\theta}]$ with $-\alpha< \underline{\theta} \editline{ \le } \bar{\theta} <+\infty$, while the second derivative above is continuous in $\theta$ over the support $[\underline{\theta}, \bar\theta]$.  
Therefore, the derivative $\partial_\alpha, \partial_\alpha^2$ and the integration $\int\dd\mu$ are interchangeable. 

Substituting the derivative formula of $ \partial_\alpha^2$ to the previous display of $\partial_\alpha^2 \ell_n(\alpha; \mu)$, changing the order of $\partial_\alpha^2$ and $\int \dd\mu$, using the definition of tilted measure $\dd\mu_{n,\mathsf{k}_n}$, we get 
    \begin{align*}
    &\partial_\alpha^2 \ell_n(\alpha; \mu)\\
    &= \int \dd \mu_{n,\mathsf{k}_n}(\theta) \Bigl(\sum_{i=1}^{\mathsf{k}_n-1} \frac{i}{\theta+i\alpha}\Bigr)^2 - \int \dd\mu_{n,\mathsf{k}_n}(\theta) \sum_{i=1}^{\mathsf{k}_n-1} \frac{i^2}{(\theta+i\alpha)^2}
  - \Bigl(\int \dd\mu_{n,\mathsf{k}_n}(\theta) \sum_{i=1}^{\mathsf{k}_n-1} \frac{i}{\theta+i\alpha}\Bigr)^2\\
  &
   -  \sum_{j=1}^n \mathsf{k}_{n,j} \sum_{i=1}^{j-1} \frac{1}{(i-\alpha)^2}\\
   &= \mathcal{V}_{\theta\sim \mu_{n,\mathsf{k}_n}}\Bigl[\sum_{i=1}^{\mathsf{k}_n-1} \frac{i}{\theta + i\alpha}\Bigr] - \mathcal{E}_{\theta\sim \mu_{n,\mathsf{k}_n}}\Bigl[ \sum_{i=1}^{\mathsf{k}_n-1}\frac{i^2}{(\theta+i\alpha)^2} \Bigr] -  \sum_{j=1}^n \mathsf{k}_{n,j} \sum_{i=1}^{j-1} \frac{1}{(i-\alpha)^2},
  \end{align*}
  where we have defined the integral operator $\mathcal{E}$ and variance operator $\mathcal{V}$ under $\mu_{n, \mathsf{k}_n}$ as follows:
  $$
  \mathcal{E}_{\theta\sim \mu_{n,\mathsf{k}_n}} [f(\theta)] \equiv \int \dd\mu_{n}(\theta) f(\theta), \quad 
  \mathcal{V}_{\theta\sim \mu_{n,\mathsf{k}_n}} [f(\theta)] \equiv \mathcal{E}_{\theta\sim \mu_{n,\mathsf{k}_n}}[f(\theta)^2] - \Bigl(\mathcal{E}_{\theta\sim \mu_{n,\mathsf{k}_n}}[f(\theta)] \Bigr)^2. 
  $$
  In particular, if we take the Dirac measure $\mu=\delta_0$, noting that $\mu_{n,\mathsf{k}_n}=\delta_0$, we get 
\begin{align}\label{eq:ell_n_0_second}
    \partial_\alpha^2 \ell_n(\alpha; \delta_0) = 0  -\frac{\mathsf{k}_n-1}{\alpha^2} - \sum_{j=1}^n \mathsf{k}_{n,j} \sum_{i=1}^{j-1} \frac{1}{(i-\alpha)^2}. 
\end{align}
  Taking the difference between $\partial_\alpha^2 \ell_n(\alpha; \delta_0)$ and $\partial_\alpha^2 \ell_n(\alpha; \mu)$,  we have 
  \begin{align*}
    -\partial_\alpha^2 \ell_n(\alpha;\mu) + \partial_{\alpha}^2 \ell_{n}(\alpha; \delta_0) = -\mathcal{V}_{\theta\sim \mu_{n,\mathsf{k}_n}}\Bigl[\sum_{i=1}^{\mathsf{k}_n-1} \frac{i}{\theta + i\alpha}\Bigr] + \mathcal{E}_{\theta\sim \mu_{n,\mathsf{k}_n}}\Bigl[ \sum_{i=1}^{\mathsf{k}_n-1}\Bigl(\frac{i^2}{(\theta+i\alpha)^2}-\frac{1}{\alpha^2}\Bigr)\Bigr]. 
  \end{align*}
  Let us bound the variance term. Since the integrand can be decomposed as 
  $$
  \sum_{i=1}^{\mathsf{k}_n-1} \frac{i}{\theta+i\alpha} = \frac{1}{\alpha} - \sum_{i=1}^{\mathsf{k}_n-1} \frac{\theta}{\alpha(\theta+i\alpha)},
  $$
  where $1/\alpha$ is constant (with respect to $\theta\sim\mu_{n,\mathsf{k}_n}$), we have
  \begin{align*}
  0\le \mathcal{V}_{\theta}\Bigl(
    \sum_{i=1}^{\mathsf{k}_n-1} \frac{i}{\theta+i\alpha}
    \Bigr) &=  \mathcal{V}_{\theta}\Bigl(
      \sum_{i=1}^{\mathsf{k}_n-1} \frac{\theta}{\alpha(\theta+i\alpha)}
      \Bigr) \le \mathcal{E}_{\theta\sim \mu_{n,\mathsf{k}_n}} \Bigl[\Bigl(
        \sum_{i=1}^{\mathsf{k}_n-1} \frac{\theta}{\alpha(\theta+i\alpha)}
      \Bigr)^2 \Bigr]
  \end{align*}
  This leads to 
  \begin{align*}
    \bigl|\partial_\alpha^2 \ell_n(\alpha;\mu) - \partial_{\alpha}^2 \ell_{n}(\alpha; \delta_0)\bigr| \le \mathcal{E}_{\theta\sim \mu_{n,\mathsf{k}_n}} \Bigl[\Bigl(
      \sum_{i=1}^{\mathsf{k}_n-1} \frac{\theta}{\alpha(\theta+i\alpha)}
    \Bigr)^2 \Bigr] + \mathcal{E}_{\theta\sim \mu_{n,\mathsf{k}_n}}\Bigl[\Bigl|\sum_{i=1}^{\mathsf{k}_n-1} \Bigl(\frac{i^2}{(\theta+i\alpha)^2}-\frac{1}{\alpha^2}\Bigr)\Bigr|\Bigr] 
  \end{align*}
  Since $\operatorname{supp}(\mu_{n,\mathsf{k}_n})\in [\underline{\theta}, \bar\theta]$, there exists a deterministic constant $C(\alpha,\bar\theta, \underline{\theta})$ such that the integrands appearing in the above display are uniformly bounded as 
  \begin{align*}
    \sup_{\theta\in\operatorname{supp}(\mu_{n,\mathsf{k}_n})}
\Bigl(
      \sum_{i=1}^{\mathsf{k}_n-1} \frac{\theta}{\alpha(\theta+i\alpha)}
    \Bigr)^2 \le \Bigl(\sum_{i=1}^{\mathsf{k}_n-1} \frac{\bar\theta- \underline{\theta}}{\alpha(\underline{\theta}+i\alpha)}\Bigr)^2 \le C(\alpha,\bar\theta, \underline{\theta}) (\log n)^2
  \end{align*}
and 
\begin{align*}
   \sup_{\theta\in \operatorname{supp}(\mu_{n,\mathsf{k}_n})}  \Bigl|\sum_{i=1}^{\mathsf{k}_n-1} \Bigl(\frac{i^2}{(\theta+i\alpha)^2}-\frac{1}{\alpha^2}\Bigr)\Bigr| &=   \sup_{\theta\in \operatorname{supp}(\mu_{n,\mathsf{k}_n})}  \Bigl|\sum_{i=1}^{\mathsf{k}_n-1} \frac{-\theta^2 - 2i\alpha\theta}{\alpha^2(\theta+i\alpha)^2}\Bigr| \\
  &\le \sum_{i=1}^{\mathsf{k}_n-1} \frac{
   (\bar\theta-\underline{\theta})^2 + 2i \alpha (\bar\theta-\underline{\theta}) 
  }{\alpha^2 (\underline{\theta}+i\alpha)^2}\\
  & \le C(\alpha,\bar\theta, \underline{\theta}) \log n. 
\end{align*}
Therefore, the difference $\partial_\alpha^2 \ell_n(\alpha;\mu) - \partial_{\alpha}^2 \ell_{n}(\alpha; \delta_0)$ is bounded as 
\begin{align*}
  |\partial_\alpha^2 \ell_n(\alpha;\mu) - \partial_{\alpha}^2 \ell_{n}(\alpha; \delta_0)|
  &\le C(\alpha,\bar\theta, \underline{\theta})  \Bigl((\log n)^2 + \log n \Bigr)
\end{align*}
Since $(\log n)^2\ll n^\alpha$, in order to prove $n^{-\alpha} \E[\partial_{\alpha}^2 \ell_{n}(\alpha; \mu)] \to - \E[\diversity ]\mathfrak{i}(\alpha)$, 
it suffices to show 
$$
  n^{-\alpha} \E[\partial_{\alpha}^2 \ell_{n}(\alpha; \delta_0)] \to - \E[\diversity ]\mathfrak{i}(\alpha). 
$$
By the expression of $\partial_{\alpha}^2 \ell_{n}(\alpha; \delta_0)$ in \eqref{eq:ell_n_0_second}, 
$$
\frac{1}{n^\alpha} \partial_{\alpha}^2 \ell_{n}(\alpha; \delta_0) 
  = \frac{\mathsf{k}_n}{n^\alpha}\Bigl(
    - \frac{1-\mathsf{k}_n^{-1}}{\alpha^2} - \sum_{j=1}^n \frac{\mathsf{k}_{n,j}}{\mathsf{k}_n} \sum_{i=1}^{j-1}\frac{1}{(i-\alpha)^2}
  \Bigr).
$$
Now, combining $n^{-\alpha}\mathsf{k}_n\asconv \diversity$ and the dominated convergence theorem (\Cref{lm:bounded_conv}), noting that $j\mapsto \sum_{i=1}^{j-1}\frac{1}{(i-\alpha)^2}$ is bounded,  we have 
\begin{align*}
  \frac{1}{n^\alpha} \partial_{\alpha}^2 \ell_{n}(\alpha; \delta_0) 
  & \asconv \diversity \Bigl(-\frac{1}{\alpha^2} - \sum_{j=1}^\infty \mathfrak{p}_\alpha(j) \sum_{i=1}^{j-1}\frac{1}{(i-\alpha)^2}  \Bigr) = -\diversity \mathfrak{i}(\alpha)
\end{align*}
where the last equality follows from the first formula of $\mathfrak{i}(\alpha)$ in \Cref{lemma:fisher_info_sibuya}. Furthermore, with $\sum_{j=1}^n \mathsf{k}_{n,j}=\mathsf{k}_n$, $- \partial_{\alpha}^2 \ell_{n}(\alpha; \delta_0)$ is dominated by $\mathsf{k}_n$ up to a constant:
$$
0 \le - \partial_{\alpha}^2 \ell_{n}(\alpha; \delta_0) = \frac{\mathsf{k}_n-1}{\alpha^2}  + \sum_{j=1}^n \mathsf{k}_{n,j} \sum_{i=1}^{j-1}\frac{1}{(i-\alpha)^2} \le \mathsf{k}_n C,  \quad C = \frac{1}{\alpha^2} + \sum_{i=1}^\infty \frac{1}{(i-\alpha)^2}<+\infty
$$
Since $n^{-\alpha} \mathsf{k}_n$ is bounded in $L_2$ by \Cref{thm:Lp_converge}, $n^{-\alpha} \partial_\alpha^2 \ell_n(\alpha;\delta_0)$ is also bounded in $L_2$ and hence uniformly integrable. Thus, the almost sure convergence $n^{-\alpha} \partial_\alpha \ell_n(\alpha;\delta_0)\asconv -\diversity \mathfrak{i}(\alpha)$ implies the convergence in mean:
$
n^{-\alpha} \E[\partial_{\alpha}^2 \ell_{n}(\alpha; \delta_0)] \to - \E[\diversity] \mathfrak{i}(\alpha), 
$
which completes the proof. 

\section{Asymptotic analysis of QMLE}
\subsection{Proof of \Cref{prop:mle_existence}} 
Recall the expression: 
$$\mathsf{k}_{n+1} = \mathsf{k}_n + \xi_{n+1}, \quad \xi_{n+1}|\mathcal{F}_n\sim \Bernoulli(p_n)$$ 
where 
$$
p_n = \frac{v_{n+1, \mathsf{k}_{n}+1}}{v_{n, \mathsf{k}_n}}, \quad 1-p_n = 1-\frac{v_{n+1, \mathsf{k}_{n}+1}}{v_{n, \mathsf{k}_n}} \overset{(*)}{=} (n-\mathsf{k}_n\alpha) \frac{v_{n+1, \mathsf{k}_n}}{v_{n,\mathsf{k}_n}}
$$
where we have used the recursion \eqref{eq:recursion} for $(*)$. 
Then:
  \begin{align*}
      \PP(\mathsf{k}_n=1) &= \prod_{m=1}^{n-1} (m-\alpha) \frac{v_{m+1,1}}{v_{m,1}} = \prod_{m=1}^{n-1} \int \diff \mu_{m,1}(\theta) \frac{m-\alpha}{\theta+m}\\
      \PP(\mathsf{k}_n=n) &=  \prod_{m=1}^{n-1} \frac{v_{m+1,m+1}}{v_{m,m}} = \prod_{m=1}^{n-1} \int \diff \mu_{m,m}(\theta) \frac{\theta+\alpha m}{\theta+m}
  \end{align*}
  where $\mu_{m,1}$ and $\mu_{m, m}$ are the tilted probability measure defined in \eqref{eq:tilted_measure}. Using $\operatorname{supp}(\mu_{n,k})\subset [\underline\theta, \bar\theta]$ for any $1\le k\le n$ and the fact that $\theta\mapsto \frac{m-\alpha}{\theta+m}$ is decreasing and $\theta\mapsto  \frac{\theta+\alpha m}{\theta+m}$ is increasing, we get 
  \begin{align*}
      \PP(\mathsf{k}_n=1) &\le \prod_{m=1}^{n-1} \frac{m-\alpha}{m+\underline\theta} = \frac{\Gamma(n-\alpha)}{\Gamma(1-\alpha)} \frac{\Gamma(1+\underline\theta)}{\Gamma(n+\underline{\theta})}\\
     \PP(\mathsf{k}_n=n) &\le \prod_{m=1}^{n-1}\frac{\bar\theta+m\alpha}{\bar\theta+m} = \alpha^{n-1} \frac{\Gamma(\bar\theta+1)}{\Gamma(\bar\theta+n)} \frac{\Gamma(\bar\theta/\alpha +n)}{\Gamma(\bar\theta/\alpha+1)}
  \end{align*}
  With $\Gamma(n+c)/\Gamma(n)\sim n^c$ for any $c>0$, we have
  \begin{align*}
   \frac{\Gamma(n-\alpha)}{\Gamma(n+\underline{\theta})} \sim n^{-(\alpha +\underline{\theta})}, \quad 
    \frac{\Gamma(\bar\theta/\alpha +n)}{\Gamma(\bar\theta+n)} \sim n^{\bar\theta(\alpha^{-1}-1)}. 
  \end{align*}
  Putting them all together, we obtain the desired upper bound of $\PP(\mathsf{k}_n=1)$ and $\PP(\mathsf{k}_n=n)$. 

\subsection{Consistency of QMLE}
In this section, we prove the consistency of QMLE. The proof strategy is similar to  \cite[Section C]{koriyama2022asymptotic}, which studies the QMLE for the Ewens--Pitman partition $(\mu=\delta_\theta)$. Throughout this section, we fix $\alpha\in (0,1)$.

Let us define the deterministic function $\Psi:(0,1)\to\mathbb{R}$ as  
  \begin{align}\label{eq:psi_def}
    \Psi(x) \equiv \frac{1}{x} - \sum_{j=1}^\infty \mathfrak{p}_\alpha(j) \sum_{i=1}^{j-1} \frac{1}{i-x}, \quad \text{where} \quad \mathfrak{p}_\alpha(j) = \frac{\alpha\prod_{i=1}^{j-1} (i-\alpha)}{j!}
  \end{align}
  \begin{lemma}[{\cite[Lemma B.3]{koriyama2022asymptotic}}]\label{lm:psi_property}
    The map 
    $\Psi:(0,1)\to\mathbb{R}$  defined in \eqref{eq:psi_def} has the following properties.  
    \begin{itemize}
      \item $\Psi:(0,1)\to\mathbb{R}$ is $C^1$. 
      \item $\Psi$ is strictly decreasing and 
      $$   
      \forall x\in (0,1), \quad 
      \Psi'(x)=-\frac{1}{x^2} - \sum_{j=2}^\infty \mathfrak{p}_\alpha(j) \sum_{i=1}^{j-1}\frac{1}{(i-x)^2}<0
      $$
      Furthermore, it holds that $\Psi'(\alpha)=-\mathfrak{i}(\alpha) < 0$, where $\mathfrak{i}(\alpha)$ is the Fisher Information of the discrete distribution $\mathfrak{p}_\alpha(j)$ defined as \eqref{eq:fisher_sibuya}. 
      \item $\Psi$ has its unique root at $x=\alpha$, i.e., $\Psi(\alpha)=0$
    \end{itemize}
  \end{lemma}

  Given a random partition $\Pi_n = (U_1, U_2, \dots U_{\mathsf{k}_n})\in \mathcal{P}_n$, define the random function $\Psi_n:(0,1)\to\mathbb{R}$ as 
  \begin{align}\label{eq:psi_n_def}
    \Psi_n(x) \equiv \frac{\mathsf{k}_n-1}{ x \mathsf{k}_n} - \sum_{j=2}^n \frac{\mathsf{k}_{n,j}}{\mathsf{k}_n}\sum_{i=1}^{j-1}\frac{1}{i-x}    
  \end{align}
  where $\mathsf{k}_n$ is the number of nonempty blocks and $\mathsf{k}_{n,j}=\sum_{i=1}^{\mathsf{k}_n}\mathbbm{1}\{|U_i|=j\}$ is the number of blocks of size $j$. The next lemma claims that $\Psi_n$ is strictly decreasing and converges to the deterministic map $\Psi$ in a suitable sense. 
  \begin{lemma}\label{lm:psi_n_property}
    The random function $\Psi_n$ in \eqref{eq:psi_n_def} and the deterministic function $\Psi$ \eqref{eq:psi_def} satisfy the following:
    \begin{itemize}
      \item If $n\ge 2$, with probability $1$, 
      $-\Psi_n'(x) \ge 2^{-1}$ for all $x\in (0,1)$. 
      \item $\Psi_n(x) \to^p \Psi(x)$ for each $x\in (0,1)$. 
      \item For any fixed closed interval $I\subset (0,1)$, the uniform convergence
      $\sup_{x\in I} |\Psi_n'(x) - \Psi'(x)|=o_p(1)$ holds. 
    \end{itemize}
  \end{lemma}
  \begin{proof}

  Let us show the lower bound of $\Psi_n'$. Note
  $$
     -\Psi_n'(x) = \frac{\mathsf{k}_n-1}{ x^2 \mathsf{k}_n} + \sum_{j=2}^n \frac{\mathsf{k}_{n,j}}{\mathsf{k}_n}\sum_{i=1}^{j-1}\frac{1}{(i-x)^2} 
  $$
  Suppose $n\ge 2$.  If $\mathsf{k}_n=1$, then we must have $\mathsf{k}_{n,j}=0$ for $j<n$ and $\mathsf{k}_{n,n}=1$. Thus, if $\mathsf{k}_n=1$, 
    $$
    \forall x\in (0,1), \quad 
      -\Psi_n'(x) = \sum_{i=1}^{n-1} \frac{1}{(i-x)^2} \ge \sum_{i=1}^{n-1} \frac{1}{i^2} \ge \frac{1}{1^2}=1
    $$
    On the other hand, if $\mathsf{k}_n\ge 2$,
    $$
    \forall x\in (0,1), \quad 
    -\Psi_n'(x) \ge \frac{\mathsf{k}_n-1}{\mathsf{k}_n} \frac{1}{x^2} \ge \frac{1}{2}\frac{1}{x^2} \ge \frac{1}{2}. 
    $$
    This completes the proof of $-\Psi_n'(x) \ge 2^{-1}$ for all $x\in(0,1)$. \\
 
    Next, we prove the pointwise convergence $\Psi_n(x) \to^p \Psi(x)$ for each $x\in (0,1)$. 
    Let us show for $x=\alpha$ first. By $\Psi(\alpha)=0$, it suffices to show $\Psi_n(\alpha) = o_p(1)$. By the formula of log-likelihood, 
    \begin{align*}
      \Psi_n(\alpha) = \frac{\partial_\alpha\ell_n(\alpha;\delta_0)}{\mathsf{k}_n} = \frac{\partial_\alpha \ell_n(\alpha;\mu)}{\mathsf{k}_n}  +  \frac{1}{\mathsf{k}_n}\int \dd \mu_{n,\mathsf{k}_n}(\theta)  \sum_{i=1}^{\mathsf{k}_n-1}\frac{\theta}{\alpha(\theta+i\alpha)},  
    \end{align*}
    Since the support of $\mu_{n, \mathsf{k}_n}$ is bounded, the second term is $O(\log \mathsf{k}_n/\mathsf{k}_n)$ as $\mathsf{k}_n\to+\infty$, which is $o_p(1)$ since $\mathsf{k}_n\to \infty$ (a.s.).  For the first term, the variance of $\partial_\alpha\ell_n(\alpha;\mu)$ is $O(n^\alpha)$, thereby $\partial_\alpha\ell_n(\alpha;\editline{\mu})=O_p(n^{\alpha/2})$. Combined with $n^{-\alpha} \mathsf{k}_n\to\diversity >0$ (a.s.), the first term is $O_p(n^{-\alpha/2})=o_p(1)$. This finishes the proof of $\Psi_n(\alpha)\to^p  \Psi(\alpha)$. For $x\ne \alpha$, by the triangle inequality, 
    \begin{align*}
      |\Psi_n(x)-\Psi(x)| \le |\Psi_n(x)-\Psi_n(\alpha)- (\Psi(x)-\Psi(\alpha))| + (\Psi_n(\alpha)-\Psi(\alpha))
    \end{align*}
    where the last term is $o_p(1)$. Note 
    \begin{align*}
      \Psi_n(x)-\Psi_n(\alpha) &= \frac{\mathsf{k}_n-1}{\mathsf{k}_n} \Bigl(\frac{1}{x}-\frac{1}{\alpha}\Bigr) - \sum_{j=2}^n \frac{\mathsf{k}_{n,j}}{\mathsf{k}_n} \sum_{i=1}^{j-1} \frac{\alpha-x}{(i-x)(i-\alpha)}\\
      \Psi(x)-\Psi(\alpha) &= \Bigl(\frac{1}{x}-\frac{1}{\alpha}\Bigr) - \sum_{j=2}^\infty \mathfrak{p}_\alpha(j) \sum_{i=1}^{j-1} \frac{\alpha-x}{(i-x)(i-\alpha)}
    \end{align*}
    and hence 
    \begin{align*}
      \Psi_n(x)-\Psi_n(\alpha)- (\Psi(x)-\Psi(\alpha)) = -\frac{1}{\mathsf{k}_n} \Bigl(\frac{1}{x}-\frac{1}{\alpha}\Bigr) + \sum_{j=2}^\infty \Bigl(\mathfrak{p}_\alpha(j)-\frac{\mathsf{k}_{n,j}}{\mathsf{k}_n}\Bigr)  \sum_{i=1}^{j-1} \frac{|\alpha-x|}{(i-x)(i-\alpha)}. 
    \end{align*}
    The first term is $o_p(1)$ by $\mathsf{k}_n\to+\infty$ and the second term is $o_p(1)$ by \Cref{lm:bounded_conv} and the fact that $j\mapsto \sum_{i=1}^{j-1} \frac{|\alpha-x|}{(i-x)(i-\alpha)}$ is bounded. 
    This completes the proof of $\Psi_n(x) \to^p \Psi(x)$ for each $x\in (0,1)$.\\
    
    Finally, we prove the uniform convergence $\sup_{x\in I} |\Psi_n'(x) - \Psi'(x)| =o_p(1)$. Note
    \begin{align*}
      \Psi_n'(x) = - \frac{\mathsf{k}_n-1}{ x^2 \mathsf{k}_n} - \sum_{j=2}^n \frac{\mathsf{k}_{n,j}}{\mathsf{k}_n}\sum_{i=1}^{j-1}\frac{1}{(i-x)^2}, \quad 
      \Psi'(x) = -\frac{1}{x^2} - \sum_{j=2}^\infty \mathfrak{p}_\alpha(j) \sum_{i=1}^{j-1}\frac{1}{(i-x)^2},  
    \end{align*}
    so that 
    $$
    |\Psi_n'(x) - \Psi'(x)|\le \frac{1}{\mathsf{k}_n x^2}+ \sum_{j=2}^\infty |\frac{\mathsf{k}_{n,j}}{\mathsf{k}_n} - \mathfrak{p}_\alpha(j)| \cdot  \sum_{i=1}^{\infty}\frac{1}{(i-x)^2}.
    $$
    Let us write $I=[a, b]$ with $0<a<b<1$. Then, we have 
    $$
    \sup_{x\in I} |\Psi_n'(x) - \Psi'(x)| \le \frac{1}{\mathsf{k}_n a^2} + \sum_{j=2}^\infty |\frac{\mathsf{k}_{n,j}}{\mathsf{k}_n} - \mathfrak{p}_\alpha(j)| \cdot \sum_{i=1}^{\infty}\frac{1}{(i-b)^2},
    $$
    where the first term is $o_p(1)$ by $\mathsf{k}_n\to+\infty$ and the second term is $o_p(1)$ by \Cref{lm:bounded_conv}. This completes the proof of $ \sup_{x\in I} |\Psi_n'(x) - \Psi'(x)| =o_p(1)$.
  \end{proof}

  Combining \Cref{lm:psi_property} and \Cref{lm:psi_n_property}, using the standard argument (see \cite[Lemma C.3]{koriyama2022asymptotic}), we obtain the following corollary. 
  \begin{corollary}\label{cor:consistency_and_uniform_conv} 
    The QMLE and $\Psi_n$ satisfy the following:
    \begin{itemize}
      \item The QMLE is consistent, i.e., $\hat\alpha_n\to^p \alpha$
      \item For any $\mathcal{F}_n$-measurable random variable $\tilde{\alpha}_n$ satisfying  $\tilde{\alpha}_n\to^p\alpha$, we have  $\Psi_n'(\tilde{\alpha}_n)\to^p -\mathfrak{i}(\alpha)$. 
    \end{itemize}
  \end{corollary}

\subsection{Stable Martingale CLT}
Let $\ell_n(\alpha;\mu)$ be the log-likelihood of the Gibbs-partition under \Cref{assumption}, i.e, 
$$
\ell_n(\alpha; \mu) \equiv \log \Bigl(\int \dd \mu(\theta) v_{n, \mathsf{k}_n}(\alpha, \theta) \Bigr) + \sum_{j=2}^{n} \mathsf{k}_{n,j} \sum_{i=1}^{j-1} \log(i-\alpha), \quad n\ge 1. 
$$
Notice that $\ell_{1}(\alpha;\mu)=0$ with probability $1$ since $\mathsf{k}_1 = 1$ (w.p.1) and $v_{1,1}=0$. 
Here we recall $\operatorname{supp}(\mu)\subset [\underline{\theta}, \bar{\theta}]$ with $-\alpha < \underline{\theta}\le 0 \le \bar{\theta} <+\infty$ by \Cref{assumption}. In \Cref{sec:proof_fisher}, we proved that  the derivative $\partial_\alpha$ and the integration $\int\dd \mu$ is interchangeable. Using the notation of the tilted probability measure $\mu_{n,k}$ in \eqref{eq:tilted_measure}, the score function can be written as 
\begin{align}\label{eq:score}
  \partial_\alpha \ell_n(\alpha;\mu) &= \int \dd \mu_{n, \mathsf{k}_n} (\theta) \sum_{i=1}^{\mathsf{k}_n-1} \frac{i}{\theta+i\alpha} - \sum_{j=2}^n \mathsf{k}_{n,j} \sum_{i=1}^{j-1} \frac{1}{i-\alpha}\nonumber \\
  & = \frac{\mathsf{k}_n-1}{\alpha}  - \int \dd \mu_{n,\mathsf{k}_n}(\theta)  \sum_{i=1}^{\mathsf{k}_n-1}\frac{\theta}{\alpha(\theta+i\alpha)}  - \sum_{j=2}^n \mathsf{k}_{n,j} \sum_{i=1}^{j-1} \frac{1}{i-\alpha}
\end{align}
where the second equality follows from $\frac{i}{\theta+i\alpha} = \frac{1}{\alpha}  - \frac{\theta}{\alpha(\theta+i\alpha)}$. The goal of this section is to show the asymptotic mixed normality of the score function:
  \begin{align*}
    \frac{\partial_\alpha \ell_n(\alpha;\mu)}{\sqrt{n^\alpha \mathfrak{i}(\alpha)}} \to \diversity^{1/2} N,  \quad \stable,
  \end{align*}
  where $\mathcal{F}_\infty = \sigma(\cup_{n=1}^\infty \mathcal{F}_n)$ and $\mathcal{F}_n=\sigma(\Pi_n)$ is the $\sigma$-field generated by the random partition $\Pi_n\in \mathcal{P}_n$ of $[n]$.
  We will prove this result by the stable Martingale CLT \cite{hausler2015stable}. 

\begin{lemma}\label{lm:martingale_diff}
   Let $\{X_{n}\}_{n=1}^\infty$ be the martingale difference of the score function, i.e., 
 $$
 X_1 \equiv \partial_{\alpha} \ell_{1} (\alpha;\mu) = 0, \quad 
  X_{m+1} \equiv \partial_{\alpha} \ell_{m+1}(\alpha; \mu) - \partial_{\alpha} \ell_{m}(\alpha; \mu), \quad m\ge 1. 
 $$
Now, we define the $\mathcal{F}_{m+1}$-measurable event $\Omega_{m+1, j}, 0\le j\le m$ as
\begin{align*}
  \forall j\in \{1, \dots m\}, \quad \Omega_{m+1,j} &= \{\text{$m+1$-th ball belongs to an existing set of size $j$}\}\\
  \Omega_{m+1, 0} &= \{\text{$m+1$-th ball belongs to a new set}\}
\end{align*}
Then, the increment $X_{m+1}$ is given by 
\begin{align*}
  X_{m+1} =\begin{cases}
    \res_{m}  -\frac{1}{j-\alpha} & \text{under $\Omega_{m+1, j}$ for $j\in \{1, \dots m\}$}\\ 
    \res'_{m} + \frac{1}{\alpha} & \text{under $\Omega_{m+1, 0}$}
  \end{cases} 
\end{align*}
where $\res_m$ and $\res_m'$ are $\mathcal{F}_m$-measurable random variable, whose absolute values are uniformly controlled as 
$$
|\res_m| + |\res_m'| \le C(\alpha,\underline\theta, \bar\theta) (\mathsf{k}_m^{-1} \log \mathsf{k}_m + m^{-1} \log m)
$$
where $C(\alpha,\underline\theta, \bar\theta)$ is a deterministic constant depending on $\alpha,\underline\theta, \bar\theta$ only.
\end{lemma}

\begin{proof}
The claim for $m=0$ is obvious so let us show for $m\ge 1$. Recall the definition of $\mathsf{k}_{m, j}$; $\mathsf{k}_{m,j}$ is the number of blocks of size $j$ given the partition $(U_1, \dots U_{\mathsf{k}_m})$ of $[m]$, and $\mathsf{k}_m=\sum_{j=1}^m\mathsf{k}_{m,j}$ is the number of non-empty blocks. Then, under $\Omega_{m+1, j}$ for $j\in\{1, \dots m\}$, $\{\mathsf{k}_{m+1,\ell}\}_{\ell=1}^{m+1}$ and $\mathsf{k}_{m+1}$ are updated as 
$$
\mathsf{k}_{m+1,\ell} = \begin{cases}
  \mathsf{k}_{m, \ell} - 1& \ell =j\\
  \mathsf{k}_{m, \ell} + 1 & \ell = j+1\\
  \mathsf{k}_{m, \ell} & \text{else}
\end{cases}, \quad \mathsf{k}_{m+1} = \mathsf{k}_m.
$$
Note $\mathsf{k}_{m, m+1}=0$ by the definition of $\mathsf{k}_{n,j}$.
Using the expression of the score function \eqref{eq:score}, the score $\partial_{\alpha}\ell_{m+1}(\alpha;\mu)$ under $\Omega_{m+1, j}$ is given by
\begin{align*}
  &\partial_{\alpha} \ell_{m+1}(\alpha;\mu) \\
  &=  \frac{\mathsf{k}_{m+1}-1}{\alpha}  - \int \dd \mu_{m+1,\mathsf{k}_{m+1}}(\theta)  \sum_{i=1}^{\mathsf{k}_{m+1}-1}\frac{\theta}{\alpha(\theta+i\alpha)} - \sum_{\ell=2}^{m+1} \mathsf{k}_{m+1,\ell} \sum_{i=1}^{\ell-1} \frac{1}{i-\alpha}\\
  &= \frac{\mathsf{k}_{m}-1}{\alpha}  - \int \dd \mu_{m+1,\mathsf{k}_{m}}(\theta)  \sum_{i=1}^{\mathsf{k}_{m}-1}\frac{\theta}{\alpha(\theta+i\alpha)} - \sum_{\ell=2}^{m} \mathsf{k}_{m,\ell} \sum_{i=1}^{\ell-1} \frac{1}{i-\alpha} + \sum_{i=1}^{j-1} \frac{1}{i-\alpha} -  \sum_{i=1}^{j} \frac{1}{i-\alpha} \\
  &= \frac{\mathsf{k}_{m}-1}{\alpha}  - \int \dd \mu_{m+1,\mathsf{k}_{m}}(\theta)  \sum_{i=1}^{\mathsf{k}_{m}-1}\frac{\theta}{\alpha(\theta+i\alpha)}  - \sum_{\ell=2}^{m} \mathsf{k}_{m,\ell} \sum_{i=1}^{\ell-1} \frac{1}{i-\alpha} - \frac{1}{j-\alpha}\\
  &= \partial_{\alpha}\ell_m(\alpha;\mu) + \int \Bigl(\dd \mu_{m+1, \mathsf{k}_{m}} (\theta) - \dd \mu_{m, \mathsf{k}_{m}} (\theta) \Bigr) \sum_{i=1}^{\mathsf{k}_m-1} \frac{\theta}{\alpha(\theta+i\alpha)}  - \frac{1}{j-\alpha}
\end{align*}
On the other hand, under the event $\Omega_{m+1,0}$, we have 
$$
\mathsf{k}_{m+1,\ell} = \begin{cases}
  \mathsf{k}_{m, \ell} + 1 & \ell = 1\\
  \mathsf{k}_{m, \ell} & \text{else}
\end{cases}, \quad \mathsf{k}_{m+1} = \mathsf{k}_m+1.
$$
Then, the score under this event is updated as 
\begin{align*}
  &\partial_{\alpha} \ell_{m+1}(\alpha;\mu)\\
  &= \frac{\mathsf{k}_{m}+1-1}{\alpha}  - \int \dd \mu_{m+1,\mathsf{k}_{m}+1}(\theta)  \sum_{i=1}^{\mathsf{k}_{m}+1-1}\frac{\theta}{\alpha(\theta+i\alpha)} - \sum_{\ell=2}^{m} \mathsf{k}_{m,\ell} \sum_{i=1}^{\ell-1} \frac{1}{i-\alpha} \\
  &= \frac{\mathsf{k}_n}{\alpha}  - \int \dd \mu_{m+1,\mathsf{k}_{m}+1}(\theta)  \sum_{i=1}^{\mathsf{k}_{m}-1}\frac{\theta}{\alpha(\theta+i\alpha)} - \int\dd\mu_{m+1, \mathsf{k}_m+1} (\theta) \frac{\theta}{\alpha(\theta+\mathsf{k}_m \alpha)}  - \sum_{\ell=2}^{m} \mathsf{k}_{m,\ell} \sum_{i=1}^{\ell-1} \frac{1}{i-\alpha}\\
  &= \partial_{\alpha} \ell_{m} (\alpha,\mu) + \frac{1}{\alpha} - \int\dd\mu_{m+1, \mathsf{k}_m+1} (\theta) \frac{\theta}{\alpha(\theta+\mathsf{k}_m \alpha)}  - \int \Bigl(\dd \mu_{m+1,\mathsf{k}_{m}+1}(\theta) - \dd \mu_{m, \mathsf{k}_m} \Bigr)  \sum_{i=1}^{\mathsf{k}_{m}-1}\frac{\theta}{\alpha(\theta+i\alpha)}. 
\end{align*}
Therefore, the increment $X_{m+1}$ is given by 
\begin{align*}
  X_{m+1} =\begin{cases}
    \res_{m}  -\frac{1}{j-\alpha} & \text{under $\Omega_{m+1, j}$ for $j\in \{1, \dots m\}$}\\ 
    \res'_{m} + \frac{1}{\alpha} & \text{under $\Omega_{m+1, 0}$}
  \end{cases} 
\end{align*}
where $\res_m$ and $\res_m'$ are $\mathcal{F}_m$-measurable random variable defined as 
\begin{align*}
  \res_m &\equiv \int \Bigl(\dd \mu_{m+1, \mathsf{k}_{m}} (\theta) - \dd \mu_{m, \mathsf{k}_{m}} (\theta) \Bigr) \sum_{i=1}^{\mathsf{k}_m-1} \frac{\theta}{\alpha(\theta+i\alpha)}\\
  \res_m' &\equiv - \int\dd\mu_{m+1, \mathsf{k}_m+1} (\theta) \frac{\theta}{\alpha(\theta+\mathsf{k}_m \alpha)}  - \int \Bigl(\dd \mu_{m+1,\mathsf{k}_{m}+1}(\theta) - \dd \mu_{m, \mathsf{k}_m} \Bigr)  \sum_{i=1}^{\mathsf{k}_{m}-1}\frac{\theta}{\alpha(\theta+i\alpha)} 
\end{align*}
Let us derive upper bounds of $\res_m$ first. Recall the definition of the tilted measure $d\mu_{n,k}(\theta)$ for $1\le k\le n$:
\begin{align*}
    d\mu_{n,k}(\theta) \equiv \frac{\dd \mu(\theta) v_{n,k}(\alpha, \theta)}{\int \dd\mu(\theta') v_{n,k}(\alpha,\theta')} \quad \text{with} \quad v_{n,k}(\alpha;\theta) = \frac{\prod_{i=1}^{k-1} (\theta + i\alpha)}{\prod_{i=1}^{n-1}(\theta + i)}.
\end{align*}
By $v_{m+1, \mathsf{k}_m} = v_{m, \mathsf{k}_m} \frac{1}{\theta+m}$, $\res_{m}$ can be written as  
\begin{align*}
  \res_{m} &= \int \Bigl(\dd \mu_{m+1, \mathsf{k}_{m}} (\theta) - \dd \mu_{m, \mathsf{k}_{m}} (\theta) \Bigr) \sum_{i=1}^{\mathsf{k}_m-1} \frac{\theta}{\alpha(\theta+i\alpha)} \\
  &= \int\dd \mu(\theta)   \Bigl( \frac{v_{m, \mathsf{k}_m}(\theta) \frac{1}{\theta+m}}{\int \dd \mu (\theta') v_{m, \mathsf{k}_m}(\theta')  \frac{1}{\theta'+m}} -  \frac{v_{m, \mathsf{k}_m}(\theta)}{\int \dd \mu (\theta') v_{m, \mathsf{k}_m}(\theta') }   \Bigr) \sum_{i=1}^{\mathsf{k}_m-1} \frac{\theta}{\alpha(\theta+i\alpha)}
\end{align*}
Since $\operatorname{supp}(\mu) \subset [\underline\theta, \bar\theta]$, the following inequality holds for all $\theta\in \operatorname{supp}(\mu)$:
$$
\frac{\frac{1}{\bar\theta+m}}{\frac{1}{\underline{\theta}+m}} \cdot 
\frac{v_{m, \mathsf{k}_m}(\theta)}{\int \dd \mu (\theta') v_{m, \mathsf{k}_m}(\theta')} \le 
\frac{v_{m, \mathsf{k}_m}(\theta) \frac{1}{\theta+m}}{\int \dd \mu (\theta') v_{m, \mathsf{k}_m}(\theta')  \frac{1}{\theta'+m}} \le \frac{\frac{1}{\underline{\theta}+m}} {\frac{1}{\bar\theta+m}}\cdot 
\frac{v_{m, \mathsf{k}_m}(\theta)}{\int \dd \mu (\theta') v_{m, \mathsf{k}_m}(\theta')}. 
$$
Using $\frac{\underline{\theta}+m}{\bar\theta+m}-1  = \frac{\underline{\theta}-\bar\theta}{\bar\theta+m}$ and 
$\frac{\bar\theta+m}{\underline{\theta}+m}-1  = \frac{\bar\theta-\underline{\theta}}{\underline\theta+m}$, 
noting that $|\frac{\bar\theta-\underline{\theta}}{\underline\theta+m}|>|\frac{\underline{\theta}-\bar\theta}{\bar\theta+m}|$ by $\bar\theta>\underline{\theta}$, we have 
\begin{align*}
  \forall \theta\in \operatorname{supp}(\mu), \quad 
  \left|\frac{v_{m, \mathsf{k}_m}(\theta) \frac{1}{\theta+m}}{\int \dd \mu (\theta') v_{m, \mathsf{k}_m}(\theta')  \frac{1}{\theta'+m}}  - \frac{v_{m, \mathsf{k}_m}(\theta)}{\int \dd \mu (\theta') v_{m, \mathsf{k}_m}(\theta')}\right|\le  \frac{\bar\theta-\underline{\theta}}{\underline\theta+m}\cdot 
\frac{v_{m, \mathsf{k}_m}(\theta)}{\int \dd \mu (\theta') v_{m, \mathsf{k}_m}(\theta')}. 
\end{align*}
Putting them all together, we get 
\begin{align}
  \Bigl| \int \Bigl(\dd \mu_{m+1, \mathsf{k}_{m}} (\theta) - \dd \mu_{m, \mathsf{k}_{m}} (\theta) \Bigr) \sum_{i=1}^{\mathsf{k}_m-1} \frac{\theta}{\alpha(\theta+i\alpha)}\Bigr| &\le \frac{\bar\theta-\underline{\theta}}{\underline\theta+m} \int \dd\mu(\theta)
\frac{v_{m, \mathsf{k}_m}(\theta)}{\int \dd \mu (\theta') v_{m, \mathsf{k}_m}(\theta')} \Bigl|\sum_{i=1}^{\mathsf{k}_m-1}\frac{\theta}{\theta+i\alpha}\Bigr|\nonumber \\
&= \frac{\bar\theta-\underline{\theta}}{\underline\theta+m} \int \dd \mu_{m, \mathsf{k}_m}(\theta) \Bigl|\sum_{i=1}^{\mathsf{k}_m-1}\frac{\theta}{\theta+i\alpha}\Bigr|\nonumber\\
&\le \frac{\bar\theta-\underline{\theta}}{\underline\theta+m} \int \dd \mu_{m, \mathsf{k}_m}(\theta) \sum_{i=1}^{\mathsf{k}_m-1}\Bigl|\frac{\theta}{\theta+i\alpha}\Bigr|\nonumber \\
&\le \frac{\bar\theta-\underline{\theta}}{\underline\theta+m} \sum_{i=1}^{\mathsf{k}_m-1} \frac{\bar\theta-\underline{\theta}}{\alpha(\underline{\theta}+i\alpha)}\label{eq:res_m_bound}
\end{align}
where we have used $\operatorname{supp}(\mu_{m, \mathsf{k}_m}) \subset [\underline{\theta}, \bar\theta]$ for the last inequality. Thus, there exists a positive constant $C=C(\alpha,\underline\theta, \bar\theta)$ such that 
$$
|\res_{m}|\le C m^{-1} \log m.
$$
Next, we derive an upper bound of $\res_m'$. Using the tilted measure again, we can write it as 
$$
\res_{m}' = - \int\dd\mu_{m+1, \mathsf{k}_m+1} (\theta) \frac{\theta}{\alpha(\theta+\mathsf{k}_m \alpha)}  - \int \Bigl(\dd \mu_{m+1,\mathsf{k}_{m}+1}(\theta) - \dd \mu_{m, \mathsf{k}_m}(\theta) \Bigr)  \sum_{i=1}^{\mathsf{k}_{m}-1}\frac{\theta}{\alpha(\theta+i\alpha)}. 
$$
The first term is easy to control; using $\operatorname{supp}(\mu_{m+1, \mathsf{k}_{m+1}})\subset [\underline{\theta}, \bar\theta]$, 
$$
\Bigl|\int\dd\mu_{m+1, \mathsf{k}_m+1} (\theta) \frac{\theta}{\alpha(\theta+\mathsf{k}_m \alpha)}\Bigr| \le \frac{\bar\theta-\underline{\theta}}{\alpha(\underline\theta+\mathsf{k}_m\alpha)} \le C(\alpha,\underline\theta, \bar\theta)\cdot \mathsf{k}_m^{-1}
$$
As for the second term, we decompose it into two terms:
\begin{align*}
  &\int \Bigl(\dd \mu_{m+1,\mathsf{k}_{m}+1}(\theta) - \dd \mu_{m, \mathsf{k}_m} (\theta)\Bigr)  \sum_{i=1}^{\mathsf{k}_{m}-1}\frac{\theta}{\alpha(\theta+i\alpha)} \\
  &= \int \Bigl(\dd \mu_{m+1,\mathsf{k}_{m}+1}(\theta) - \dd \mu_{m+1, \mathsf{k}_m}(\theta) \Bigr)  \sum_{i=1}^{\mathsf{k}_{m}-1}\frac{\theta}{\alpha(\theta+i\alpha)} \\
  &+ \int \Bigl(\dd \mu_{m+1,\mathsf{k}_{m}}(\theta) - \dd \mu_{m, \mathsf{k}_m}(\theta) \Bigr)  \sum_{i=1}^{\mathsf{k}_{m}-1}\frac{\theta}{\alpha(\theta+i\alpha)}, 
\end{align*}
where the absolute value of the second term is bounded by $m^{-1} \log m$ up to a deterministic constant by \eqref{eq:res_m_bound}. Let us bound the first term using the same argument as \eqref{eq:res_m_bound}. By the definition of the tilted measure, using $v_{m+1, \mathsf{k}_m+1} = v_{m+1, \mathsf{k}_m} (\theta + \mathsf{k}_m\alpha)$, we have 
\begin{align*}
  &\int \Bigl(\dd \mu_{m+1,\mathsf{k}_{m}+1}(\theta) - \dd \mu_{m+1, \mathsf{k}_m} (\theta)\Bigr)  \sum_{i=1}^{\mathsf{k}_{m}-1}\frac{\theta}{\alpha(\theta+i\alpha)} \\
  &= \int \dd\mu(\theta) \Bigl(\frac{v_{m+1, \mathsf{k}_m} (\theta) \cdot (\theta + \mathsf{k}_m\alpha ) }{\int \dd\mu(\theta') v_{m+1, \mathsf{k}_m}(\theta') \cdot (\theta' + \mathsf{k}_m\alpha ) } - \frac{v_{m+1, \mathsf{k}_m} (\theta)}{\int \dd\mu(\theta') v_{m+1, \mathsf{k}_m}(\theta') } \Bigr)  \sum_{i=1}^{\mathsf{k}_{m}-1}\frac{\theta}{\alpha(\theta+i\alpha)},
\end{align*}
and for all $\theta\in \operatorname{supp}(\mu)\subset [\underline{\theta}, \bar\theta]$, 
\begin{align*}
  \frac{\underline\theta + \mathsf{k}_m\alpha }{\bar\theta + \mathsf{k}_m\alpha } \cdot \frac{v_{m+1, \mathsf{k}_m} (\theta)}{\int \dd\mu(\theta') v_{m+1, \mathsf{k}_m}(\theta')}\le  \frac{v_{m+1, \mathsf{k}_m} (\theta) \cdot (\theta + \mathsf{k}_m\alpha ) }{\int \dd\mu(\theta') v_{m+1, \mathsf{k}_m}(\theta') \cdot (\theta' + \mathsf{k}_m\alpha ) }  \le \frac{\bar\theta + \mathsf{k}_m\alpha }{\underline\theta + \mathsf{k}_m\alpha } \cdot  \frac{v_{m+1, \mathsf{k}_m} (\theta)}{\int \dd\mu(\theta') v_{m+1, \mathsf{k}_m}(\theta')}.
\end{align*}
Using $\frac{\underline\theta + \mathsf{k}_m\alpha }{\bar\theta + \mathsf{k}_m\alpha } -1 = \frac{\underline\theta-\bar\theta}{\bar\theta+\mathsf{k}_m\alpha}$ and $\frac{\bar\theta + \mathsf{k}_m\alpha }{\underline\theta + \mathsf{k}_m\alpha } -1 = \frac{\bar\theta-\underline\theta}{\underline\theta+\mathsf{k}_m\alpha}$, with $|\frac{\bar\theta-\underline\theta}{\underline\theta+\mathsf{k}_m\alpha}| > |\frac{\underline\theta-\bar\theta}{\bar\theta+\mathsf{k}_m\alpha}|$ by $\bar\theta>\underline{\theta}$, 
we have 
\begin{align*}
  \Bigl|\int \Bigl(\dd \mu_{m+1,\mathsf{k}_{m}+1}(\theta) - \dd \mu_{m+1, \mathsf{k}_m} (\theta)\Bigr)  \sum_{i=1}^{\mathsf{k}_{m}-1}\frac{\theta}{\alpha(\theta+i\alpha)}\Bigr| &\le \frac{\bar\theta-\underline\theta}{\underline\theta + \mathsf{k}_m\alpha }\int \dd \mu_{m+1, \mathsf{k}_m}(\theta) \Bigl|\sum_{i=1}^{\mathsf{k}_{m}-1}\frac{\theta}{\alpha(\theta+i\alpha)}\Bigr|\\
  &\le \frac{\bar\theta-\underline\theta}{\underline\theta + \mathsf{k}_m\alpha } \sum_{i=1}^{\mathsf{k}_m-1} \frac{\bar\theta-\underline{\theta}}{\alpha(\underline{\theta}+i\alpha)}
\end{align*}
Therefore, there exists a deterministic constant $C=C(\alpha, \underline{\theta}, \bar\theta)$ such that 
\begin{align*}
  |\res_{m}'| \le C(\alpha, \underline{\theta}, \bar\theta) (
    \mathsf{k}_m^{-1} + \mathsf{k}_m^{-1}\log \mathsf{k}_m + m^{-1}\log m)
\end{align*}
This completes the proof. 
\end{proof}
Now we derive the almost sure limit of the quadratic variation of the martingale difference $X_{m+1}$. 
\begin{lemma}\label{lm:quadratic_variation}
Let $X_{m+1}$ be the martingale difference defined in \Cref{lm:martingale_diff}.  Then, we have
  $$
  n^{-\alpha} \sum_{m=0}^{n-1} \E[X_{m+1}^2|\mathcal{F}_{m}] \to^p \mathfrak{i}(\alpha) \diversity,
  $$
  where $\diversity = \lim_{n\to+\infty} n^{-\alpha} \editline{\mathsf{k}_n}$ (a.s.) and $\mathfrak{i}(\alpha)$ is the Fisher Information of the discrete distribution $\mathfrak{p}_\alpha(j)$. 
\end{lemma}

\begin{proof}
Below, we write $A_n\lesssim B_n$ if there exists a deterministic constant $C$ which only depends on $(\alpha, \underline\theta, \bar\theta)$ such that $A_n \le C B_n$ with probability $1$. 
Recall the definition of $\Omega_{m+1, j}$ in \Cref{lm:martingale_diff}. 
By \eqref{eq:rule}, the conditional probability of these events are given by 
\begin{align*}
  \forall j\in \{1, \dots, m\}, \quad \PP(\Omega_{m+1,j}|\mathcal{F}_m) &= \mathsf{k}_{m,j} \frac{v_{m+1, \mathsf{k}_m}}{v_{m, \mathsf{k}_m}} (j-\alpha)\\
  \PP(\Omega_{m+1, 0}|\mathcal{F}_m) &= \frac{v_{m+1, \mathsf{k}_{m}+1}}{v_{m, \mathsf{k}_m}},
\end{align*}
where $\mathsf{k}_{m,j}$ is the number of blocks of size $j$ among the given partition $(U_1, \dots U_{\mathsf{k}_m})$ of $[m]$. Then, by the expression of the martingale difference given by \Cref{lm:martingale_diff}, we have 
$$
\E[X_{m+1}^2|\mathcal{F}_{m}] = \sum_{j=1}^{m}\Bigl(\frac{1}{(j-\alpha)}+\res_m\Bigr)^2 \mathsf{k}_{m,j} \frac{v_{m+1, \mathsf{k}_m}}{v_{m, \mathsf{k}_m}} (j-\alpha) + \bigl(\frac{1}{\alpha} + \res_{m}'\bigr)^2 \frac{v_{m+1, \mathsf{k}_{m}+1}}{v_{m, \mathsf{k}_m}}
$$
with $|\res_m'|+|\res_m|\lesssim \mathsf{k}_m^{-1}\log\mathsf{k}_m + m^{-1}\log m$. Now we claim that
$$
 \frac{m}{\mathsf{k}_m} \E[X_{m+1}^2|\mathcal{F}_m] \asconv \alpha \mathfrak{i}(\alpha) 
$$
as $m\to+\infty$. Here, by \Cref{lm:concentrate_v_nk}, $\frac{v_{m+1, \mathsf{k}_{m}}} {v_{m, \mathsf{k}_m}}$ and $\frac{v_{m+1, \mathsf{k}_{m}+1}} {v_{m, \mathsf{k}_m}}$ are concentrated as 
$$
\Bigl|\frac{v_{m+1, \mathsf{k}_{m}}} {v_{m, \mathsf{k}_m}}- \frac{1}{m}\Bigr|\lesssim \frac{1}{m^2}, \quad \Bigl|\frac{v_{m+1, \mathsf{k}_{m}+1}} {v_{m, \mathsf{k}_m}} - \frac{\mathsf{k}_m}{m}\alpha \Bigr|\lesssim \frac{1}{m}. 
$$
Note 
\begin{align*}
  &\frac{m}{\mathsf{k}_m} \cdot \sum_{j=1}^{m}\Bigl(\frac{1}{(j-\alpha)}+\res_m\Bigr)^2 \mathsf{k}_{m,j} \frac{v_{m+1, \mathsf{k}_m}}{v_{m, \mathsf{k}_m}} (j-\alpha)  \\
  &= \sum_{j=1}^{m}\Bigl(\frac{1}{(j-\alpha)}+\res_m\Bigr)^2 \frac{\mathsf{k}_{m,j}}{\mathsf{k}_m} \cdot m \frac{v_{m+1, \mathsf{k}_m}}{v_{m, \mathsf{k}_m}} (j-\alpha)\\
  &= \sum_{j=1}^{m}\Bigl(\frac{1}{(j-\alpha)}+\res_m\Bigr)^2 (j-\alpha) \frac{\mathsf{k}_{m,j}}{\mathsf{k}_m} + \sum_{j=1}^{m}\Bigl(\frac{1}{(j-\alpha)}+\res_m\Bigr)^2 (j-\alpha) \Bigl(m \frac{v_{m+1, \mathsf{k}_m}}{v_{m, \mathsf{k}_m}} - 1\Bigr) \frac{\mathsf{k}_{m,j}}{\mathsf{k}_m}, 
\end{align*}
where the second term is $o_{a.s}(1)$:
\begin{align*}
&\left|\sum_{j=1}^{m}\Bigl(\frac{1}{(j-\alpha)}+\res_m\Bigr)^2 (j-\alpha) \Bigl(m \frac{v_{m+1, \mathsf{k}_m}}{v_{m, \mathsf{k}_m}} - 1\Bigr) \frac{\mathsf{k}_{m,j}}{\mathsf{k}_m}\right| \\
&\le  \sup_{j\in [m]} \Bigl(\frac{1}{(j-\alpha)}+\res_m\Bigr)^2 (j-\alpha) \cdot \Bigl|m \frac{v_{m+1, \mathsf{k}_m}}{v_{m, \mathsf{k}_m}} - 1\Bigr| && \text{$\sum_{j=1}^m \frac{\mathsf{k}_{m,j}}{\mathsf{k}_m}=1$}\\
&\lesssim  \sup_{j\in [m]} \Bigl(\frac{1}{(j-\alpha)^2}+\res_m^2\Bigr)  (j-\alpha) \cdot \frac{1}{m} && \Bigl|\frac{v_{m+1, \mathsf{k}_{m}}} {v_{m, \mathsf{k}_m}}- \frac{1}{m}\Bigr|\lesssim \frac{1}{m^2} \\
&\le  \sup_{j\in [m]} \Bigl(
  \frac{1}{j-\alpha} + (j-\alpha)\bigl(m^{-2}\log^2 m + \mathsf{k}_m^{-2} \log^2 \mathsf{k}_m\bigr) 
\Bigr)\cdot \frac{1}{m} && |\res_m| \lesssim \mathsf{k}_m^{-1}\log\mathsf{k}_m + m^{-1}\log m
\\
&\lesssim  \frac{1}{m}\Bigl(
\frac{1}{1-\alpha} + (m-\alpha) \bigl(m^{-2}\log^2 m + \mathsf{k}_m^{-2} \log^2 \mathsf{k}_m\bigr)
\Bigr)\\
&\asconv 0 && \mathsf{k}_m \asconv +\infty  
\end{align*}
By the same argument, we have 
\begin{align*}
  \frac{m}{\mathsf{k}_m}\cdot \bigl(\frac{1}{\alpha} + \res_{m}'\bigr)^2 \frac{v_{m+1, \mathsf{k}_{m}+1}}{v_{m, \mathsf{k}_m}} = \Bigl(\frac{1}{\alpha} + \res_{m}'\Bigr)^2 \alpha + \bigl(\frac{1}{\alpha} + \res_{m}'\bigr)^2 \frac{m}{\mathsf{k}_m} \Bigl(\frac{v_{m+1, \mathsf{k}_{m}+1}} {v_{m, \mathsf{k}_m}} - \frac{\mathsf{k}_m}{m}\alpha \Bigr),
\end{align*}
where the second term is $o_{a.s.}(1)$:
\begin{align*}
\left|\bigl(\frac{1}{\alpha} + \res_{m}'\bigr)^2 \frac{m}{\mathsf{k}_m} \Bigl(\frac{v_{m+1, \mathsf{k}_{m}+1}} {v_{m, \mathsf{k}_m}} - \frac{\mathsf{k}_m}{m}\alpha \Bigr)\right| &\lesssim \bigl(\frac{1}{\alpha} + \res_{m}'\bigr)^2 \frac{m}{\mathsf{k}_m} \cdot \frac{1}{m} && \Bigl|\frac{v_{m+1, \mathsf{k}_{m}+1}} {v_{m, \mathsf{k}_m}} - \frac{\mathsf{k}_m}{m}\alpha \Bigr|\lesssim \frac{1}{m}. \\
&\asconv 0 && |\res_m'| \asconv 0, \quad \mathsf{k}_m\asconv +\infty. 
\end{align*}
Therefore, we obtain
$$
 \frac{m}{\mathsf{k}_m} \E[X_{m+1}^2|\mathcal{F}_m]  = \sum_{j=1}^{m}\Bigl(\frac{1}{(j-\alpha)}+\res_m\Bigr)^2 (j-\alpha) \frac{\mathsf{k}_{m,j}}{\mathsf{k}_m}  + \Bigl(\frac{1}{\alpha} + \res_{m}'\Bigr)^2 \alpha + o_{a.s.} (1). 
$$
Here, $\Bigl(\frac{1}{\alpha} + \res_{m}'\Bigr)^2 \alpha \asconv \frac{1}{\alpha}$ by $\res_m'\asconv 0$ while 
\begin{align*}
  &\Bigl|\sum_{j=1}^{m}\Bigl(\frac{1}{(j-\alpha)}+\res_m\Bigr)^2 (j-\alpha) \frac{\mathsf{k}_{m,j}}{\mathsf{k}_m} -   \sum_{j=1}^{m}\Bigl(\frac{1}{(j-\alpha)}\Bigr)^2 (j-\alpha) \frac{\mathsf{k}_{m,j}}{\mathsf{k}_m} \Bigr|\\
  &\le \sup_{j\in [m]} \Bigl|\Bigl(\frac{1}{(j-\alpha)}+\res_m\Bigr)^2  - \frac{1}{(j-\alpha)^2}\Bigr| && \frac{\sum_{j=1}^m \mathsf{k}_{m,j}}{\mathsf{k}_m}=1\\
  &=\sup_{j\in [m]} |\res_m|  \Bigl|\frac{2}{(j-\alpha)}+\res_m \Bigr|\\
  &\le |\res_m| \bigl(\frac{2}{(1-\alpha)}+|\res_m|\bigr)\\
  &\asconv 0 && \res_m \asconv 0.
\end{align*}
Thus, we get 
$$
 \frac{m}{\mathsf{k}_m} \E[X_{m+1}^2|\mathcal{F}_m] = \sum_{j=1}^{m} \frac{1}{(j-\alpha)} \frac{\mathsf{k}_{m,j}}{\mathsf{k}_m} + \frac{1}{\alpha} + o_{a.s.}(1). 
$$
Combined with the dominated convergence-type theorem in \Cref{lm:bounded_conv}, we get 
$$
 \frac{m}{\mathsf{k}_m} \E[X_{m+1}^2|\mathcal{F}_m] \asconv \sum_{j=1}^\infty \frac{1}{(j-\alpha)} \mathfrak{p}_\alpha(j) + \frac{1}{\alpha} = \alpha \mathfrak{i}(\alpha)
$$
where we have used the second expression in \Cref{lemma:fisher_info_sibuya} for the equality. 
Combined with $m^{-\alpha}\mathsf{k}_m\asconv \diversity$, we get
$$
\frac{\E[X_{m+1}^2|\mathcal{F}_m]}{\alpha m^{\alpha-1}}  \asconv  \mathfrak{i}(\alpha) \diversity.
$$
Let us show $n^{-\alpha} \sum_{m=0}^{n-1} \E[X_{m+1}^2|\mathcal{F}_m] - \mathfrak{i}(\alpha) \diversity  \asconv 0$ by standard algebra. Note 
\begin{align*}
  &\Bigl|n^{-\alpha} \sum_{m=0}^{n-1} \E[X_{m+1}^2|\mathcal{F}_m] - \mathfrak{i}(\alpha) \diversity\Bigr| \\
  &= \Bigl|\frac{1}{n^\alpha} \sum_{m=1}^{n-1} \alpha m^{\alpha-1} \frac{\E[X_{m+1}^2|\mathcal{F}_m]}{\alpha m^{\alpha-1}} - \mathfrak{i}(\alpha)\diversity\Bigr|\\
  &\le \frac{1}{n^\alpha} \sum_{m=1}^{n-1} \alpha m^{\alpha-1}\Bigl|
    \frac{\E[X_{m+1}^2|\mathcal{F}_m]}{\alpha m^{\alpha-1}} - \mathfrak{i}(\alpha) \diversity 
  \Bigr| + \Bigl| \frac{1}{n^\alpha} \sum_{m=1}^{n-1} \alpha m^{\alpha-1} - 1\Bigr| \mathfrak{i}(\alpha) \diversity
\end{align*}
The second term converges to $0$ almost surely by
$\frac{1}{n^\alpha} \sum_{m=1}^{n-1} \alpha m^{\alpha-1}\to 1$. 
For the first term, we know from $\frac{\E[X_{m+1}^2|\mathcal{F}_m]}{\alpha m^{\alpha-1}}  \asconv  \mathfrak{i}(\alpha) \diversity$ that with probability $1$, 
$$
\forall \epsilon>0, \quad \exists N_\epsilon (\omega) >0 \quad \sup_{m\ge N_{\epsilon}} \Bigl|
  \frac{\E[X_{m+1}^2|\mathcal{F}_m]}{\alpha m^{\alpha-1}} - \mathfrak{i}(\alpha) \diversity 
\Bigr| \le \epsilon
$$
Thus, for any $n\ge N_\epsilon(\omega)$,
 \begin{align*}
  &\frac{1}{n^\alpha} \sum_{m=1}^{n-1} \alpha m^{\alpha-1}\Bigl|
    \frac{\E[X_{m+1}^2|\mathcal{F}_m]}{\alpha m^{\alpha-1}} - \mathfrak{i}(\alpha) \diversity 
  \Bigr| \\
  &\le \frac{1}{n^\alpha} \underbrace{\sum_{m=1}^{N_{\epsilon}} \alpha m^{\alpha-1} \Bigl|
    \frac{\E[X_{m+1}^2|\mathcal{F}_m]}{\alpha m^{\alpha-1}} - \mathfrak{i}(\alpha) \diversity 
  \Bigr|}_{\equiv M_\epsilon(\omega)}  + \frac{1}{n^\alpha} \sum_{m=N_{\epsilon}}^{n-1} \alpha m^{\alpha-1} \epsilon\\
  &\le \frac{1}{n^\alpha} M_\epsilon(\omega) + \frac{1}{n^\alpha} \sum_{m=1}^{n-1} \alpha m^{\alpha-1} \epsilon.
 \end{align*}
 Combined with $\frac{1}{n^\alpha} \sum_{m=1}^{n-1} \alpha m^{\alpha-1}\to 1$, there exists a random integer $N'_\epsilon(\omega)$ such that 
$$
\forall n\ge  N'_\epsilon(\omega), \quad 
\frac{1}{n^\alpha} \sum_{m=1}^{n-1} \alpha m^{\alpha-1}\Bigl|
  \frac{\E[X_{m+1}^2|\mathcal{F}_m]}{\alpha m^{\alpha-1}} - \mathfrak{i}(\alpha) \diversity 
\Bigr| \le \epsilon + (1+1)\epsilon = 3\epsilon
$$
Since $\epsilon$ is taken arbitrary, we get 
$$
\frac{1}{n^\alpha} \sum_{m=1}^{n-1} \alpha m^{\alpha-1}\Bigl|
  \frac{\E[X_{m+1}^2|\mathcal{F}_m]}{\alpha m^{\alpha-1}} - \mathfrak{i}(\alpha) \diversity 
\Bigr|  \asconv 0
$$
This finishes the proof.
\end{proof}

\begin{theorem}\label{thm:score_clt}
The score function has the asymptotic mixed normality:
  \begin{align*}
    \frac{\partial_\alpha \ell_n(\alpha;\mu)}{\sqrt{n^\alpha \mathfrak{i}(\alpha)}} \to \diversity^{1/2} N,  \quad \stable
  \end{align*}
\end{theorem}
\begin{proof}
  By \Cref{lm:martingale_diff} and \Cref{lm:quadratic_variation}, the martingale difference of the score function is bounded, and the limit of the quadratic variation is given by $n^{-\alpha} \sum_{m=0}^{n-1} \E[X_{m+1}^2|\mathcal{F}_{m}] \to^p \mathfrak{i}(\alpha) \diversity$. Then the claim follows from \cite[Theorem 6.23]{hausler2015stable}. 
\end{proof}

\subsection{Proof of \Cref{thm:mle_clt}}
Recall $\Psi_n(\alpha)=\mathsf{k}_n^{-1} \cdot \partial_\alpha \ell_n(\alpha;\delta_0)$.
Let us consider the event $\{1<\mathsf{k}_n<n\}$, which holds with high probability by \Cref{prop:mle_existence}. Under this event, $\hat{\alpha}_n$ satisfies 
$$
{\Psi}_n(\hat{\alpha}_n) = 0. 
$$
Furthermore, by the consistency $\hat\alpha_n\to^p\alpha$ from \Cref{cor:consistency_and_uniform_conv}, we have $\hat{\alpha}_n \in I$ for $I=[\alpha/2,  (1+\alpha)/2] \subset (0,1)$ with high probability. Then, by Taylor's expansion,  there exists a random sequence $\tilde{\alpha}_n \to^p \alpha$ such that 
\begin{align*}
  \Psi_n(\alpha) = \Psi_n(\alpha) - \Psi_n(\hat\alpha_n) = -\Psi_{n}'(\tilde{\alpha}_n) (\alpha-\hat{\alpha}_n), 
\end{align*}
and  $ -\Psi_{n}'(\tilde\alpha_n)\to^p \mathfrak{i}(\alpha)$ by \Cref{cor:consistency_and_uniform_conv}. On the other hand, by \eqref{eq:psi_n_def} and \eqref{eq:score}, the LHS can be written as 
\begin{align*}
  \Psi_n(\alpha) = \frac{1}{\mathsf{k}_n} \Bigl(\frac{\mathsf{k}_n-1}{ \alpha} - \sum_{j=2}^n {\mathsf{k}_{n,j}}\sum_{i=1}^{j-1}\frac{1}{i-\alpha}\Bigr) = \frac{1}{\mathsf{k}_n} \Bigl({\partial_\alpha \ell_n(\alpha; \mu)} + \int \dd \mu_{n,\mathsf{k}_n}(\theta)  \sum_{i=1}^{\mathsf{k}_n-1}\frac{\theta}{\alpha(\theta+i\alpha)}\Bigr), 
\end{align*}
where $\int \dd \mu_{n,\mathsf{k}_n}(\theta)  \sum_{i=1}^{\mathsf{k}_n-1}\frac{\theta}{\alpha(\theta+i\alpha)} = O(\log \mathsf{k}_n)$ by $\operatorname{supp}(\mu_{n,\mathsf{k}_n})\subset [\underline\theta, \bar\theta]$. Putting the above displays together, we obtain
\begin{align*}
  n^{\alpha/2} (-\Psi'_n(\tilde{\alpha}_n)) (\alpha-\tilde{\alpha}_n) &= n^{\alpha/2} \Psi_n(\alpha)\\
  &= \frac{n^{\alpha}}{\mathsf{k}_n} \frac{\partial_\alpha \ell_n(\alpha;\mu)}{n^{\alpha/2}} + \frac{n^{\alpha/2}}{\mathsf{k}_n} \int \dd \mu_{n,\mathsf{k}_n}(\theta)  \sum_{i=1}^{\mathsf{k}_n-1}\frac{\theta}{\alpha(\theta+i\alpha)}\\
  &= \frac{n^{\alpha}}{\mathsf{k}_n} \frac{\partial_\alpha \ell_n(\alpha;\mu)}{n^{\alpha/2}} + O(n^{\alpha/2} \mathsf{k}_n^{-1}\log \mathsf{k}_n),
\end{align*}
where $O(n^{\alpha/2} \mathsf{k}_n^{-1}\log \mathsf{k}_n)\asconv 0$ by $\mathsf{k}_n/n^\alpha \asconv \diversity>0$. Now, \Cref{thm:Lp_converge} implies 
\begin{align*}
  \frac{\partial_\alpha \ell_n(\alpha;\mu)}{\sqrt{n^\alpha \mathfrak{i}(\alpha)}} \to \diversity^{1/2} N \quad \stable.
\end{align*}
Combined with $\mathsf{k}_n/n^\alpha \asconv \diversity$ and Slutsky's lemma for the stable convergence (see \Cref{lm:CS_stable}\editline{-(2)}), we obtain 
\begin{align*}
  n^{\alpha/2} (-\Psi'_n(\tilde{\alpha}_n)) (\alpha-\tilde{\alpha}_n) \to \frac{1}{\diversity} \sqrt{\diversity \cdot \mathfrak{i}(\alpha)} N =  \sqrt{\frac{\mathfrak{i}(\alpha)}{\diversity}} N \quad  \stable.
\end{align*}
Combined with $ -\Psi_{n}'(\tilde\alpha_n)\to^p \mathfrak{i}(\alpha)$, we complete the proof. 

\begin{lemma}\label{lemma:qmle_bounded_l2}
  $n^{-\alpha/2}\mathsf{k}_n|\hat\alpha_n-\alpha|$ is bounded in $\mathcal{L}_2$. 
\end{lemma}
\begin{proof}
\eqref{eq:qmle_cases} implies $\hat{\alpha}_n=0$ if $\mathsf{k}_n=1$ and $\hat{\alpha}_n=1$ if $\mathsf{k}_n=n$. Then, we can decompose the second moment $\E\bigl[\mathsf{k}_n^2|\hat\alpha_n-\alpha|^2\bigr]$ into three terms:
  \begin{align*}
      &n^{-\alpha} \E\bigl[\mathsf{k}_n^2|\hat\alpha_n-\alpha|^2\bigr] \\
      &= n^{-\alpha}\E\Bigl[\mathsf{k}_n^2|\hat\alpha_n-\alpha|^2 (\mathbbm{1}\{\mathsf{k}_n=1\} + \mathbbm{1}\{\mathsf{k}_n=n\}+  \mathbbm{1}\{1<\mathsf{k}_n<n\})\Bigr]\\
      &= n^{-\alpha} \alpha^2 \PP(\mathsf{k}_n=1) + 
      n^{2-\alpha} (1-\alpha)^2 \PP(\mathsf{k}_n=n) + n^{-\alpha} \E\bigl[\mathsf{k}_n^2|\hat\alpha_n-\alpha|^2 \mathbbm{1}\{1<\mathsf{k}_n<n\}\bigr]
  \end{align*}
Here, \Cref{prop:mle_existence} implies 
\begin{align*}
\PP(\mathsf{k}_n = 1) = o(1), \quad 
\PP(\mathsf{k}_n = n) \lesssim n^{\bar{\theta}(\alpha^{-1} - 1)} \alpha^n = o(n^{\alpha-2})
\end{align*}
so that 
\begin{align*}
 n^{-\alpha} \E\bigl[\mathsf{k}_n^2|\hat\alpha_n-\alpha|^2\bigr] &= n^{-\alpha} \E\bigl[\mathsf{k}_n^2|\hat\alpha_n-\alpha|^2 \mathbbm{1}\{1<\mathsf{k}_n<n\}\bigr] \\
 &+ \editline{ n^{-\alpha} \E\bigl[|\hat\alpha_n-\alpha|^2 \mathbbm{1}\{\mathsf{k}_n=1\}\bigr] + n^{2-\alpha} \E\bigl[|\hat\alpha_n-\alpha|^2 \mathbbm{1}\{\mathsf{k}_n=n\}\bigr]},  
\end{align*}
\editline{Here, noting $|\hat\alpha_n-\alpha|^2 \le 1$, the second and the third terms are $o(1)$ by $\PP(\mathsf{k}_n = 1) = o(1)$ and $\PP(\mathsf{k}_n = n) = o(n^{\alpha-2})$, respectively. }
Under the event $\{1<\mathsf{k}_n<n\}$, by the calculation right before \Cref{lemma:qmle_bounded_l2},  there exists a random sequence $\tilde{\alpha}_n\in(0,1)$ such that 
  $$
  \mathsf{k}_n (-\Psi_n'(\tilde\alpha_n)) (\alpha-\hat\alpha_n) = \partial_\alpha\ell_n(\alpha;\mu) + \int \dd \mu_{n,\mathsf{k}_n}(\theta)  \sum_{i=1}^{\mathsf{k}_n-1}\frac{\theta}{\alpha(\theta+i\alpha)}. 
  $$ 
  Here, the absolute value of the second term is bounded by $C \log n$ up to some constant $C=C(\alpha, \underline\theta, \bar\theta)$. Using the uniform lower bound $\inf_{x\in (0,1)} (-\Psi_n'(x))\ge 1/2$ from \Cref{lm:psi_n_property}, we get 
  $$
  \mathbbm{1}\{1<\mathsf{k}_n<n\} \mathsf{k}_n|\alpha-\hat\alpha_n| \le 2\bigl(|\partial_\alpha\ell_n(\alpha;\mu)| + C \log n\bigr).
  $$
  Thus, combined with $\E[\partial_\alpha\ell_n(\alpha;\mu)^2]\asymp n^{\alpha}$ from \Cref{theorem:fisher}, we obtain
  $$
  n^{-\alpha}
\E\bigl[\mathsf{k}_n^2|\hat\alpha_n-\alpha|^2 \mathbbm{1}\{1<\mathsf{k}_n<n\}\bigr] \le 8 n^{-\alpha} \Bigl(\E\bigl[
  |\partial_\alpha\ell_n(\alpha;\mu)|^2 \bigr]+ C^2 (\log n)^2\Bigr)
 = O(1). 
  $$
  This completes the proof. 
\end{proof}

\section{Limit distribution of f-Divergence}
\subsection{Proof of \Cref{thm:f_div}}
Let $f$ be a convex function such that it is locally twice differentiable and its second derivative is Hölder continuous at $1$. That is, there exists $\delta > 0, C > 0$, and $\beta > 0$ such that $f$ is twice differentiable on $(1 - \delta, 1 + \delta)$, and
$|f''(1 + x) - f''(1)| \le C |x|^\beta$ for all $|x| < \delta$. Now we consider the following event 
$$
\Omega_{n} \equiv \Bigl\{(2\alpha-1)\vee \frac{\alpha}{2} \le \hat\alpha_n \le \frac{\alpha+1}{2} \wedge 2\alpha \Bigr\}
$$
By the consistency $\hat\alpha_n\to^p\alpha$ and $\alpha\in (0,1)$, the event $\Omega_n$ holds with high probability. Furthermore, under $\Omega_{n}$, it can be easily checked that 
\begin{align*}
  \frac{1}{2} \le  \frac{\hat\alpha_n}{\alpha} \le 2, \quad \frac{1}{2}\le \frac{|U_i|-\hat\alpha_n}{|U_i|-\alpha} \le 2 \quad \text{for all $i\in \{1, \dots \mathsf{k}_n\}$}. 
\end{align*}
Throughout this proof, we denote $A_n \lesssim B_n$ whenever there exists a deterministic constant $C=C(\alpha, \underline\theta, \bar\theta, f)$ such that $A_n\le C \cdot B_n$. 
Note 
\begin{align*}
  \frac{1}{\simp_{n,0}} = \frac{v_{n,\mathsf{k}_n}}{v_{n+1, \mathsf{k}_n+1}} =
   \int \diff \mu_{n+1, \mathsf{k}_n+1} (\theta)  \frac{\theta+ n}{\theta+\mathsf{k}_n\alpha}.
\end{align*}
Using $\operatorname{supp}(\mu_{n+1, \mathsf{k}_n+1})\subset [\underline{\theta}, \bar\theta]$ with $-\alpha < \underline{\theta} < \bar\theta < +\infty$, we have 
\begin{align}
  \Bigl|\frac{1}{\simp_{n,0}} - \frac{n}{\mathsf{k}_n\alpha}\Bigr| &\le \sup_{\theta\in [\underline{\theta}, \bar\theta]}\Bigl|\frac{\theta+ n}{\theta+\mathsf{k}_n\alpha} - \frac{n}{\mathsf{k}_n\alpha}\Bigr|  \lesssim \frac{n}{\mathsf{k}_n^2} 
  \label{eq:p0_1}\\
  \Bigl|\frac{\hat\simp_{n,0}}{\simp_{n,0}} - \frac{\hat\alpha_n}{\alpha}\Bigr| &\lesssim \frac{1}{\mathsf{k}_n} && \text{by $\hat\simp_{n,0}=\frac{\mathsf{k}_n\hat\alpha_n}{n}$ and \eqref{eq:p0_1}}\label{eq:p0_2}
  \\
  \Bigl|\frac{\hat\simp_{n,0}}{\simp_{n,0}} - 1\Bigr| &\lesssim \frac{1}{\mathsf{k}_n} + |\hat\alpha_n-\alpha| && \text{by \eqref{eq:p0_2}}. \label{eq:p0_3}
\end{align}
By the same argument, we have 
\begin{align*}
  \frac{1}{\simp_{n,i}} = \frac{v_{n,\mathsf{k}_n}}{v_{n+1, \mathsf{k}_n}} \frac{1}{|U_i|-\alpha} =\frac{1}{|U_i|-\alpha} \int \diff \mu_{n+1, \mathsf{k}_n}(\theta) (\theta + n),
\end{align*}
which leads to 
\begin{align}
  \Bigl|
  \frac{1}{\simp_{n,i}} - \frac{n}{|U_i|-\alpha} 
\Bigr| &\lesssim \frac{1}{|U_i|-\alpha} \label{eq:pi_1}\\
\sup_{i\in\{1, \dots, \mathsf{k}_n\}}\Bigl|\frac{\hat\simp_{n,i}}{\simp_{n,i}} -  \frac{|U_i|-\hat\alpha_n}{|U_i|-\alpha}\Bigr| &\lesssim \frac{1}{n} && \text{by \eqref{eq:pi_1}, $\hat\simp_{n,i} = \frac{|U_i|-\hat\alpha_n}{n}$, and $\frac{|U_i|-\hat\alpha_n}{|U_i|-\alpha}\le 2$ under $\Omega_{n}$}
\label{eq:pi_2}\\
  \sup_{i\in\{1, \dots, \mathsf{k}_n\}}\Bigl|\frac{\hat\simp_{n,i}}{\simp_{n,i}} - 1\Bigr| &\lesssim \frac{1}{n} + |\hat\alpha_n - \alpha| && \text{by \eqref{eq:pi_2} and $|U_i| \ge 1$} \label{eq:pi_3}
\end{align}
Combining \eqref{eq:p0_3} and \eqref{eq:pi_3}, noting $\mathsf{k}_n^{-1}=O_p(n^{-\alpha})$ and $(\hat\alpha_n-\alpha)=O_p(n^{-\alpha/2})$, we obtain
\begin{align}
\sup_{i=\in\{0, 1, \dots, \mathsf{k}_n\}}  \Bigl|\frac{\hat\simp_{n,i}}{\simp_{n,i}} - 1\Bigr| \lesssim \frac{1}{\mathsf{k}_n} + \frac{1}{n} + |\hat\alpha_n-\alpha| = O_p(n^{-\alpha/2}).\label{eq:pi_3_uniform_bound}   
\end{align}
By Taylor's theorem of $f(x)$ around $x=1$ up to the second order term, we have 
\begin{align*}
    \mathsf{D}_f(\bm{\hat\simp}_n || \bm{\simp}_n) = \sum_{i=1}^{\mathsf{k}_n}\simp_{n,i} f(\hat\simp_{n,i}/{\simp}_{n,i}) = \sum_{i=1}^{\mathsf{k}_n}\simp_{n,i} f(1) + \simp_{n,i} f'(1)\Bigl(\frac{\hat\simp_{n,i}}{{\simp}_{n,i}}-1\Bigr) + \simp_{n,i} \frac{f''(1+c_{n,i})}{2}\Bigl(\frac{\hat\simp_{n,i}}{{\simp}_{n,i}}-1\Bigr)^2.
\end{align*}
Here $c_{n,i}$ is a random variable between $1$ and $\hat\simp_{n,i}/\simp_{n,i}$. By \eqref{eq:pi_3_uniform_bound},  $|c_{n,i}|$ is uniformly bounded as 
\begin{align}\label{eq:c_n_uniform_bound}
\sup_{i=\in\{0, 1, \dots, \mathsf{k}_n\}}  |c_{n,i}|\le \sup_{i=\in\{0, 1, \dots, \mathsf{k}_n\}}   \Bigl|\frac{\hat\simp_{n,i}}{\simp_{n,i}} - 1\Bigr| = O_p(n^{-\alpha/2})  
\end{align}
Substituting $f(1)=0$ and $\sum_{i=0}^{\mathsf{k}_n} \simp_{n,i}=\sum_{i=0}^{\mathsf{k}_n} \hat\simp_{n,i}=1$ to the previous display of $\mathsf{D}_f(\bm{\hat\simp}_n || \bm{\simp}_n)$, we are left with
$$
\mathsf{D}_f(\bm{\hat\simp}_n || \bm{\simp}_n) = \sum_{i=0}^{\mathsf{k}_n} \simp_{n,i} \frac{f''(1+c_{n,i})}{2}\Bigl(\frac{\hat\simp_{n,i}}{{\simp}_{n,i}}-1\Bigr)^2.
$$
By the local Hölder continuity of $f''$ around $1$, we have $|f''(1 + x) - f''(1)| \le C |x|^\beta$ for all $|x| < \delta$ where $\delta, \beta, C$ are fixed positive constant. Then, under the event $\{\sup_{i}|c_{n,i}| \le \delta\}$, which holds with high probability by
\eqref{eq:c_n_uniform_bound}, it holds that 
\begin{align*}
\Bigl|
  &\mathsf{D}_f(\bm{\hat\simp}_n || \bm{\simp}_n) - \frac{f''(1)}{2} \sum_{i=0}^{\mathsf{k}_n} \simp_{n,i} \Bigl(\frac{\hat\simp_{n,i}}{{\simp}_{n,i}}-1\Bigr)^2
\Bigr| \\
&\le  \frac{1}{2} \sup_{i\in\{0,\dots \mathsf{k}_n\}} |f''(1+c_{n,i})-f''(1)| \cdot 
\sum_{i=0}^{\mathsf{k}_n} \simp_{n,i} \Bigl(\frac{\hat\simp_{n,i}}{{\simp}_{n,i}}-1\Bigr)^2\\
&\lesssim \Bigl(\sup_{i\in \{0\dots, \mathsf{k}_n\}
} |c_{n,i}|\Bigr)^\beta \sum_{i=0}^{\mathsf{k}_n}\simp_{n,i}\Bigl(\frac{\hat\simp_{n,i}}{{\simp}_{n,i}}-1\Bigr)^2\\
&= O_p(n^{-\alpha\beta/2}) \sum_{i=0}^{\mathsf{k}_n}\simp_{n,i}\Bigl(\frac{\hat\simp_{n,i}}{{\simp}_{n,i}}-1\Bigr)^2 && \text{by \eqref{eq:c_n_uniform_bound}}. 
\end{align*}
Rearranging the above display, we are left with 
$$
\mathsf{D}_f(\bm{\hat\simp}_n || \bm{\simp}_n) = \Bigl(
\frac{f''(1)}{2} + O_p(n^{-\alpha\beta/2})
\Bigr) \sum_{i=0}^{\mathsf{k}_n}\simp_{n,i} \Bigl(\frac{\hat\simp_{n,i}}{{\simp}_{n,i}}-1\Bigr)^2
$$
Next, let us derive the limit of $\sum_{i=0}^{\mathsf{k}_n}\simp_{n,i}\Bigl(\frac{\hat\simp_{n,i}}{{\simp}_{n,i}}-1\Bigr)^2$. Note 
\begin{align*}
  \sum_{i=0}^{\mathsf{k}_n}\simp_{n,i}\Bigl(\frac{\hat\simp_{n,i}}{{\simp}_{n,i}}-1\Bigr)^2 = \frac{(\hat\simp_{n,0}-\simp_{n,0})^2}{\simp_{n,0}} + \sum_{i=1}^{\mathsf{k}_n}\frac{(\hat\simp_{n,i}-\simp_{n,i})^2}{\simp_{n,i}}
\end{align*}
Now we define $\Delta_{n,0}$ and $\delta_{n,0}$ as 
$$
\Delta_{n,0}\equiv \frac{\mathsf{k}_n^2}{n}\Bigl(\frac{1}{\simp_{n,0}} - \frac{n}{\mathsf{k}_n\alpha}\Bigr), \quad \delta_{n,0}\equiv n\Bigl(|\simp_{n,0} - \hat\simp_{n,0}| - \frac{\mathsf{k}_n}{n} |\alpha-\hat\alpha_n|\Bigr)
$$
so that 
$$
  \frac{(\hat\simp_{n,0}-\simp_{n,0})^2}{\simp_{n,0}} = \Bigl(\frac{n}{\mathsf{k}_n\alpha} + \frac{n}{\mathsf{k}_n^2}\Delta_{n,0}\Bigr) \cdot \Bigl(
\frac{\mathsf{k}_n}{n} |\alpha-\hat\alpha_n| + \frac{\delta_{n,0}}{n}
  \Bigr)^2
$$
Here, 
\eqref{eq:p0_1} and \eqref{eq:delta_0_bound} yield  
$|\Delta_{n,0}|  \lesssim 1$ and $|\delta_{n,0}|\lesssim 1$, and hence 
$$
\frac{(\hat\simp_{n,0}-\simp_{n,0})^2}{\simp_{n,0}}
= \Bigl(\frac{n}{\mathsf{k}_n\alpha} + O_p(n^{1-2\alpha})\Bigr) \cdot \Bigl(
\frac{\mathsf{k}_n}{n} |\alpha-\hat\alpha_n| + O_p(n^{-1})\Bigr)^2. 
$$
Expanding the square, using
$n/\mathsf{k}_n = O_p(n^{1-\alpha})$ and $\mathsf{k}_n |\alpha-\hat\alpha_n|/n = O_p(n^{\alpha-\alpha/2 - 1})= O_p(n^{\alpha/2 - 1})$, we get 
\begin{align*}
  \frac{(\hat\simp_{n,0}-\simp_{n,0})^2}{\simp_{n,0}}
&= \Bigl(\frac{n}{\mathsf{k}_n\alpha} + O_p(n^{1-2\alpha})\Bigr) \cdot \Bigl(
\frac{\mathsf{k}_n^2}{n^2} |\alpha-\hat\alpha_n|^2 + O_p(n^{\alpha/2-2})
  \Bigr)\\
    &= \frac{\mathsf{k}_n}{n\alpha} |\alpha-\hat\alpha_n|^2 + O_p(n^{1-\alpha})O_p(n^{\alpha/2-2}) + O_p(n^{1-2\alpha}) O_p(n^{\alpha-2}) + O_p(n^{1-2\alpha + \alpha/2 -2})\\
  &= \frac{\mathsf{k}_n}{n\alpha} |\alpha-\hat\alpha_n|^2 + O_p(n^{-1-\alpha/2}). 
\end{align*}

Now define $\Delta_{n,i}$ and $\delta_{n,i}$ for each $i\in \{1, \dots, \mathsf{k}_n\}$ as
\begin{align*}
  \Delta_{n,i} \equiv \frac{1}{\simp_{n,i}} - \frac{n}{|U_i|-\alpha}, \quad 
  \delta_{n,i} \equiv \frac{n^2}{|U_i|-\alpha}\Bigl(|\hat\simp_{n,i}-\simp_{n,i}| - \frac{|\hat\alpha_n-\alpha|}{n}\Bigr)
\end{align*}
so that 
\begin{align*}
\frac{(\hat\simp_{n,i}-\simp_{n,i})^2}{\simp_{n,i}} &= \Bigl(\frac{n}{|U_i|-\alpha} + \Delta_{n,i}\Bigr) \Bigl(
  \frac{|\hat\alpha_n-\alpha|}{n} + \frac{|U_i|-\alpha}{n^2}\delta_{n,i}
  \Bigr)^2\\
  &= \frac{|\hat\alpha_n-\alpha|^2}{\editlinenew{n}(|U_i|-\alpha)} + \frac{|U_i|-\alpha}{n^3}\delta_{n,i}^2 +  2  \frac{|\hat\alpha_n-\alpha|}{n^2} \delta_{n,i} + \Delta_{n,i}\Bigl(
  \frac{|\hat\alpha_n-\alpha|}{n} + \frac{|U_i|-\alpha}{n^2}\delta_{n,i}
  \Bigr)^2
\end{align*}
By \eqref{eq:pi_1} and \eqref{eq:delta_i_bound}, it holds that $\sup_{i\in \{1, \dots, \mathsf{k}_n\}}|\Delta_{n,i}|\lesssim 1$ and $\sup_{i\in \{1, \dots, \mathsf{k}_n\}}|\delta_{n,i}|\lesssim 1$. Therefore, we get 
\begin{align*}
  &\sum_{i=1}^{\mathsf{k}_n} \Bigl|\frac{(\hat\simp_{n,i}-\simp_{n,i})^2}{\simp_{n,i}} - \frac{|\hat\alpha_n-\alpha|^2}{\editlinenew{n}(|U_i|-\alpha)} \Bigr| \\
  &\lesssim \sum_{i=1}^{\mathsf{k}_n} \Bigl(\frac{|U_i|}{n^3} + \frac{|\hat\alpha_n-\alpha|}{n^2} + \frac{|\hat\alpha_n-\alpha|^2}{n^2} + \frac{|U_i|^2}{n^4}\Bigr)\\
  & \lesssim \sum_{i=1}^{\mathsf{k}_n}  \Bigl( \frac{|U_i|}{n^3} + \frac{|\hat\alpha_n-\alpha|}{n^2}\Bigr) && \text{by $|U_i|\le n$. }\\
  &=\frac{1}{n^2} + \frac{|\hat\alpha_n-\alpha|\mathsf{k}_n}{n^2} && \text{by $\sum_i |U_i|=n$}\\
  &= O_p(n^{-2}) + O_p(n^{\alpha/2-2}) && \text{by $(\hat\alpha_n-\alpha)=O_p(n^{-\alpha/2})$ and $\mathsf{k}_n=O_p(n^\alpha)$}. 
\end{align*}
Thus, 
\begin{align*}
  \sum_{i=1}^{\mathsf{k}_n}\frac{(\hat\simp_{n,i}-\simp_{n,i})^2}{\simp_{n,i}} &= \sum_{i=1}^{\mathsf{k}_n} \frac{|\hat\alpha_n-\alpha|^2}{\editlinenew{n}(|U_i|-\alpha)} + O_p(n^{\alpha/2-2}) = \sum_{j=1}^n \mathsf{k}_{n,j} \frac{|\hat\alpha_n-\alpha|^2}{\editlinenew{n}(j-\alpha)} +  O_p(n^{\alpha/2-2})
\end{align*}
Combining them all together, 
\begin{align*}
\sum_{i=0}^{\mathsf{k}_n}\frac{(\hat\simp_{n,i}-\simp_{n,i})^2}{\simp_{n,i}} &= \frac{\mathsf{k}_n}{n\alpha} |\alpha-\hat\alpha_n|^2  + \sum_{j=1}^n \mathsf{k}_{n,j} \frac{|\hat\alpha_n-\alpha|^2}{\editlinenew{n}(j-\alpha)} + O_p(n^{-1-\alpha/2}) + O_p(n^{-2+\alpha/2})  \\
&= \frac{\alpha \mathsf{k}_n}{n}|\hat\alpha_n-\alpha|^2 \Bigl(
  \frac{1}{\alpha^2} + \sum_{j=1}^{n} \frac{\mathsf{k}_{n,j}}{\mathsf{k}_n} \frac{1}{\editlinenew{\alpha}(j-\alpha)}
\Bigr) + O_p(n^{-1-\alpha/2}) && \text{by $\alpha\in(0,1)$}
\end{align*}
and hence 
$$
\mathsf{D}_f(\bm{\hat\simp}_n || \bm{\simp}_n) = \Bigl(
\frac{f''(1)}{2} + O_p(n^{-\alpha\beta/2})
\Bigr) \Bigl(\frac{\alpha \mathsf{k}_n}{n}|\hat\alpha_n-\alpha|^2 \Bigl(
  \frac{1}{\alpha^2} + \sum_{j=1}^{n} \frac{\mathsf{k}_{n,j}}{\mathsf{k}_n} \frac{1}{\editlinenew{\alpha}(j-\alpha)}
\Bigr) + O_p(n^{-1-\alpha/2})\Bigr)
$$
By \Cref{lm:bounded_conv} and the second expression in \Cref{lemma:fisher_info_sibuya}, 
$$
  \frac{1}{\alpha^2} + \sum_{j=1}^{n} \frac{\mathsf{k}_{n,j}}{\mathsf{k}_n} \frac{1}{\editlinenew{\alpha}(j-\alpha)} \to^p \frac{1}{\alpha^2} + \sum_{j=1}^\infty \mathfrak{p}_\alpha(j) \frac{1}{\editlinenew{\alpha}(j-\alpha)} = \mathfrak{i}(\alpha)
$$
while $\sqrt{\mathsf{k}_n \mathfrak{i}(\alpha)} (\hat\alpha_n-\alpha) \to N$ $\stable$ by \Cref{thm:mle_clt}. By the continuous mapping theorem (\editline{\Cref{lm:CS_stable}{-(3)}}) \editline{for the function $g:x\mapsto|x|$} and Slutsky's lemma from \Cref{lm:CS_stable}\editline{-(2)}, we have 
$$
 \mathsf{k}_n|\hat\alpha_n-\alpha|^2 \Bigl(
  \frac{1}{\alpha^2} + \sum_{j=1}^{n} \frac{\mathsf{k}_{n,j}}{\mathsf{k}_n} \frac{1}{\editlinenew{\alpha}(j-\alpha)}
\Bigr) \to |N|^2 \mathfrak{i}(\alpha)^{-1}  \mathfrak{i}(\alpha) = |N|^2 \quad \stable
$$
where $N\sim \mathcal{N}(0,1) \indep \mathcal{F}_\infty$.  
This in particular means that the LHS is $O_p(1)$. 
Putting the above displays together, we obtain 
\begin{align*}
&n \cdot \mathsf{D}_f(\bm{\hat\simp}_n || \bm{\simp}_n) \\
&= n\cdot \Bigl(
\frac{f''(1)}{2} + O_p(n^{-\alpha\beta/2})
\Bigr) \Bigl(\frac{\alpha \mathsf{k}_n}{n}|\hat\alpha_n-\alpha|^2 \Bigl(
  \frac{1}{\alpha^2} + \sum_{j=1}^{n} \frac{\mathsf{k}_{n,j}}{\mathsf{k}_n} \frac{1}{\editlinenew{\alpha}(j-\alpha)}
\Bigr) + O_p(n^{-1-\alpha/2})\Bigr)  \\
&= \Bigl(
\frac{f''(1)}{2} + O_p(n^{-\alpha\beta/2})
\Bigr) \Bigl({\alpha \mathsf{k}_n}|\hat\alpha_n-\alpha|^2 \Bigl(
  \frac{1}{\alpha^2} + \sum_{j=1}^{n} \frac{\mathsf{k}_{n,j}}{\mathsf{k}_n} \frac{1}{\editlinenew{\alpha}(j-\alpha)}
\Bigr) + O_p(n^{-\alpha/2})\Bigr) \\
&= \frac{f''(1)}{2} {\alpha \mathsf{k}_n} |\hat\alpha_n-\alpha|^2 \Bigl(
  \frac{1}{\alpha^2} + \sum_{j=1}^{n} \frac{\mathsf{k}_{n,j}}{\mathsf{k}_n} \frac{1}{\editlinenew{\alpha}(j-\alpha)}
\Bigr)  + O_p(n^{- \alpha\beta/2}) + O_p(n^{-\alpha/2})\\
&\to \frac{f''(1)}{2}\alpha N^2  \quad \stable,  
\end{align*}
which completes the proof.

\subsection{Proof of \Cref{theorem:tv}}
\Cref{theorem:tv} follows from \Cref{thm:mle_clt}, \Cref{lemma:qmle_bounded_l2}, and \Cref{prop:tv_decompose} below. 
  \begin{lemma}\label{prop:tv_decompose}
  \begin{align*}
    \Bigl|\tv(\bm\simp_n, \bm{\hat\simp}_n) - \frac{\mathsf{k}_n}{n}|\hat\alpha_n - \alpha|\Bigr|\le  \frac{\bar{\theta}-\underline{\theta}}{n+\underline{\theta}} 
  \end{align*}
\end{lemma}
\begin{proof}[Proof of \Cref{prop:tv_decompose}]
  The TV distance can be decomposed into two terms:
$$
\tv(\bm{\hat\simp}_n, \bm\simp_n)  = \frac{1}{2}|\hat\simp_{n,0}-\simp_{n,0}| + \frac{1}{2}\sum_{i=1}^{\mathsf{k}_n} |\hat\simp_{n,i}-\simp_{n,i}|, 
$$
where
\begin{align*}
  |\simp_{n,0} - \hat\simp_{n,0}| &= \Bigl|\frac{v_{n+1, \mathsf{k}_{n}+1}}{v_{n, \mathsf{k}_n}}-\frac{\mathsf{k}_n}{n}\hat\alpha_n\Bigr| =  \Bigl|\frac{v_{n+1, \mathsf{k}_{n}+1}}{v_{n, \mathsf{k}_n}} - \frac{\mathsf{k}_n}{n}\alpha  + \frac{\mathsf{k}_n}{n}(\alpha -\hat\alpha_n)\Bigr|
\end{align*}
Combined with $|\frac{v_{n+1, \mathsf{k}_{n}+1}}{v_{n, \mathsf{k}_n}} - \frac{\mathsf{k}_n}{n}\alpha| \le \frac{\bar{\theta}-\underline\theta}{n+\underline{\theta}}$ from \Cref{lm:concentrate_v_nk}, we have 
\begin{align}\label{eq:delta_0_bound}
  \Bigl|  |\simp_{n,0} - \hat\simp_{n,0}| - \frac{\mathsf{k}_n}{n} |\alpha-\hat\alpha_n| \Bigr|\le \Bigl|\frac{v_{n+1, \mathsf{k}_{n}+1}}{v_{n, \mathsf{k}_n}} - \frac{\mathsf{k}_n}{n}\alpha\Bigr| \le \frac{\bar{\theta}-\underline\theta}{n+\underline{\theta}}. 
\end{align}
As for the second term$\sum_{i=1}^{\mathsf{k}_n} |\simp_{n,i} - \hat\simp_{n,i}|$. Note
\begin{align*}
  |\simp_{n,i} - \hat\simp_{n,i}| &=  \Bigl|
    \frac{v_{n+1, \mathsf{k}_n}}{v_{n,\mathsf{k}_n}}(|U_i|-\alpha) - \frac{|U_i|-\hat{\alpha}_n}{n}
  \Bigr| = \Bigl|\Bigl(\frac{v_{n+1, \mathsf{k}_n}}{v_{n, \mathsf{k}_n}} -\frac{1}{n}\Bigr) (|U_i|-\alpha) + \frac{\hat\alpha_n-\alpha}{n}\Bigr|.
\end{align*}
Using $|\frac{v_{n+1, \editline{\mathsf{k}_n}}}{v_{n, \mathsf{k}_n}} -\frac{1}{n}| \le \frac{\bar{\theta}-\underline{\theta}}{n(n+\underline{\theta})}$ from \Cref{lm:concentrate_v_nk}
\begin{align}
  \Bigl| |\simp_{n,i} - \hat\simp_{n,i}| - \frac{|\hat\alpha_n-\alpha|}{n} \Bigr|
  \le \Bigl|\frac{v_{n+1, \mathsf{k}_n}}{v_{n, \mathsf{k}_n}} -\frac{1}{n}\Bigr|(|U_i|-\alpha) \le \frac{\bar{\theta}-\underline{\theta}}{n(n+\underline{\theta})} \cdot (|U_i|-\alpha)\label{eq:delta_i_bound}
\end{align}
Therefore, we get 
\begin{align*}
  \Bigl|\sum_{i=1}^{\mathsf{k}_n} |\simp_{n,i} - \hat\simp_{n,i}| - \frac{\mathsf{k}_n}{n}|\hat\alpha_n-\alpha|\Bigr| &\le \sum_{i=1}^{\mathsf{k}_n}   \Bigl| |\simp_{n,i} - \hat\simp_{n,i}| - \frac{|\hat\alpha_n-\alpha|}{n} \Bigr|\\
  &\le \sum_{i=1}^{\mathsf{k}_n}  \frac{\bar{\theta}-\underline{\theta}}{n(n+\underline{\theta})} \cdot (|U_i|-\alpha)\\
  &= \frac{\bar{\theta}-\underline{\theta}}{n(n+\underline{\theta})} (n-\alpha \mathsf{k}_n)\\
  &\le \frac{\bar{\theta}-\underline{\theta}}{n+\underline{\theta}} 
\end{align*}
This completes the proof. 
\end{proof}

\subsection{Proof of \Cref{theorem:ratio_consistency}}
Let $C_n \equiv n \frac{v_{n+1, \mathsf{k}_n}}{v_{n, \mathsf{k}_n}}  - 1$. Note $|C_n|\le \frac{\bar{\theta}+\alpha}{n+\underline\theta}$ by \Cref{lm:concentrate_v_nk}. The ratio $\frac{\sum_{i\in I}\simp_{n,i}}{\sum_{i\in I}\hat \simp_{n,i}}$ can be expressed as 
\begin{align*}
  \frac{\sum_{i\in I}\simp_{n,i}}{\sum_{i\in I}\hat \simp_{n,i}} = n \frac{v_{n+1, \mathsf{k}_n}}{v_{n, \mathsf{k}_n}} \frac{\sum_{i\in I} (|U_i|-\alpha)}{\sum_{i\in I} (|U_i|-\hat{\alpha}_n)} = (1+C_n) \frac{\frac{1}{|I|}\sum_{i\in I}|U_i|-\alpha}{\frac{1}{|I|}\sum_{i\in I}|U_i|-\hat{\alpha}_n}.
\end{align*}
Rearranging the above display, we get 
\begin{align*}
&\Bigl(\frac{1}{|I|}\sum_{i\in I}|U_i|-\hat{\alpha}_n\Bigr)\Bigl(1-   \frac{\sum_{i\in I}\simp_{n,i}}{\sum_{i\in I}\hat \simp_{n,i}} \Bigr) - (\alpha-\hat{\alpha}_n) \\
&= \Bigl(\frac{1}{|I|}\sum_{i\in I}|U_i|-\hat{\alpha}_n\Bigr)-(1+C_n)\Bigl(\frac{1}{|I|}\sum_{i\in I}|U_i|-\alpha\Bigr) - (\alpha-\hat{\alpha}_n) \\
    &= - C_n\Bigl(\frac{1}{|I|}\sum_{i\in I}|U_i|-\alpha\Bigr)
\end{align*}
for all $I\subset \{1, \dots, \mathsf{k}_n\}$. 
Now we define the random set $\mathcal{I}_n$ as 
$$
\mathcal{I}_n = \Bigl\{I \subset  \{1, \dots \mathsf{k}_n\}, \quad \frac{1}{|I|}\sum_{i\in I}|U_i| \le n \cdot \delta_n \Bigr\} \quad \text{where} \quad \mathsf{k}_n^{-1} \le \delta_n = o_p(n^{-\alpha/2}). 
$$
\editline{Note that the random set $\mathcal{I}_n$ is always non-empty since it contains $I=\{1, \dots, \mathsf{k}_n\}$ by the assumption $\delta_n \ge \mathsf{k}_n^{-1}$.} 
Using $|C_n|\le \frac{\bar{\theta}+\alpha}{n+\underline\theta}$ and $ \frac{1}{|I|}\sum_{i\in I}|U_i| \le n \delta_n$ for all $I\in \mathcal{I}_n$, we have
\begin{align*}
 \max_{I\in \mathcal{I}_n} \left|\Bigl(\frac{1}{|I|}\sum_{i\in I}|U_i|-\hat{\alpha}_n\Bigr)\Bigl(1-   \frac{\sum_{i\in I}\simp_{n,i}}{\sum_{i\in I}\hat \simp_{n,i}} \Bigr) - (\alpha-\hat{\alpha}_n) \right| \le |C_n|  \max_{I\in \mathcal{I}_n}  \Bigl(\frac{1}{|I|}\sum_{i\in I}|U_i|-\alpha\Bigr) \le \frac{\bar\theta+\alpha}{\bar\theta+n} \cdot n \cdot \delta_n
\end{align*}
Multiplying the both sides by $\sqrt{\mathsf{k}_n \mathfrak{i}(\hat{\alpha}_n)} =O_p(n^{\alpha/2})$, with  $\delta_n=o_p(n^{-\alpha/2})$, we get
\begin{align}\label{eq:sup_I_conv}
  \sqrt{\mathsf{k}_n \mathfrak{i}(\hat{\alpha}_n)} \cdot \max_{I\in \mathcal{I}_n} \left|\Bigl(\frac{1}{|I|}\sum_{i\in I}|U_i|-\hat{\alpha}_n\Bigr)\Bigl(1-   \frac{\sum_{i\in I}\simp_{n,i}}{\sum_{i\in I}\hat \simp_{n,i}} \Bigr) - (\alpha-\hat{\alpha}_n) \right| = o_p(1).
\end{align}
Therefore, by the triangle inequality, we have 
\begin{align*}
  &\left| \max_{I\in \mathcal{I}_n}\sqrt{\mathsf{k}_n \mathfrak{i}(\hat{\alpha}_n)} \Bigl(\frac{1}{|I|}\sum_{i\in I}|U_i|-\hat{\alpha}_n\Bigr){\Bigl|1-   \frac{\sum_{i\in I}\simp_{n,i}}{\sum_{i\in I}\hat \simp_{n,i}} \Bigr|} -  \sqrt{\mathsf{k}_n \mathfrak{i}(\hat{\alpha}_n)} |\alpha-\hat{\alpha}_n| \right|\\
  &\le \sqrt{\mathsf{k}_n \mathfrak{i}(\hat{\alpha}_n)} \cdot   \max_{I\in \mathcal{I}_n} \left|
\Bigl(\frac{1}{|I|}\sum_{i\in I}|U_i|-\hat{\alpha}_n\Bigr){\Bigl|1-   \frac{\sum_{i\in I}\simp_{n,i}}{\sum_{i\in I}\hat \simp_{n,i}} \Bigr|} - |\alpha-\hat{\alpha}_n| \right|  
\\
&\le  \sqrt{\mathsf{k}_n \mathfrak{i}(\hat{\alpha}_n)} \cdot \max_{I\in \mathcal{I}_n} \left|\Bigl(\frac{1}{|I|}\sum_{i\in I}|U_i|-\hat{\alpha}_n\Bigr)\Bigl(1-   \frac{\sum_{i\in I}\simp_{n,i}}{\sum_{i\in I}\hat \simp_{n,i}} \Bigr) - (\alpha-\hat{\alpha}_n) \right| && \text{the triangle inequality}\\
&= o_p(1) && \text{by \eqref{eq:sup_I_conv}}. 
\end{align*}
\editline{
Combined with $ \sqrt{\mathsf{k}_n \mathfrak{i}(\hat{\alpha}_n)} (\alpha-\hat{\alpha}_n) \to N$ $\stable$ by \Cref{thm:mle_clt} and substituting $\sum_{i\in I}\hat \simp_{n,i} = \frac{1}{n}\sum_{i \in I} (|U_i|-\hat\alpha_n) = \frac{|I|}{n} (\frac{1}{|I|} \sum_{i \in I} |U_i|-\hat\alpha_n)$ for $I\subset \{1, \dots, \mathsf{k}_n\}$, we get
\begin{align*}
    \max_{I\in \mathcal{I}_n}\sqrt{\mathsf{k}_n \mathfrak{i}(\hat{\alpha}_n)} \cdot \frac{n}{|I|} \cdot \Bigl|\sum_{i\in I}\hat \simp_{n,i} -  \sum_{i\in I}\simp_{n,i} \Bigr| \to N \quad \stable.
\end{align*}
This completes the proof of \eqref{eq:forall_I} in \Cref{theorem:ratio_consistency}. 
}
\editline{
\begin{remark}\label{rm:exist_I}
Using \eqref{eq:sup_I_conv}, by the same argument for $\max_{I\in \mathcal{I}_n}$, we also have 
\begin{align*}
  \left| \min_{I\in \mathcal{I}_n}\sqrt{\mathsf{k}_n \mathfrak{i}(\hat{\alpha}_n)} \Bigl(\frac{1}{|I|}\sum_{i\in I}|U_i|-\hat{\alpha}_n\Bigr){\Bigl|1-   \frac{\sum_{i\in I}\simp_{n,i}}{\sum_{i\in I}\hat \simp_{n,i}} \Bigr|} -  \sqrt{\mathsf{k}_n \mathfrak{i}(\hat{\alpha}_n)} |\alpha-\hat{\alpha}_n| \right| = o_p(1).
\end{align*}
Therefore, \eqref{eq:forall_I} also holds with $\forall I\in \mathcal{I}_n$ replaced by $\exists I \in \mathcal{I}_n$. 
\end{remark}
}

\subsection{Proof of \Cref{prop:compare_uniform_local}}
  Recall the definition of $\mathcal{I}_n$ in \eqref{eq:def_random_subsets} and the simplex $\bm{\simp}_n$ in \eqref{eq:def_simplex_estimator}.
  The first claim, $\max_{I\in \mathcal{I}_n} \frac{|I|}{n \sqrt{\mathsf{k}_n}} = \frac{\sqrt{\mathsf{k}_n}}{n}$, immediately follows from $\mathcal{I}_n \subset 2^{\{1, \dots, \mathsf{k}_n\}}$ and $\{1, \dots, \mathsf{k}_n\} \in \mathcal{I}_n$ by the assumption $\delta_n\ge \mathsf{k}_n^{-1}$. 

The second claim follows from the fact that for any $I\in \mathcal{I}_n$, noting $ I \subset \{1, \dots, \mathsf{k}_n\}$, 
\begin{align*}
  \frac{\frac{|I|}{n \sqrt{\mathsf{k}_n}}}{\bm{\hat\simp}_n(I)} = \frac{\frac{|I|}{n \sqrt{\mathsf{k}_n}}}{\frac{1}{n}\sum_{i \in I} (|U_i|-\hat\alpha_n)} = \frac{1}{\sqrt{\mathsf{k}_n}} \frac{1}{\frac{1}{|I|}\sum_{i\in I}|U_i| - \hat{\alpha}_n} \le \frac{1}{\sqrt{\mathsf{k}_n}} \frac{1}{1- \hat{\alpha}_n} 
\end{align*}
where we used $\frac{1}{|I|}\sum_{i\in I}|U_i|\ge \frac{1}{|I|}\sum_{i\in I} 1=1$ in the last inequality. 

As for the third claim, for each $I\subset \{0, 1, \dots, \mathsf{k}_n\}$ such that $I\notin \mathcal{I}_n$, let us consider the two cases $0\in I$ and $0\not\in I$ (i.e., $I\subset\{1, \dots, \mathsf{k}_n\}$) separately. In the former case, using $\bm{\hat\simp}_n(I) \ge \hat{\simp}_{n,0} = \frac{\mathsf{k}_n}{n}\hat{\alpha}_n$, we get 
$$
 \frac{\frac{\sqrt{\mathsf{k}_n}}{n}}{\bm{\hat\simp}_n(I)} \le  \frac{\frac{\sqrt{\mathsf{k}_n}}{n}}{\frac{\mathsf{k}_n}{n}\hat{\alpha}_n} = \frac{1}{\sqrt{\mathsf{k}_n}\hat{\alpha}_n}.
$$
In the latter case, by the definition of $\mathcal{I}_n$, we must have $\frac{1}{|I|}\sum_{i\in I} |U_i| > n \cdot \delta_n$ and hence
$$
  \frac{\frac{\sqrt{\mathsf{k}_n}}{n}}{\bm{\hat\simp}_n(I)} = \frac{\frac{\sqrt{\mathsf{k}_n}}{n}}{\frac{1}{n}\sum_{i \in I} (|U_i|-\hat\alpha_n)} < \frac{\sqrt{\mathsf{k}_n}}{|I| (n\delta_n - \hat{\alpha}_n)} \le \frac{\sqrt{\mathsf{k}_n}}{n\delta_n - 1}
$$
where we used $|I|\ge 1$ and $\hat{\alpha}_n\le 1$ in the last inequality. 
Therefore, taking the maximum of the above two upper bounds, we obtain the claim.

\section{Comparison with other simplex estimators}\label{sec:comparison}
  Recall that our estimator $\bm{\hat\simp}_n$ defined in \eqref{eq:def_simplex_estimator} is derived from the true simplex $\bm{\simp}_n$ \eqref{eq:def_simplex_true} by plugging the Dirac measure $\mu=\delta_0$ and the QMLE $\alpha=\hat{\alpha}_n$. There are other estimators that we can naturally think of. 

One simple alternative is the empirical frequency estimator:
  $$
  \bm{\tilde{\simp}}_n  \equiv \Bigl(
0, \frac{|U_1|}{n}, \dots, \frac{|U_{\mathsf{k}_n}|}{n} 
  \Bigr),
  $$
 which is obtained from the true simplex $\bm{\simp}_n$ by setting $\mu=\delta_0$ and $\alpha=0$. 
Now we derive its convergence rate and show that it is suboptimal compared with our estimator $\bm{\hat\simp}_n$. Using the proof of \Cref{prop:tv_decompose} with $\hat{\alpha}$ replaced by  $0$, we obtain
  \begin{align*}
\tv(\bm{\tilde{\simp}}_n , \bm{\simp}_n) = \frac{\mathsf{k}_n}{n}\alpha + C_n, \quad |C_n|\le \frac{\bar{\theta}-\underline{\theta}}{n+\underline{\theta}}. 
  \end{align*}
  Combined with $n^{-\alpha}\mathsf{k}_n\to \diversity$  almost surely and in $L_2$ (\Cref{thm:Lp_converge}), multiplying both sides by $n^{1-\alpha}$ and noting that the second term vanishes, we obtain
  $$
    n^{1-\alpha} \tv(\bm{\tilde{\simp}}_n , \bm{\simp}_n) \to \alpha \diversity  \quad \text{almost surely and in $L_2$}.
  $$  
  In particular, we have  $\tv(\bm{\tilde{\simp}}_n , \bm{\simp}_n) = O_p(n^{-1+\alpha})$. On the other hand, our estimator $\bm{\hat\simp}_n$ satisfies $\tv(\bm{\hat{\simp}_n} , \bm{\simp}_n) = O_p(n^{-1+\alpha/2})$ by \Cref{theorem:tv}, which decays faster than the rate $n^{-1+\alpha}$. Thus, the frequency estimator is suboptimal in terms of convergence rate under the TV distance. 

  Instead of plugging $\mu=\delta_0$, one may also consider $\mu=\delta_{\theta}$ for some $\theta\ne 0$ such that $\theta>-\alpha$. For clarity, let us denote the resulting estimator by $\bm{\hat\simp}_n^\theta$, namely
  $$
 \bm{\hat \simp}_{n}^\theta \equiv \Bigl(
  \frac{\theta + \mathsf{k}_n \hat{\alpha}_n}{n+\theta}, 
\frac{|U_1|-\hat{\alpha}_n}{n+\theta}, \dots, \frac{|U_{\mathsf{k}_n}|-\hat{\alpha}_n}{n+\theta}
  \Bigr).
  $$  
We claim that this estimator is asymptotically equal to our estimator $\bm{\hat \simp}_{n}$. Indeed, by the definition of $\bm{\hat \simp}_{n}^\theta $ and $\bm{\hat \simp}_{n}$, using $\sum_{i=1}^{\mathsf{k}_n} |U_i| = n$, we have 
  \begin{align*}
  \tv(\bm{\hat \simp}_{n}^\theta, \bm{\hat\simp}_n) &= \frac{1}{2} \Bigl|  \frac{\theta + \mathsf{k}_n \hat{\alpha}_n,}{n+\theta} - \frac{\mathsf{k}_n}{n}\hat{\alpha}_n
    \Bigr| + \frac{1}{2} \sum_{i=1}^{\mathsf{k}_n} \Bigl|
 \frac{|U_i|-\hat{\alpha}_n}{n+\theta}-
 \frac{|U_i|-\hat{\alpha}_n}{n}
    \Bigr|\\
    &= \frac{1}{2} \Bigl|\frac{(n-\mathsf{k}_n\hat\alpha_n)\theta}{n(n+\theta)}\Bigr| + \frac{1}{2} \sum_{i=1}^{\mathsf{k}_n} \Bigl|
\frac{-(|U_i|-\hat{\alpha}_n)\theta}{n(n+\theta)}
    \Bigr| \\
    &= \Bigl(\frac{1}{2} + \frac{1}{2}\Bigr) \frac{(n-\mathsf{k}_n\hat\alpha_n) |\theta|}{n(n+\theta)} && \text{by $\sum_{i=1}^{\mathsf{k}_n} |U_i| = n$}\\
    &\le \frac{|\theta|}{n+\theta} && \text{by $n-\mathsf{k}_n\hat\alpha_n \le n$}
  \end{align*}
  Therefore, by the triangle inequality,
  $$
    \tv(\bm{\hat \simp}_{n}^\theta, \bm{\simp}_n) =   \tv(\bm{\hat \simp}_{n}, \bm{\simp}_n) + C_n, \quad |C_n|\le \frac{|\theta|}{n+\theta}. 
  $$
Multiplying both sides by $n^{1-\alpha/2}$, noting that the error term $C_n$ vanishes, we obtain
\Cref{theorem:tv} for $\bm{\hat\simp}_n^\theta$ as well. 

For these reasons, we focused on the current simple estimator obtained by setting  $\alpha=\hat{\alpha}_n$ and $\mu=\delta_0$.

\end{document}